\documentclass[11pt]{article}
\usepackage{declare-style-article}
\usepackage{declare-maths-article}
\usepackage{declare-text-article}
\usepackage{declare-theorems-article}
\hypersetup{
	pdfauthor 
		= {Nikita Nikolaev},
	pdftitle 
		= {Exact Solutions for the Singularly Perturbed Riccati Equation and Exact WKB Analysis},
	pdfsubject
		= {},
	pdfcreator
		= {},
	pdfproducer
		= {},
	pdfkeywords
		= {exact perturbation theory, singular perturbation theory, Borel summation, Borel-Laplace theory, asymptotic analysis, Gevrey asymptotics, resurgence, exact WKB analysis, exact WKB method, nonlinear ODEs, Riccati equation, Schrödinger equation},
}

\DeclareSymbolFont{bbold}{U}{bbold}{m}{n}
\DeclareSymbolFontAlphabet{\mathbbold}{bbold}

\usepackage{appendix}
\usepackage[font={footnotesize}]{caption}
\usepackage[utf8]{inputenc}
\fancypagestyle{frontpage}{

\cfoot{}
\lfoot{\footnotesize 
First appeared: 14 August 2020\\
Contact: \href{mailto:n.nikolaev@sheffield.ac.uk}{n.nikolaev@sheffield.ac.uk}}
}

\usepackage{titletoc}
\usepackage{subcaption}
\usepackage{changepage}

\titlecontents*{subsubsection}
  [5em]
  {\scriptsize\itshape}
  {}
  {}
  {}
  [\:\|\:]
  []
  
\begin{document}


\title{Exact Solutions for the\\ Singularly Perturbed Riccati Equation \\ and Exact WKB Analysis}

\author{Nikita Nikolaev}

\affil{\small School of Mathematics and Statistics, University of Sheffield, United Kingdom
\\
Section of Mathematics, University of Geneva, Switzerland
}

\date{21 November 2021}

\maketitle
\thispagestyle{frontpage}

\begin{abstract}
The singularly perturbed Riccati equation is the first-order nonlinear ODE $\hbar \del_x f = af^2 + bf + c$ in the complex domain where $\hbar$ is a small complex parameter.
We prove an existence and uniqueness theorem for exact solutions with prescribed asymptotics as $\hbar \to 0$ in a halfplane.
These exact solutions are constructed using the Borel-Laplace method; i.e., they are Borel summations of the formal divergent $\hbar$-power series solutions.
As an application, we prove existence and uniqueness of exact WKB solutions for the complex one-dimensional Schrödinger equation with a rational potential.
\end{abstract}

{\small

\textbf{2020 MSC:} 
	\href{https://zbmath.org/classification/?q=cc%3A34M60}{34M60} (primary); 
	\href{https://zbmath.org/classification/?q=cc%3A34E10}{34E10},
	\href{https://zbmath.org/classification/?q=cc%3A34E20}{34E20},
	\href{https://zbmath.org/classification/?q=cc%3A40G10}{40G10},
}


\enlargethispage{1cm}
{\begin{spacing}{0.9}
\small
\setcounter{tocdepth}{3}
\tableofcontents
\end{spacing}
}

\section{Introduction}

The purpose of this article is to analyse the \dfn{singularly perturbed Riccati equation}
\eqntag{\label{200225120958}
	\hbar \del_x f = af^2 + b f + c
\fullstop{,}
}
where $x$ is a complex variable and $\hbar$ is a small complex perturbation parameter, and where the coefficients $a,b,c$ are holomorphic functions of $(x, \hbar)$ which admit asymptotic expansions as $\hbar \to 0$.
The main problem we pose here is to construct \textit{canonical exact solutions}; i.e., solutions that are holomorphic in both variables and which are uniquely determined by their prescribed asymptotics as $\hbar \to 0$.
This is a quintessential problem in singular perturbation theory.

\paragraph{Motivation.}
Existence and uniqueness theory for first-order ODEs is obviously a very well developed subject which can also be analysed in the presence of a parameter like $\hbar$.
However, it gives no information about the asymptotic behaviour of solutions as ${\hbar \to 0}$.
Attempting to solve an equation like \eqref{200225120958} by expanding it in powers of $\hbar$ generically leads to divergent power series solutions.

Of course, the subject of Riccati equations is vast with an exceptionally long history, appearing in a very wide variety of contexts (see, for example, \cite{MR0357936}).
Our motivation has two primary sources.

One is the exact WKB analysis of Schrödinger equations in the complex domain \cite{MR729194,MR819680,MR1232828,MR1209700,MR1232828,MR2182990,MR3280000}.
This very powerful approximation technique was popularised in the early days of quantum mechanics and goes back to as early as Liouville.
However, the natural question of existence of \textit{exact} solutions with prescribed asymptotic behaviour as $\hbar \to 0$ (often called \textit{exact WKB solutions}) has remained open in general (though in the course of finishing a draft of this paper, we became aware of the recent work of Nemes \cite{MR4226390}).
Our main result can be used to give a positive answer to this question in a large class of problems (generalising in particular the recent results of Nemes).
This is briefly described in a special case in \autoref{211119204333} and a full description is given in \cite{MY210623112236}.

Another interesting problem serving as motivation for this paper is encountered in the analysis of singularly perturbed differential systems and more generally meromorphic connections on holomorphic vector bundles over Riemann surfaces.
Given a singularly perturbed differential system with a singular point, the question is that of constructing a filtration by growth rates on the vector space of local solutions which is holomorphically varying in $\hbar$ and has a well-defined limit as $\hbar \to 0$.
For a large class of systems, the main result in this article can be used to construct such filtrations and furthermore show that they converge to the eigendecomposition of the unperturbed system as $\hbar \to 0$ as well as to the eigendecomposition of the principal part of the system as $x$ tends to the singular point; see \cite{nikolaev2019triangularisation}.

\enlargethispage{0.5cm}
\paragraph{Setting and Overview of Main Results.}
Take a domain $X \subset \Complex_x$, a sector $S \subset \Complex_\hbar$ at the origin, and consider a Riccati equation \eqref{200225120958} whose coefficients $a,b,c$ are holomorphic functions of $(x,\hbar) \in X \times S$ which admit locally uniform asymptotic expansions $\hat{a}, \hat{b}, \hat{c} \in \cal{O} (U) \bbrac{\hbar}$ as $\hbar \to 0$ in $S$.
More details are presented in \autoref{200723174748}, but for the purposes of this introduction, let us focus on the most ubiquitous scenario where $a,b,c$ are in fact polynomials in $\hbar$.
The leading-order part in $\hbar$ of the Riccati equation \eqref{200225120958} is the quadratic equation $a_0 f_0^2 + b_0 f_0 + c_0 = 0$ which generically has two distinct local holomorphic solutions $f^\pm_0$ away from \textit{turning points} (i.e., the zeros of the discriminant $\DD_0 \coleq b_0^2 - 4a_0 c_0$).

Let $U \subset X$ be a domain free of turning points that supports a univalued square-root branch $\sqrt{ \DD_0 }$.
Then it is well-known (see \autoref{200118111737}) that \eqref{200225120958} has precisely two formal solutions $\hat{f}_\pm \in \cal{O} (U) \bbrac{\hbar}$ which are uniquely determined by the leading-order solutions $f^\pm_0$ via a recursion on the coefficients $f^\pm_k$.
The main goal of this paper is to promote --- in a \textit{canonical} way --- the formal solutions $\hat{f}_\pm$ to \textit{exact solutions} $f_\pm$ (formally defined in \autoref{200723174748}); i.e., holomorphic solutions defined on $U_0 \times S_0$ where $U_0 \subset U$ and $S_0 \subset S$ is some sectorial domain such that $f_\pm \sim \hat{f}_\pm$ as $\hbar \to 0$ in $S_0$.

Although existence of exact solutions is a classical fact in the theory of singularly perturbed differential equations (see, e.g., \cite[Theorem 26.1]{MR0460820}) they are inherently non-unique due to the problem of missing exponential corrections in asymptotic expansions.
Part of the issue is that classical techniques in general give no control on the size of the opening of the sectorial domain $S_0$ (e.g., see the remark in \cite[p.144]{MR0460820}, immediately following Theorem 26.1).
In particular, it is impossible in general to `identify' a given exact solution with its asymptotic formal solution.

In this paper, we develop a general procedure applicable to a large class of problems to obtain \textit{canonical} exact solutions which indeed can be identified in a precise sense with their corresponding asymptotic formal solutions.
In order to achieve this, the opening angle $|A|$ of $S$ must be at least $\pi$, the most fundamental case being $|A| = \pi$.
For the purposes of this introduction, let us assume that $A = (-\tfrac{\pi}{2}, +\tfrac{\pi}{2})$.

A basepoint $x_0 \in X$ that is not a turning point, choose a local square-root branch $\sqrt{\DD_0}$, and consider the \textit{Liouville transformation}
\begin{equation}
\label{210821173420}
	z = \Phi (x) \coleq \int\nolimits_{x_0}^x \sqrt{ \DD_0 (t) } \dd{t}	
\fullstop
\end{equation}
Suppose that $x_0$ has a neighbourhood $W \subset X$ which is mapped by $\Phi$ to a horizontal strip $H = \set{z ~\big|~ -r < \Im (z) < r}$ of some width $r > 0$.
Suppose furthermore that the $\hbar$-polynomial coefficients $a_k, b_k, c_k$ of $a,b,c$ are bounded on $W$ by $\sqrt{\DD_0}$.
Then, under these assumptions, the main results of this paper can be summarised as follows.

\begin{thm}{210822125944}
The Riccati equation \eqref{200225120958} has a pair of canonical exact solutions $f_\pm$ near $x_0 \in X$ which are asymptotic to the formal solutions $\hat{f}_\pm$ as $\hbar \to 0$ in the right halfplane.
Namely, there is a neighbourhood $U_0 \subset X$ of $x_0$ and a sectorial subdomain $S_0 \subset S$ with the same opening $A$ such that the Riccati equation \eqref{200225120958} has a unique pair of holomorphic solutions $f_\pm$ on $U_0 \times S_0$ which are Gevrey asymptotic to $\hat{f}_\pm$ as $\hbar \to 0$ along the closed arc $\bar{A}$ uniformly for all $x \in U_0$:
\begin{equation}
\label{210823102331}
	f_\pm \simeq \hat{f}_\pm
\qqqquad
	\text{as $\hbar \to 0$ along $\bar{A}$, unif. $\forall x \in U_0$\fullstop}
\end{equation}
Moreover, $f_\pm$ is the uniform Borel resummation of the formal solution $\hat{f}_\pm$:
\eqntag{
	f_\pm = \cal{S} \big[ \, \hat{f}_\pm \, \big]
\fullstop
}
\end{thm}

This is a special case of \autoref{211026151811} and \autoref{211027140644} which are the two main results of this paper.

\paragraph{Discussion and method.}
We construct the canonical exact solutions $f_\pm$ by employing relatively basic and classical techniques from complex analysis which form the basis for the more modern and sophisticated theory of resurgent asymptotic analysis.
Namely, we use the \textit{Borel-Laplace method}, also known as the theory of \textit{Borel-Laplace summability}.
We stress that the Borel-Laplace method ``is nothing other than the theory of Laplace transforms, written in slightly different variables'', echoing the words of Alan Sokal \cite{MR558468}.
As such, we have tried to keep our presentation very hands-on and self-contained, so the knowledge of basic complex analysis should be sufficient to follow.

An additional significant benefit of our approach is that we obtain uniqueness of the solution in the same sector where the initial data is specified.
This feature does not hold for other less explicit approaches, such as for example \cite[Theorem 26.1]{MR0460820} where an existence theorem is proved only on a smaller subsector and there is no hope of uniqueness.

Finally, we want to take the opportunity to acknowledge the unpublished work of Koike and Schäfke on the Borel summability of WKB solutions of Schrödinger equations with polynomial potentials.
See \cite[\S3.1]{takei2017wkb} for a brief account of their work.
Their ideas (which were kindly explained to the author in a private communication from Kohei Iwaki) provided the initial inspiration for the more general strategy of the proof pursued in this article.

\paragraph*{Acknowledgements.}
I want to express special gratitude to Marco Gualtieri, Marta Mazzocco, and Jörg Teschner for encouraging me to press on with finishing this project and writing this paper.
I also want to very much thank Anton Alekseev for his patience during the long time that this project took.
I am also really grateful to Kohei Iwaki for sharing his personal notes which so clearly explain the arguments of Koike and Schäfke.
In addition, I want to thank Dylan Allegretti, Marco Gualtieri, Kohei Iwaki, Andrew Neitzke, and Shinji Sasaki for helpful discussions.
Finally, I want to thank Marco Gualtieri and Marta Mazzocco for really useful suggestions for improving an earlier draft of this paper.
This work was supported by the NCCR SwissMAP of the SNSF.

\section{Singularly Perturbed Riccati Equations}
\label{200723174748}

\paragraph{Background assumptions.}
\label{210730141800}
Throughout the paper, we fix a complex plane $\Complex_x$ with coordinate $x$ and another complex plane $\Complex_\hbar$ with coordinate $\hbar$.
Let $X$ be a domain in $\Complex_x$ or indeed a coordinate chart on a Riemann surface.
Let $S \subset \Complex_\hbar$ be \hyperref[210217114252]{sectorial domain} at the origin with opening arc $A$.
We assume that $0 < |A| \leq 2\pi$.

Consider the Riccati equation
\begin{equation}\label{211121182754}
	\hbar \del_x f = a f^2 + b f + c
\end{equation}
whose coefficients $a,b,c$ are holomorphic functions of $(x,\hbar) \in X \times S$ admitting locally uniform asymptotic expansions with holomorphic coefficients as $\hbar \to 0$ along $A$:
\eqntag{\label{200306083101}
\begin{aligned}
	a (x, \hbar) &\sim \hat{a} (x, \hbar) \coleq \sum_{n = 0}^\infty a_n (x) \hbar^n
\fullstop{,}
\\
	b (x, \hbar) &\sim \hat{b} (x, \hbar) \coleq \sum_{n = 0}^\infty b_n (x) \hbar^n
\fullstop{,}
\\
	c (x, \hbar) &\sim \hat{c} (x, \hbar) \coleq \sum_{n = 0}^\infty c_n (x) \hbar^n
\fullstop
\end{aligned}
\qqquad
\text{as $\hbar \to 0$ along $A$, loc.unif. $\forall x \in X$\fullstop}
}
In symbols, $a,b,c \in \cal{A}_\textup{loc} (X;A)$ and $\hat{a}, \hat{b}, \hat{c} \in \cal{O} (X) \bbrac{\hbar}$.
Basic notions from asymptotic analysis as well as our notation and conventions are summarised in \autoref{210714114437}.
The main problem we pose in this article is to find canonical \textit{exact solutions} of the Riccati equation \eqref{211121182754} in the following sense.

\begin{defn}{210826135801}
Fix any phase $\theta \in A$.
A \dfn{local $\theta$-exact solution} of the Riccati equation \eqref{211121182754} near a point $x_0 \in X$ is a holomorphic solution $f = f(x,\hbar)$, defined on a domain $U_0 \times S_0$ where $U_0 \subset X$ is a neighbourhood of $x_0$ and $S_0 \subset S$ is a sectorial subdomain with opening $A_0 \subset A$ containing $\theta$, such that $f$ admits an asymptotic expansion with holomorphic coefficients as $\hbar \to 0$ along $A_0$ uniformly for all $x \in U_0$.

A \dfn{$\theta$-exact solution} on a domain $U \subset X$ is a holomorphic solution $f = f(x,\hbar)$ which is a local $\theta$-exact solution near every point in $U$.
That is, $f$ is a holomorphic solution defined on a domain $\UUU \subset U \times S$ with the following property: for every $x_0 \in X$, there is a domain neighbourhood $U_0 \subset U$ of $x_0$ and a sectorial domain $S_0 \subset S$ with opening $A_0 \subset A$ containing $\theta$ such that $f$ admits an asymptotic expansion with holomorphic coefficients as $\hbar \to 0$ along $A_0$ uniformly for all $x \in U_0$.
\end{defn}

\paragraph{Examples.}
The following is a list, included here for illustrative purposes only, containing a few explicit examples of Riccati equations to which the main results in this paper can be applied.

The most typical situation is one where the coefficients $a,b,c$ of the Riccati equation \eqref{211121182754} are polynomials in $\hbar$ with coefficients which are rational functions of $x$.
In this case, $X$ is the complement of the poles in $\Complex_x$, and the sectorial domain $S$ can be taken to be the whole open right halfplane $\set{\Re (\hbar) > 0}$.
The simplest example is:
\begin{enumerate}
\item \DS{\hbar \del_x f = f^2 - x}.
\end{enumerate}
This Riccati equation is examined in great detail in \autoref{211027171318}.
It arises in the exact WKB analysis of the Airy equation $\hbar^2 \del_x^2 \psi (x, \hbar) = x \psi (x, \hbar)$ (see \cite{MY210623112236}).
In this case, $X = \Complex_x$ and the sectorial domain $S$ is the open right halfplane $\set{\Re (\hbar) > 0}$.
If $U$ is any of the three sectorial domains in $\Complex_x$ given by $\set{0 < \arg (x) < +4\pi/3}$, or $\set{+2\pi/3 < \arg (x) < +2\pi}$, or $\set{-2\pi/3 < \arg (x) < +2\pi/3}$, then on each of these domains the main existence and uniqueness result in this paper produces a pair of canonical exact solutions.
More generally:
\begin{enumerate}
\setcounter{enumi}{1}
\item  \DS{\hbar \del_x f = f^2 + q(x)}
\end{enumerate}
where $q(x)$ is any polynomial or a rational function with poles of order $2$ or higher.
In this case, $S$ can again be arranged to be the right halfplane, and $U$ is a sectorial domain near a pole of $q(x)$ of order $2$ or higher.

Many Riccati equations arise from the WKB analysis of classical second-order differential equation.
For example, the following Riccati equation appears in the WKB analysis of the Gauss hypergeometric equation:
\begin{enumerate}
\setcounter{enumi}{2}
\item \DS{\hbar \del_x f = f^2 + \frac{\gamma - (\alpha + \beta + 1) x}{x(x-1)} f + \frac{\alpha \beta}{x (x-1)}} for any $\alpha, \beta, \gamma \in \Complex^\ast$.
\end{enumerate}

Riccati equations also arise in the analysis of singularly perturbed second-order systems.
For example, the Riccati equation
\begin{enumerate}
\setcounter{enumi}{3}
\item \DS{\hbar \del_x f = \hbar f^2 + (1-x) f + x \hbar}
\end{enumerate}
\vspace{-0.5\baselineskip}
arises in the analysis of the system $\hbar \del_x \psi + \mtx{ 1 & -\hbar \\ x\hbar & x} \psi = 0$.
See \cite{nikolaev2019triangularisation}.

Our methods also apply to the following nontrivial deformation of example (1):
\begin{enumerate}
\setcounter{enumi}{4}
\item \DS{\hbar \del_x f = f^2 - x + x \EE (x, \hbar)
\qtext{where}
	\EE (x, \hbar)
		\coleq \int_0^{+\infty} \frac{e^{- \xi/\hbar}}{x + \xi} \dd{\xi}}.
\end{enumerate}
The sectorial domain $S$ in this case is the open right halfplane.
The function $\EE$ is holomorphic in $\hbar \in S$ and it admits a locally-uniform asymptotic expansion as $\hbar \to 0$ in the right halfplane (see \autoref{200708191952} for more details).
Notice, however, that $\EE$ is not holomorphic at $\hbar = 0$, and it also has non-isolated singularities along the negative real axis in $\Complex_x$.
Nevertheless, if $U$ is the domain given by $\set{-2\pi/3 < \arg (x) < +2\pi/3}$ or by $\set{0 < \arg (x) < +\pi}$ or $\set{+\pi < \arg (x) < +2\pi}$, then our method yields canonical exact solutions on $U$.

\section{Formal Perturbation Theory}
\label{210714152619}

In this section, we analyse the Riccati equation from a purely formal perspective whereby we ignore all analytic considerations in the $\hbar$-variable.

\paragraph{}
Thus, we consider the \dfn{formal Riccati equation}
\begin{equation}\label{211003140338}
	\hbar \del_x \hat{f} = \hat{a} \hat{f}^2 + \hat{b} \hat{f} + \hat{c}
\fullstop{,}
\end{equation}
where $\hat{a}, \hat{b}, \hat{c}$ are arbitrary formal power series in $\hbar$ with holomorphic coefficients on some domain $X$ in $\Complex_x$.
In symbols, $\hat{a}, \hat{b}, \hat{c} \in \cal{O} (X) \bbrac{\hbar}$.
By definition, a \dfn{formal solution} of \eqref{211003140338} on a domain $U \subset X$ is any formal power series with holomorphic coefficients 
\eqntag{\label{200219162841}
	\hat{f} = \hat{f} (x, \hbar) = \sum_{k=0}^\infty f_k (x) \hbar^k
	~\in~ \cal{O} (U) \bbrac{\hbar}
}
that satisfies the formal equation \eqref{211003140338}.
Here, the derivative $\del_x \hat{f}$ is interpreted as term-by-term differentiation.

\paragraph{}
A formal Riccati equation \eqref{211003140338} arises from an analytic Riccati equation \eqref{211121182754} by replacing the coefficients $a,b,c$ with their asymptotic power series $\hat{a}, \hat{b}, \hat{c}$ from \eqref{200306083101}.
Solutions of \eqref{211003140338} are regarded as formal solutions of \eqref{211121182754} as opposed to `true' solutions that are meant to have some analytic meaning in the variable $\hbar$.
Note also that if the coefficients $a,b,c$ are polynomials in $\hbar$, then $\hat{a} = a, \hat{b} = b, \hat{c} = c$ so the formal Riccati equation \eqref{211003140338} is exactly the same as the analytic Riccati equation \eqref{211121182754}.

\subsection{Leading-Order Solutions}
\label{200601224547}

\paragraph{}
Consider the \dfn{leading-order equation} corresponding to \eqref{211003140338}:
\eqntag{\label{200521115322}
	a_0 f_0^2 + b_0 f_0 + c_0 = 0
\fullstop
}
It is a quadratic equation in the unknown variable $f_0$, and we refer to its solutions as \dfn{leading-order solutions} of the Riccati equation.
Generically, they are locally holomorphic but may have poles and branch-point singularities.

\paragraph{}
The discriminant of \eqref{200521115322},
\eqntag{\label{200219161938}
	\DD_0 \coleq b_0^2 - 4 a_0 c_0
\fullstop{,}
}
which we call the \dfn{leading-order discriminant} of the Riccati equation, is a holomorphic function on $X$. 
We always assume that $\DD_0$ is not identically zero.
The zeros of $\DD_0$ are called \dfn{turning points} of the Riccati equation, and all other points in $X$ are called \dfn{regular points}.

Locally, away from turning points, the Riccati equation always has at least one holomorphic leading-order solution.
For reference, we state the following elementary lemma (whose proof is presented in \autoref{200614173443} for completeness).

\begin{lem}[\textbf{holomorphic leading-order solutions}]{200614161900}
Let $U \subset X$ be any domain free of turning points such that a univalued square-root branch $\sqrt{ \DD_0 }$ of $\DD_0$ can be chosen on $U$.
Then the leading-order equation \eqref{200521115322} has at least one holomorphic solution on $U$.
In addition, if $a_0$ is nonvanishing on $U$, then \eqref{200521115322} has two holomorphic solutions.
Moreover, any holomorphic solution is bounded on $U$ whenever the coefficients $a_0, b_0, c_0$ are bounded by $\sqrt{\DD_0}$ on $U$.
\end{lem}

\paragraph{}
\label{211102183205}
We will always label the leading-order solutions as follows:
\begin{subequations}\label{210715120210}
\eqnstag{\label{200702084051}
	f_0^\pm \coleq \frac{-b_0 \pm \sqrt{\DD_0}}{2 a_0} 
	&\qquad\text{if $a_0 \not\equiv 0$;}
\\\label{200702084052}
	f_0^+ \coleq - c_0 / b_0
	&\qquad\text{if $a_0 \equiv 0$.}
}
\end{subequations}
This choice of labels yields the following relations:
\eqntag{\label{210715115850}
	\pm \sqrt{\DD_0} = 2a_0 f_0^\pm + b_0
\qqtext{and}
	\sqrt{\DD_0} = a_0 \big( f_0^+ - f_0^- \big)
\fullstop
}
Thus, if $a_0$ is nonvanishing, then both $f_0^\pm$ from \eqref{200702084051} are holomorphic functions on $U$.
If $a_0$ has zeros in $U$, then $f_0^+$ from \eqref{200702084051} remains holomorphic on $U$, but $f_0^-$ has poles where $a_0$ has zeros.
If $a_0 \equiv 0$, then $f_0^+$ from \eqref{200702084052} is a holomorphic function on $U$.

\begin{rem}{210714162921}
If $a_0 \equiv 0$, then $b_0$ is the preferred square root of $\DD_0$.
This choice makes the labels in \eqref{200702084051} consistent with \eqref{200702084052}, because in this case the solution $f_0^+$ from \eqref{200702084051} is necessarily holomorphic on $U$ (whether or not $a_0$ has zeros in $U$) and converges to $-c_0/b_0$ as $a_0 \to 0$ uniformly in $x$ (whilst the solution $f_0^-$ has no limit).
So, according to our notation, $f_0^+$ is always holomorphic away from turning points.
\end{rem}

\subsection{Existence and Uniqueness of Formal Solutions}
\label{200521121035}

The following elementary theorem says that a formal Riccati equation \eqref{211003140338} always has at least one local solution away from turning points, and it is uniquely specified in the leading-order.

\begin{thm}[\textbf{Formal Existence and Uniqueness Theorem}]{200118111737}
Consider the formal Riccati equation \eqref{211003140338}.
Assume its leading-order discriminant $\DD_0$ is not identically zero.
Let $U \subset X$ be a domain free of turning points that supports a univalued square-root branch $\sqrt{\DD_0}$.
\begin{enumerate}
\item If $a_0 \equiv 0$, then \eqref{211003140338} has a unique formal solution $\hat{f}_+ \in \cal{O} (U) \bbrac{\hbar}$.
Its leading-order term is $f_0^+$ from \eqref{200702084052}.
\item If $a_0 \not\equiv 0$ and nonvanishing on $U$, then \eqref{211003140338} has exactly two distinct formal solutions $\hat{f}_\pm \in \cal{O} (U) \bbrac{\hbar}$.
Their leading-order terms $f^\pm_0$ are given by \eqref{200702084051}.
\item If $a_0 \not\equiv 0$ but has zeros in $U$, then \eqref{211003140338} has a unique formal solution $\hat{f}_+ \in \cal{O} (U) \bbrac{\hbar}$.
Its leading-order term $f_0^+$ is the unique holomorphic leading-order solution on $U$ given by \eqref{200702084051}.
\end{enumerate}
Moreover, the coefficients $f_k^\pm$ of $\hat{f}_\pm$ for $k \geq 1$ are given by the following recursive formula:
\eqntag{\label{191207215212}
	f_k^\pm = \pm \frac{1}{\sqrt{ \DD_0 }} \del_x f_{k-1}^\pm \mp \frac{1}{\sqrt{ \DD_0 }}
		\left(
					\sum_{k_1 + k_2 + k_3 = k}^{k_2, k_3 \neq k} a_{k_1} f_{k_2}^\pm f_{k_3}^\pm
					+ \sum_{k_1 + k_2 = k}^{k_2 \neq k} b_{k_1} f_{k_2}^\pm
					+ c_k
	\right)
\fullstop
}
\end{thm}

\begin{proof}
We expand the formal Riccati equation \eqref{211003140338} order-by-order in $\hbar$:
\eqnstag{\label{191207213609}
	\hbar^0 ~\big|~ ~ &
		0 = a_0 f_0^2 + b_0 f_0 + c_0
\fullstop{;}
\\\label{191207212036}
	\hbar^1 ~\big|~ ~ &
		\del_x f_0 = (2a_0 f_0 + b_0) f_1 + a_1 f_0^2 + b_1 f_0 + c_1
\fullstop{;}
\\\label{191207212042}
	\hbar^2 ~\big|~ ~ &
		\del_x f_1 = (2a_0 f_0 + b_0) f_2 + a_0 f_1^2 + 2 a_1 f_0 f_1 + a_2 f_0^2 + b_1 f_1 + b_2 f_0 + c_2
\fullstop{;}
\\\nonumber
	\vdots \phantom{~~\big|~ ~}
\\[-0.75\baselineskip]\label{191208144433}
	\hbar^k ~\big|~ ~ & 
		\del_x f_{k-1} = (2a_0 f_0 + b_0) f_k
					+ \sum_{k_1 + k_2 + k_3 = k}^{k_2, k_3 \neq k} a_{k_1} f_{k_2} f_{k_3} 
					+ \sum_{k_1 + k_2 = k}^{k_2 \neq k} b_{k_1} f_{k_2} 
					+ c_k
\fullstop
\\[-0.75\baselineskip]\nonumber
	\vdots \phantom{~~\big|~ ~}
}
Observe that these are no longer differential equations because the derivative term at each order depends only on the solutions from previous orders.
If we fix a leading-order solution $f_0^\pm$, then the expression $(2 a_0 f_0^\pm + b_0)$, appearing as a factor in front of $f^\pm_k$ in each equation \eqref{191208144433}, is simply $\pm \sqrt{\DD_0}$.
From the assumption that $\DD_0 \not\equiv 0$ it follows that at each order in $\hbar$ we can uniquely solve for $f_k^\pm$.
This establishes the formula \eqref{191207215212}, from which the other statements readily follow.
\end{proof}

\begin{rem}{210714163859}
In part (3) of \autoref{200118111737}, the Riccati equation \eqref{211003140338} also has a \textit{singular} formal solution $\hat{f}_-$ on $U$ whose leading-order term is the singular leading-order solution $f_0^-$ on $U$.
The singularities of the coefficients of $\hat{f}_-$ are poles occurring at the zeros of $a_0$.
We will examine in more detail the singularities of formal (and exact) solutions in a forthcoming paper.
\end{rem}

\begin{rem}[formal discriminant]{211028174846}
Since in the generic situation the Riccati equation has precisely two formal solutions $\hat{f}_+, \hat{f}_-$, we can introduce a notion of discriminant for the Riccati equation analogous to the discriminant of a quadratic equation by simply mimicking the formula.

Thus, let $U \subset X$ be a domain free of turning points that supports a univalued square-root branch $\sqrt{\DD_0}$, and suppose that $a_0$ is nonvanishing on $U$.
We define the \dfn{formal discriminant} of the Riccati equation \eqref{211003140338} by the following formula:
\begin{equation}\label{211028183358}
	\hat{\DD} (x, \hbar) 
		\coleq \hat{a}^2 \big( \hat{f}_+ - \hat{f}_- \big) \big( \hat{f}_- - \hat{f}_+ \big)
\fullstop
\end{equation}
It is a formal power series with holomorphic coefficients on $U$ (i.e., $\hat{\DD} \in \cal{O} (U) \bbrac{\hbar}$), and its leading-order term is precisely the leading-order discriminant $\DD_0$.
This quantity plays an important role in addressing global questions in the WKB analysis that will be studied elsewhere.
\end{rem}

\subsection{Gevrey Regularity of Formal Solutions}
\label{200811170816}

In this subsection, we prove the following general result about the regularity of formal solutions, which generalises Proposition A.1.1 in \cite[p.19]{MR1166808} (see also \cite[p.252]{MR729194}), where it is assumed that $\hat{a} = -1, \hat{b} = 0$, and $\hat{c}$ is an entire holomorphic function of $x$ only (i.e., $\hat{c} (x, \hbar) = c_0 (x)$).

\begin{prop}[\textbf{Local Gevrey Regularity of Formal Solutions}]{200118123730}
Consider a formal Riccati equation \eqref{211003140338} on $X$ with leading-order discriminant $\DD_0 \not\equiv 0$.
Let $U \subset X$ be any domain free of turning points that supports a univalued square-root branch $\sqrt{\DD_0}$, and let $\hat{f}$ be a formal solution on $U$.

If the coefficients $\hat{a},\hat{b},\hat{c}$ are locally uniformly Gevrey series on $U$, then so is $\hat{f}$.
In particular, the formal Borel transform $\hat{\phi} (x, \xi) \coleq \hat{\Borel} [ \, \hat{f} \, ] (x, \xi)$ of $\hat{f}$ is a locally uniformly convergent power series in $\xi$.
\end{prop}

Let us make a few remarks before presenting the proof.
In symbols, this proposition says that if $\hat{a},\hat{b},\hat{c} \in \cal{G}_\textup{loc} (U) \bbrac{\hbar}$, then $\hat{f} \in \cal{G}_\textup{loc} (U) \bbrac{\hbar}$ and $\hat{\phi} \in \cal{O}_\textup{loc} (U) \set{\xi}$.
The latter means that, for any compactly contained subset $U_0 \Subset U$, there is a disc $\Disc_0 \subset \Complex_\xi$ around the origin such that $\hat{\phi} (x, \xi)$ is a holomorphic function on $U_0 \times \Disc_0$.

Concretely, \autoref{200118123730} says that if the coefficients $a_k, b_k, c_k$ of the power series $\hat{a},\hat{b},\hat{c}$ grow at most like $k!$, then the coefficients $f_k$ of any formal solution $\hat{f}$ likewise grow at most like $k!$.
This is made precise in the following corollary.

\begin{cor}[\textbf{at most factorial growth}]{211028172730}
Consider a formal Riccati equation \eqref{211003140338} on $X$ with leading-order discriminant $\DD_0 \not\equiv 0$.
Let $U \subset X$ be any domain free of turning points that supports a univalued square-root branch $\sqrt{\DD_0}$, and let $\hat{f}$ be a formal solution on $U$.

Take any pair of nested compactly contained subsets $U_0 \Subset U_1 \Subset U$, and suppose that there are real constants $\AA, \BB > 0$ such that
\eqntag{\label{211028173208}
	\big| a_k (x) \big|,
	\big| b_k (x) \big|,
	\big| c_k (x) \big|
	\leq \AA \BB^{k} k!
\qqquad
	(\forall k \geq 0, \forall x \in U_1)
\fullstop
}
Then there are real constants $\CC, \MM > 0$ such that
\eqntag{\label{211028173423}
	\big| f_k (x) \big|
	\leq \CC \MM^{k} k!
\qqqqqqqqquad
	(\forall k \geq 0, \forall x \in U_0)
\fullstop
}
\end{cor}

\begin{rem}{210726151552}
In general it is not the case that a formal solution $\hat{f}$ is a convergent power series even if the coefficients $\hat{a},\hat{b},\hat{c}$ are convergent.
In other words, if the class $\cal{G}_\textup{loc} (U) \bbrac{\hbar}$ in \autoref{200118123730} is replaced with the class $\cal{O}_\textup{loc} (U) \set{\hbar}$ or $\cal{O} (U) \set{\hbar}$ in either or both the hypothesis and the conclusion, then the assertion becomes false.
This is typical of singular perturbation theory.
The assertion also becomes false if $\cal{G}_\textup{loc} (U) \bbrac{\hbar}$ is replaced with $\cal{G} (U) \bbrac{\hbar}$ in both the hypothesis and the conclusion.
This is because the estimates on the coefficients $f_k$ necessarily involve Cauchy estimates on the derivatives of lower orders.
\end{rem}

\begin{proof}[Proof of \autoref{200118123730}.]
Let $\Disc_\RR \subset U$ be any sufficiently small disc of some radius $\RR > 0$ on which $\hat{a}, \hat{b}, \hat{c}$ are uniformly Gevrey and $\sqrt{\DD_0}$ is bounded both above and below by a nonzero constant.
Thus, there are real constants $\AA, \BB > 0$ which give the following uniform bounds:
\eqntag{\label{191207222837}
	\big| a_k (x) \big|,
	\big| b_k (x) \big|,
	\big| c_k (x) \big|
	\leq \AA \BB^{k} k!
\qqtext{and}
	\AA^{-1} \leq \big| \sqrt{\DD_0 (x)} \big| \leq \AA
}
for all integers $k \geq 0$ and all $x \in \Disc_\RR$.
It will be convenient for us to assume without loss of generality that $\AA \geq 3$ and $\RR < 1$.
We will prove that the solution $\hat{f}$ is a uniformly Gevrey power series on any compactly contained subset of $\Disc_\RR$.
In fact, we will prove something a little bit stronger as follows.
For any $r \in (0, \RR)$, denote by $\Disc_r \subset \Disc_\RR$ the concentric subdisc of radius $r$.
Then \autoref{200118123730} follows from the following claim.

\begin{claim*}{200120105124}
There exist real constants $\CC, \MM > 0$ such that, for any $r \in (0,\RR)$,
\eqntag{\label{191211212938}
	\big| f_k (x) \big| \leq \CC \MM^k \delta^{-k} k!
}
for all integers $k \geq 0$ and uniformly for all $x \in \Disc_r$, where $\delta \coleq \RR - r$.
(The constants $\CC, \MM$ are independent of $r, x, k$, but may depend on $\RR, \AA, \BB$.)
In particular, for any $r \in (0,\RR)$, the power series $\hat{f}$ is Gevrey uniformly for all $x \in \Disc_r$.
\end{claim*}

\textsc{Proof.}
First, it is easy to find a constant $\CC > 0$ (independent of $r$) such that
\eqntag{\label{191209170218}
	\big| f_0 (x) \big| \leq \CC
}
uniformly for all $x \in \Disc_\RR$ (see \autoref{200614161900}).
Without loss of generality, assume that 
\eqntag{\label{211112130258}
	\CC \geq \AA \geq 3
\fullstop
}
Then the bound \eqref{191211212938} will be demonstrated in two main steps.
First, we will recursively construct a sequence $(\MM_k)_{k=0}^\infty$ of positive real numbers such that, for all $k \geq 0$ and all $r \in (0,\RR)$, we have the following uniform bound for all $x \in \Disc_r$:
\eqntag{\label{190317144208}
		\big| f_k (x) \big| \leq \CC \MM_k \delta^{-k} k!
\fullstop
}
Then we will show that there is a constant $\MM > 0$ (independent of $r$) such that $\MM_k \leq \MM^k$ for all $k$.

\textsc{Construction of $(\MM_k)_{k = 0}^\infty$.}
The bound \eqref{190317144208} for $k = 0$ is just the bound \eqref{191209170218} if we put $\MM_0 \coleq 1$.
Now we use induction on $k$ and formula \eqref{191207215212}.
Assume that we have already constructed positive real numbers $\MM_0, \ldots, \MM_{k-1}$ such that, for all $i = 0, \ldots, k-1$, all $r \in (0,\RR)$, and all $x \in \Disc_r$, we have the bound
\eqntag{\label{191212164433}
	\big| f_i (x) \big| \leq \CC \MM_i \delta^{-i} i!
}
In order to derive an estimate for $f_k$, we first need to estimate the derivative term $\del_x f_{k-1}$, for which we use Cauchy estimates as follows.

\textsc{Sub-Claim.}
\textit{
For all $r \in (0,\RR)$ and all $x \in \Disc_r$,
\eqntag{\label{191212164536}
	\big| \del_x f_{k-1} (x) \big| \leq \CC^2 \MM_{k-1} \delta^{-k} k!
\fullstop
}}%
\begin{subproof}[Proof of Sub-Claim.]
For every $r \in (0,\RR)$, define
\eqn{
	\delta_k \coleq
		\delta \frac{k}{k+1}
\qqtext{and}
	r_k \coleq \RR - \delta_k
\fullstop
}
Inequality \eqref{191212164433} holds in particular for $i = k-1$, $r = r_k$.
So for all $x \in \Disc_{r_k}$, we find:
\eqns{
	\big| f_{k-1} (x) \big|
		&\leq \CC \MM_{k-1}  \delta_k^{1-k} (k-1)!
		= \CC \MM_{k-1} \delta^{1-k} \frac{k}{k+1}
			\left( \tfrac{k+1}{k} \right)^{k} (k-1)!
\\		&\leq \CC^2 \MM_{k-1} \delta^{-k} k! \tfrac{\delta}{k+1}
\fullstop
}
Here, we have used the estimate $( 1 + 1/k )^{k-1} \leq e \leq \CC$.
Finally, notice that for every $x \in \Disc_r$, the closed disc around $x$ of radius $r_k - r = \delta - \delta_k = \frac{\delta}{k+1}$ is contained inside the disc $\Disc_{r_k}$.
Therefore, Cauchy estimates imply \eqref{191212164536}.
\end{subproof}

Using \eqref{191207222837}, \eqref{211112130258}, \eqref{191212164433}, \eqref{191212164536}, and the fact that $\delta < 1$, we can now estimate $f_k$:
\eqns{
	\big| f_k \big|
		&\leq \CC
			\left(
		 	\big| \del_x f_{k-1} \big|
			+ \sum_{k_1 + k_2 + k_3 = k}^{k_2, k_3 \neq k} \big| a_{k_1} \big| \cdot \big| f_{k_2} \big| \cdot \big| f_{k_3} \big|
			+ \sum_{k_1 + k_2 = k}^{k_2 \neq k} \big| b_{k_1} \big| \cdot \big| f_{k_2} \big| 
			+ \big| c_k \big|
			\right)
\\		&\leq \CC
			\left(
				\CC^2 \MM_{k-1} \delta^{-k} k!
			+ \delta^{-k} \CC^3 k! \sum_{k_1 + k_2 + k_3 = k}^{k_2, k_3 \neq k} \BB^{k_1} \MM_{k_2} \MM_{k_3}
			\right.
\\		&\hspace{3.65cm}
			\left.
			+ \delta^{-k} \CC^2 k! \sum_{k_1 + k_2 = k}^{k_2 \neq k} \BB^{k_1} \MM_{k_2}
			+ \CC \BB^k k!
			\right)
\\		&\leq \CC^4 \left( 
				\MM_{k-1}
					+ \!\!\!
						\sum_{k_1 + k_2 + k_3 = k}^{k_2, k_3 \neq k} \BB^{k_1}
						\MM_{k_2} \MM_{k_3}
					+ \sum_{k_1 + k_2 = k}^{k_2 \neq k} \BB^{k_1} \MM_{k_2} 
					+ \BB^k 
			\right) \delta^{-k} k!
}
Here, we used the inequality $k_1! k_2! \leq (k_1 + k_2)!$.
We can therefore define, for $k \geq 1$,
\eqntag{\label{191212173415}
	\MM_k \coleq 
			\CC^3 \left( 
				\MM_{k-1}
					+ \!\!\!\! 
						\sum_{k_1 + k_2 + k_3 = k}^{k_2, k_3 \neq k} \BB^{k_1}
						\MM_{k_2} \MM_{k_3}
					+ \sum_{k_1 + k_2 = k}^{k_2 \neq k} \BB^{k_1} \MM_{k_2} 
					+ \BB^k 
			\right)
\fullstop
}

\textsc{Construction of $\MM$.}
To see that $\MM_k \leq \MM^k$ for some $\MM > 0$, we argue as follows.
Consider the following power series in an abstract variable $t$:
\eqn{
	p (t) \coleq \sum_{k=0}^\infty \MM_k t^k
\qqtext{and}
	q (t) \coleq \sum_{k=1}^\infty \BB^k t^k
\quad \in \quad \Complex \bbrac{t}
\fullstop
}
Note that $p(0) = \MM_0 = 1$ and $q (0) = 0$, and notice that $q (t)$ is convergent.
We will show that $p(t)$ is also convergent.
The key is the observation that they satisfy the following algebraic equation:
\eqntag{\label{191209122442}
	p(t) - 1 = \CC^3 \Big( t p(t) + q(t) p(t)^2 + q(t) p(t) + q(t) \Big)
\fullstop
}
This equation was found by trial and error, and it is straightforward to verify directly by substituting the power series $p(t), q(t)$ and comparing the coefficients of $t^k$ using the defining formula \eqref{191212173415} for $\MM_k$.
Now, consider the following holomorphic function in two complex variables $(t,p)$: 
\eqn{
	\FF (t,p) \coleq -p + 1 + \CC^3 \Big( t p + q(t) p^2 + q(t) p + q(t) \Big)
\fullstop
}
It has the following properties:
\eqn{
	\FF (0, 1) = 0
\qqtext{and}
	\evat{\frac{\del \FF}{\del p}}{(t,p) = (0,1)} = -1 \neq 0
\fullstop
}
Thus, by the Holomorphic Implicit Function Theorem, there exists a unique holomorphic function $\PP (t)$ near $t = 0$ such that $\PP (0) = 1$ and $\FF \big(\PP (t), t) \big) = 0$.
Thus, $p(t)$ must be its Taylor series expansion at $t = 0$, so $p (t) \in \Complex \set{t}$ and its coefficients grow at most exponentially: there is a constant $\MM > 0$ such that $\MM_k \leq \MM^k$.
This completes the proof of the Claim and hence of \autoref{200118123730}.	
\end{proof}

\section{WKB Geometry}
\label{211028202110}

In this intermediate section, we introduce a coordinate transformation which plays a central role in the construction of exact solutions in \autoref{210823153207}.
It is used to determine regions in $\Complex_x$ where the Borel-Laplace method can be applied to the Riccati equation.

The material of this section can essentially be found in \cite[\S9-11]{MR743423} (see also \cite[\S3.4]{MR3349833}).
These references use the language of foliations given by quadratic differentials on Riemann surfaces.
The relevant quadratic differential is $\DD_0 (x) \dd{x}^2$.
The reader may be more familiar with the set of critical leaves of this foliation which is encountered in the literature under various names including \textit{Stokes curves}, \textit{Stokes graph}, \textit{spectral network}, \textit{geodesics}, and \textit{critical trajectories} \cite{MR3003931, MR2182990, MR1209700, MR3115984, 1902.03384}.

To keep the discussion a little more elementary, we state the relevant definitions and facts by appealing directly to explicit formulas using the \textit{Liouville transformation} (defined below) commonly used in the WKB analysis of Schrödinger equations.

\subsection{The Liouville Transformation}

\paragraph{}
Throughout this section, we remain in the background setting of \autoref{210730141800}.
Recall the leading-order discriminant $\DD_0 = b_0^2 - 4 a_0 c_0$ which is a holomorphic function on $X$, assumed not identically zero.
Fix a phase $\theta \in \Real / 2\pi \Integer$, a basepoint $x_0 \in X$ and a univalued square-root branch $\sqrt{\DD_0}$ near $x_0$ (i.e., either in a disc or a sectorial neighbourhood of $x_0$).
Consider the following local coordinate transformation near $x_0$, called the \dfn{Liouville transformation}:
\eqntag{\label{211026152238}
	z = \Phi (x)
		\coleq \int_{x_0}^x \sqrt{ \DD_0 (t) } \dd{t}
\fullstop
}
Let $U \subset X$ be any domain which is free of turning points, supports a univalued square-root branch $\sqrt{\DD_0}$ (e.g., $U$ is simply connected), and contains $x_0$ in the interior or on the boundary.
Then the Liouville transformation defines a (possibly multivalued) local biholomorphism $\Phi: U \too \Complex_z$.
Notice that turning points are precisely the locations in $X$ where $\Phi$ fails to be conformal.

\begin{rem}{211026191634}
The basepoint of integration $x_0$ can in principle be chosen even on the boundary of $X$ or at infinity in $\Complex_x$ provided that the integral is well-defined.
Liouville transformations such as \eqref{211026152238} are encountered in the analysis of the Schrödinger equation $\hbar^2 \del_x^2 \psi - \QQ (x) \psi = 0$ as described for example in Olver's textbook \cite[\S6.1]{MR1429619}.
However, note that our formula \eqref{211026152238} in the special case of the Schrödinger equation reads
\eqntag{
	\Phi (x) = \int\nolimits_{x_0}^x \sqrt{\DD_0 (t)} \dd{t} = 2 \int\nolimits_{x_0}^x \sqrt{\QQ (t)} \dd{t}
\fullstop{,}
}
which differs from formula (1.05) in \cite[\S6.1]{MR1429619} by a factor of $2$.	
\end{rem}

\begin{rem}{211026191712}
The main utility of the Liouville transformation \eqref{211026152238} is that it transforms the differential operator $\frac{1}{\sqrt{\DD_0}} \del_x$ (which has already appeared prominently in formula \eqref{191207215212} for the formal solutions) into the constant-coefficient differential operator $\del_z$.
In the language of differential geometry, $\Phi_\ast : \frac{1}{\sqrt{\DD_0}} \del_x \mapsto \del_z$.
This property is obviously independent of the chosen basepoint $x_0$.
This straightening-out of the local geometry using the Liouville transformation is heavily exploited in our construction of exact solutions.
\end{rem}

\subsection{WKB Trajectories}

\paragraph{}
Let $x_0 \in X$ be a regular point and consider the Liouville transformation \eqref{211026152238}.
A \dfn{WKB $\theta$-trajectory} through $x_0$ is the real $1$-dimensional smooth curve $\Gamma_\theta$ on $X$ locally determined by the following equation:
\eqntag{\label{211109193109}
	\Gamma_\theta = \Gamma_\theta (x_0) \quad:\quad
	\Im \big( e^{-i\theta} \Phi (x) \big) = 0
\fullstop
}
A \dfn{WKB $\theta$-trajectory $\pm$-ray} (or simply a \dfn{WKB ray} if the context is clear) emanating from $x_0$ is the component $\Gamma^\pm_\theta$ of $\Gamma_\theta$ given respectively by
\eqntag{\label{211109193325}
	\Gamma^\pm_\theta = \Gamma_\theta^\pm (x_0) \quad:\quad
	\pm \Re \big( e^{-i\theta} \Phi (x) \big) \geq 0
\fullstop{,}
}
WKB trajectories are regarded by definition as being maximal under inclusion.
Explicitly, \eqref{211109193109} and \eqref{211109193325} read
\eqntag{\label{211109193534}
	\Im \left( \: e^{-i\theta} \int\nolimits_{x_0}^x \sqrt{\DD_0 (t)} \dd{t} \right)
	= 0
\qtext{and}
	\pm \Re \left( \: e^{-i\theta} \int\nolimits_{x_0}^x \sqrt{\DD_0 (t)} \dd{t} \right)
	\geq 0
\fullstop
}

\paragraph{}
The Liouville transformation $\Phi$ with basepoint $x_0$ maps the WKB $\theta$-trajectory $\Gamma_\theta (x_0)$ to a possibly infinite straight line segment $(\tau_{-} e^{i\theta}, \tau_{+} e^{i\theta}) \subset e^{i\theta} \Real \subset \Complex_z$ containing the origin $0 = \Phi (x_0)$; i.e., with $\tau_- < 0 < \tau_+$.
Maximality means that this line segment is the largest possible image.
The image of the WKB $\theta$-trajectory $\pm$-ray emanating from $x_0$ is then respectively the line segment $[0, \tau_{+}e^{i\theta})$ or $(\tau_-e^{i\theta}, 0]$.

All other nearby WKB $\theta$-trajectories can be locally described by an equation of the form $\Im \big( e^{-i\theta} \Phi (x) \big) = c$ for some $c \in \Real$.
That is, if $U_0 \subset X$ is a simply connected neighbourhood of $x_0$ free of turning points, then any WKB $\theta$-trajectory $\Gamma'_\theta$ intersecting $U_0$ is locally given by this equation with $c = \Im e^{-i\theta} \Phi (x'_0)$ for some $x'_0 \in U_0$.
Its image in $\Complex_z$ under $\Phi$ is an interval on the parallel line containing $z'_0 \coleq \Phi (x'_0)$:
\eqn{
	z'_0 + e^{i\theta} \Real \coleq \set{z = z'_0 + \xi ~\big|~ \xi \in e^{i\theta}\Real}
\fullstop
}

\paragraph{}
Our primary focus is \dfn{infinite WKB rays} $\Gamma_\theta^\pm$, defined as having $|\tau_\pm| = \infty$, respectively.
An \dfn{infinite WKB trajectory} is one with at least one infinite ray.
A \dfn{generic WKB trajectory} is one with both rays being infinite.

An infinite WKB trajectory may be a \dfn{closed WKB trajectory} if it is a simple closed curve in the complement of the turning points.
A closed WKB $\theta$-trajectory has the property that there is a nonzero time $\omega \in \Real$ such that $\Phi^{-1} (e^{i\theta}\omega) = \Phi^{-1} (0)$, \cite[\S9.2]{MR743423}.
This only happens when the Liouville transformation is analytically continued along the trajectory to a multivalued function.
We refer to the smallest possible positive such $\omega \in \Real_+$ as the \dfn{WKB trajectory period}.
It follows from general considerations (see \cite[\S9]{MR743423}) that if the WKB $\theta$-trajectory through $x_0$ is a closed trajectory, then all nearby WKB $\theta$-trajectories are also closed with the same period.

\paragraph{}
A nonclosed infinite WKB ray may tend to a single point, limit to a dense subset of $X$, or escape $X$ altogether.
Formally, the \dfn{limit} of an infinite WKB ray $\Gamma_\theta^\pm$ by definition respectively is the limit set
\eqn{
	\bar{\Phi^{-1} \Big( [\tau e^{i\theta},+\infty \cdot e^{i\theta}) \Big)}
\qtext{as}
	\tau \to +\infty
\qtext{or}
	\bar{\Phi^{-1} \Big( (-\infty \cdot e^{i\theta},\tau e^{i\theta}] \Big)}
\qtext{as}
	\tau \to -\infty
\fullstop
}
Obviously, this definition is independent of the chosen basepoint $x_0$ along the trajectory.
If the limit is a single point $x_\infty \in \Complex_x$, then this point (sometimes called an \textit{infinite critical point}) is necessarily a pole of $\DD_0$ of order $m \geq 2$, \cite[\S10.2]{MR743423}.
Given $\alpha \in \set{+, -}$, it also follows from general considerations that if the WKB $\theta$-trajectory $\alpha$-ray emanating from $x_0$ tends to an infinite critical point, then $x_0$ has a disc neighbourhood $U_0$ such that every WKB $\theta$-trajectory $\alpha$-ray emanating from $U_0$ tends to the same infinite critical point.

\paragraph{}
\dfn{Finite WKB rays} --- those with finite $\tau_+$ or $\tau_-$ --- are inadmissible for our construction of exact solutions in \autoref{210823153207}.
As $\tau$ approaches $\tau_{+}$ or $\tau_{-}$ respectively, such a WKB trajectory either tends to a turning point or escapes to the boundary of $X$ in finite time.
If it tends to a single point on the boundary of $X$, this point is either a turning point or a simple pole of the discriminant $\DD_0$, \cite[\S10.2]{MR743423}.
For this reason, turning points and simple poles are sometimes collectively referred to as \textit{finite critical points}.
A \dfn{singular WKB ray} is one that approaches a finite critical point.
They are important in the global analysis of exact solutions which will be discussed in detail elsewhere.

\paragraph{}
A \dfn{WKB $\theta$-strip domain} containing $x_0$ is any domain neighbourhood of $x_0$ which is swept out by generic WKB $\theta$-trajectories.
It necessarily has the form
\eqntag{\label{211102172208}
\begin{aligned}
	W_\theta
		&= W_\theta (x_1, r)
		\coleq \Phi^{-1} \big( H_\theta \big) \subset \Complex_x
\\ \text{where}\quad
	H_\theta &= H_\theta (z_1, r) \coleq \set{z ~\Big|~ \op{dist} (z, z_1 + e^{i\theta} \Real) < r } \subset \Complex_z
\fullstop
\end{aligned}
}
for some $r > 0$ and some $x_1 = \Phi^{-1} (z_1) \in X$.
Similarly, a \dfn{WKB $(\theta,\pm)$-halfstrip domain} containing $x_0$ is any domain neighbourhood of $x_0$ which is swept out by infinite WKB $\theta$-trajectory $\pm$-rays.
It necessarily has the form
\eqntag{\label{211102172208}
\begin{aligned}
	W^\pm_\theta
		&= W^\pm_\theta (x_1, r)
		\coleq \Phi^{-1} \big( H^\pm_\theta \big) \subset \Complex_x
\\ \text{where}\quad
	H^\pm_\theta &= H^\pm_\theta (z_1, r) \coleq \set{z ~\Big|~ \op{dist} (z, z_1 + e^{i\theta} \Real_\pm) < r } \subset \Complex_z
\fullstop
\end{aligned}
}
Note that we obviously have $W_\theta = W_\theta^- \cup W_\theta^+$.
The intersection $W_\theta^- \cap W_\theta^+ = \Phi^{-1} \big( \set{ |z - z_1| < r } \big)$ may be called a \dfn{WKB disc} around $x_1$, and it is clearly independent of $\theta$.

If the WKB $\theta$-trajectory through $x_0$ is not closed, then WKB $\theta$-strip $W_\theta$ is a simply connected domain conformally equivalent to the infinite strip $H_\theta$ via the Liouville transformation $\Phi : W_\theta \iso H_\theta$.

On the other hand, if the WKB $\theta$-trajectory through $x_0$ is closed, then $W_\theta$ is swept out by closed WKB $\theta$-trajectories, so $W_\theta$ has the topology of an annulus.
In this case, $W_\theta$ is sometimes called a \dfn{WKB $\theta$-ring domain}.
The Liouville transformation $\Phi$ is a multivalued holomorphic function on $W_\theta$, but the inverse $\Phi^{-1} : H_\theta \to W_\theta$ is still necessarily a local biholomorphism.

\newpage
\section{Exact Perturbation Theory}
\label{210823153207}

We can now state and prove our main results.
Throughout this section, we remain in the background setting of \autoref{210730141800}.
Namely, $X$ is a domain in $\Complex_x$ and $S \subset \Complex_\hbar$ is a sectorial domain with opening $A$.
In addition, we assume that $|A| = \pi$ so that $A = A_\theta \coleq (\theta -\tfrac{\pi}{2}, \theta + \tfrac{\pi}{2})$ for some $\theta \in \Real$.

\subsection{Existence and Uniqueness of Local Exact Solutions}
\label{211107195800}
The main result of this paper is the following theorem.

\begin{thm}[\textbf{Main Exact Existence and Uniqueness Theorem}]{211026151811}
\mbox{}\newline
Consider the Riccati equation
\begin{equation}\label{211118123850}
	\hbar \del_x f = a f^2 + b f + c
\end{equation}
whose coefficients $a,b,c$ are holomorphic functions of $(x,\hbar) \in X \times S$ admitting locally uniform asymptotic expansions $\hat{a}, \hat{b}, \hat{c} \in \cal{O} (X) \bbrac{\hbar}$ as $\hbar \to 0$ along $A$.
Assume that the leading-order discriminant $\DD_0 = b_0^2 - 4 a_0 c_0$ is not identically zero.
Fix a regular point $x_0 \in X$, a square-root branch $\sqrt{\DD_0}$ near $x_0$, and a sign $\alpha \in \set{+, -}$.
In addition, we assume the following hypotheses:
\begin{enumerate}
\item there is a WKB $(\theta, \alpha)$-halfstrip domain $W = W_\theta^\alpha \subset X$ containing $x_0$; assume in addition that $a_0$ is nonvanishing on $W$ if $\alpha = -$;
\item the asymptotic expansions of the coefficients $a,b,c$ are valid with Gevrey bounds as $\hbar \to 0$ along the closed arc $\bar{A}_\theta = [\theta - \tfrac{\pi}{2}, \theta + \tfrac{\pi}{2}]$, with respect to the asymptotic scale $\sqrt{\DD_0}$, uniformly for all $x \in W$:
\begin{equation}
\label{211028195400}
	a \simeq \hat{a}, \quad
	b \simeq \hat{b}, \quad
	c \simeq \hat{c}
\qquad
	\text{as $\hbar \to 0$ along $\bar{A}_\theta$, wrt $\sqrt{\DD_0}$, unif. $\forall x \in W$\fullstop}
\end{equation}
\end{enumerate}

Then the Riccati equation has a canonical local exact solution $f^\theta_\alpha$ near $x_0$ which is asymptotic to the formal solution $\hat{f}_\alpha$ as $\hbar \to 0$ in the direction $\theta$.
Namely, for any compactly contained domain $U_0 \Subset W$, there is a sectorial domain $S_0 \subset S$ with the same opening $A_\theta$ such that the Riccati equation has a unique holomorphic solution $f^\theta_\alpha$ on $U_0 \times S_0$ which is Gevrey asymptotic to $\hat{f}_\alpha$ as $\hbar \to 0$ along the closed arc $\bar{A}_\theta$ uniformly for all $x \in U_0$:
\eqntag{\label{211102180713}
	f^\theta_\alpha \simeq \hat{f}_\alpha
\qqqquad
	\text{as $\hbar \to 0$ along $\bar{A}_\theta$, unif. $\forall x \in U_0$\fullstop}
}
\end{thm}

We will prove this theorem in \autoref{211117174240}.
First, let us make some remarks.


\begin{rems}{211027171620}
\textbf{(1)}
For the reader's convenience, we recall here that hypothesis (2) in \autoref{211026151811} explicitly means that there are real constants $\AA, \BB > 0$ such that for all $n \geq 0$, all $x \in W$, and all sufficiently small $\hbar \in S$,
\begin{equation}
\label{210731112258}
	\left| a (x, \hbar) - \sum_{k=0}^{n-1} a_k (x) \hbar^k \right|
		\leq \Big| \sqrt{\DD_0 (x)} \Big| \AA \BB^n n! |\hbar|^n
\fullstop{,}
\end{equation}
and similarly for $b$ and $c$.
See \autoref{210714114437} for more details.
Likewise, the asymptotic condition \eqref{211102180713} reads explicitly as follows: there are real constants $\CC, \MM > 0$ such that for all $n \geq 0$, all $x \in U_0$, and all sufficiently small $\hbar \in S_0$,
\begin{equation}
\label{211102180942}
	\left| f_\alpha^\theta (x, \hbar) - \sum_{k=0}^{n-1} f^\alpha_k (x) \hbar^k \right|
		\leq \CC \MM^n n! |\hbar|^n
\fullstop
\end{equation}

\textbf{(2)}
If all the hypotheses of \autoref{211026151811} are satisfied for both signs $\alpha = +, -$, then we obviously obtain a pair of distinct canonical exact solutions $f^\theta_+, f^\theta_-$ near $x_0$.
Let us also note that the solution $f_\alpha^\theta$ obviously does not depend in any serious way on the chosen basepoint $x_0$.

\textbf{(3)}
In many applications, including the exact WKB analysis of Schrödinger, the coefficients of the Riccati equation \eqref{211118123850} do not satisfy hypothesis (2) in \autoref{211026151811} on the nose and we have to do an additional transformation in order to apply our theorem.
This is discussed in \autoref{211103172432}.

\textbf{(4)}
The somewhat abstract general hypotheses of this theorem get significantly simplified in many notable situations which are discussed in \autoref{211107122319}.

\textbf{(5)}
The sectorial domain $S_0$ in the conclusion of \autoref{211026151811} can be chosen to be a Borel disc $\set{\hbar ~\big|~ \Re (e^{i\theta}/\hbar) > 1/d_0}$ bisected by the direction $\theta$ of sufficiently small diameter $d_0 > 0$.
\end{rems}

We have the following immediate corollary of the uniqueness property of canonical exact solutions.

\begin{cor}[\textbf{extension to larger domains}]{211102175415}
Let $U \subset X$ be a domain free of turning points that supports a univalued square-root branch $\sqrt{\DD_0}$.
Fix a sign $\alpha \in \set{+, -}$ such that the leading-order solution $f_0^\alpha$ is holomorphic on $U$.
In addition, assume that hypotheses (1-3) in \autoref{211026151811} are satisfied for every point $x_0 \in U$.

Then the Riccati equation \eqref{211118123850} has a canonical exact solution $f^\theta_\alpha$ on $U$ asymptotic to the formal solution $\hat{f}_\alpha$ as $\hbar \to 0$ in the direction $\theta$.
Namely, there is a domain $\mathbb{U} \subset U \times S$ and a holomorphic solution $f^\theta_\alpha$ defined on $\mathbb{U}$ with the following property: for every point $x_0 \in U$, there is a neighbourhood $U_0 \subset U$ of $x_0$ and a sectorial domain $S_0 \subset S$ with the same opening $A_\theta$ such that $U_0 \times S_0 \subset \mathbb{U}$ and $f_\alpha^\theta$ is the unique holomorphic solution on $U_0 \times S_0$ satisfying \eqref{211102180713}.
In particular, $f_\alpha^\theta$ is the unique solution on $\UUU$ with the following locally uniform Gevrey asymptotics:
\eqntag{\label{211117171819}
	f^\theta_\alpha \simeq \hat{f}_\alpha
\qqqquad
	\text{as $\hbar \to 0$ along $\bar{A}_\theta$, loc.unif. $\forall x \in U$\fullstop}
}
\end{cor}


In particular, the domain $U \subset X$ in \autoref{211102175415} can be a union of WKB halfstrips.
In fact, examining the proof of \autoref{211026151811} more closely, it is readily seen that on any WKB halfstrip we can state the asymptotic property of canonical exact solutions more precisely as follows.

\begin{prop}[\textbf{asymptotics on WKB halfstrips}]{211118140510}
Assume all the hypotheses of \autoref{211026151811}.
Then the canonical local exact solution $f_\alpha^\theta$ uniquely extends to an exact solution on $W$ with the following locally uniform Gevrey asymptotics:
\eqntag{\label{211113104236}
	f^\theta_\alpha \simeq \hat{f}_\alpha
\qqqquad
	\text{as $\hbar \to 0$ along $\bar{A}_\theta$, loc.unif. $\forall x \in W$\fullstop}
}
In fact, even more is true.
Let $r > 0$ be such that $W = W_\theta^\alpha (x'_0, r)$.
For any $r_0 \in (0, r)$, let $W_0 \coleq W_\theta^\alpha (x'_0, r_0)$.
Then there is a sectorial domain $S_0 \subset S$ with the same opening $A_\theta$ such that the canonical exact solution $f_\alpha^\theta$ extends to a holomorphic solution on $W_0 \times S_0$ with the following uniform Gevrey asymptotic property:
\eqntag{\label{211113110026}
	f^\theta_\alpha \simeq \hat{f}_\alpha
\qqqquad
	\text{as $\hbar \to 0$ along $\bar{A}_\theta$, unif. $\forall x \in W_0$\fullstop}
}
Thus, $f_\alpha^\theta$ is a well-defined holomorphic solution on a domain of the form
\eqn{
	\WWW \coleq \Cup_{r_0 \in (0, r)} W_0 \times S_0 ~ \subset W \times S
\fullstop
}
\end{prop}

\begin{proof}
The fact that $f_\alpha^\theta$ extends to any $W_0$ as stated follows immediately from the proof of \autoref{211026151811}.
The uniqueness part of the construction guarantees that all these extensions coincide.
\end{proof}

\subsubsection{Riccati Equations on Horizontal Halfstrips}

The strategy of the proof of \autoref{211026151811} is to use the Liouville transformation $\Phi_\alpha$ to transform the Riccati equation into one in standard form over a horizontal halfstrip in the $z$-space, and then apply the Borel-Laplace method.
First, we give a general description of this standard form of the Riccati equation and prove the corresponding version of the exact existence and uniqueness theorem (\autoref{210714172405}).
Then we will show that any Riccati equation satisfying the assumptions of \autoref{211026151811} can be put into this standard form thereby deducing our claims.

\paragraph{}
Let $H_+ \subset \Complex_z$ be a horizontal halfstrip around the positive real axis $\Real_+ \subset \Complex_z$ of some radius $r > 0$ and let $S_+ \subset \Complex_\hbar$ a Borel disc bisected by the positive real axis of some diameter $d > 0$:
\eqntag{\label{211103174239}
	H_+ \coleq \set{z ~\big|~ \op{dist} (z, \Real_+) < r}
\qtext{and}
	S_+ \coleq \set{ \hbar ~\big|~ \Re (1/\hbar) > 1/d}
\fullstop
}
Recall that the opening of $S_+$ is the semicircular arc $A_+ \coleq (-\frac{\pi}{2},+\frac{\pi}{2})$.
Consider the following singularly perturbed Riccati equation on $H_+ \times S_+$:
\eqntag{
\label{210714172412}
	\hbar \del_z \FF = \FF + \hbar \big( \AA_2 \FF^2 + \AA_1 \FF + \AA_0 \big)
\fullstop{,}
}
where $\AA_i$ are holomorphic functions of $(z, \hbar) \in H_+ \times S_+$ which admit uniform Gevrey asymptotic expansions as $\hbar \to 0$ along the closed arc $\bar{A}_+$:
\eqntag{
	\AA_i \simeq \hat{\AA}_i,
\qquad
\text{as $\hbar \to 0$ along $\bar{A}_+$, unif. $\forall z \in H_+$\fullstop}
}
Denote their leading-order parts by $a_i = a_i (z)$.
In symbols, $\AA_i \in \cal{G} \big(H_+; \bar{A}_+ \big)$.

The corresponding leading-order equation is simply $\FF_0 = 0$.
By \autoref{200118111737} part (1), this Riccati equation has a unique formal solution $\hat{\FF}_+ \in \cal{O} \big( H_+ \big) \bbrac{\hbar}$, and its leading order part is $\FF_0^+ = 0$ and its next-to-leading order part is $\FF_1^+ = - a_0$.

\newpage
\begin{lem}[\textbf{Main Technical Lemma}]{210714172405}
For every $r_0 \in (0, r)$, there is $d_0 \in (0, d]$ such that the Riccati equation \eqref{210714172412} has a canonical exact solution $\FF_+$ defined on
\eqntag{\label{210522170909}
	H^+_0 \times S^+_0
	\coleq 
	\set{z ~\big|~ \op{dist} (z, \Real_+) < r_0} 
		\times \set{ \hbar ~\big|~ \Re (1/\hbar) > 1/d_0}
	\subset
	H_+ \times S_+
\fullstop
}
Namely, $\FF_+$ is the unique holomorphic solution on $H^+_0 \times S^+_0$ which admits the formal solution $\hat{\FF}_+$ as its uniform Gevrey asymptotic expansion along $\bar{A}_+$:
\eqntag{\label{210522173054}
	\FF_+ (z, \hbar) \simeq \hat{\FF}_+ (z, \hbar)
\quad
\text{as $\hbar \to 0$ along $\bar{A}_+$, unif. $\forall z \in H^+_0$\fullstop}
}
In symbols, $\FF_+ \in \cal{G} (H^+_0; \bar{A}_+)$.
Moreover, it has the following properties.
\begin{enumerate}
\item [\textup{(P1)}]
The formal Borel transform
\eqn{
	\hat{\phi}_+ (z, \xi) \coleq \hat{\Borel} [\, \hat{\FF}_+ \,] (z, \xi)
}
converges uniformly on $H^+_0$.
In symbols, $\hat{\phi}_+ \in \cal{O} (H^+_0) \set{\xi}$ and $\hat{\FF}_+ \in \cal{G} (H^+_0) \bbrac{\hbar}$.
\item [\textup{(P2)}]
For any $\epsilon \in (0, r - r_0)$, let $\Xi_+ \coleq \set{ \xi ~\big|~ \op{dist} (\xi, \Real_+) < \epsilon}$.
Then the analytic Borel transform
\eqn{
	\phi_+ (z, \xi) \coleq \Borel_+ [\, \FF_+ \,] (z, \xi)
}
is uniformly convergent for all $(z, \xi) \in H^+_0 \times \Xi_+$.
\item [\textup{(P3)}]
The Laplace transform $\Laplace_+ \big[ \, \phi_+ \, \big] (z, \hbar)$ is uniformly convergent for all $(z, \hbar) \in H^+_0 \times S^+_0$ and satisfies
\eqntag{\label{210714174600}
	\FF_+ (z, \hbar) 
		= \Laplace_+ \big[ \, \phi_+ \, \big] (z, \hbar)
\fullstop
}
\item [\textup{(P4)}]
Therefore, $\FF_+$ is the uniform Borel resummation of its asymptotic power series $\hat{\FF}_+$: for all $(z, \hbar) \in H^+_0 \times S^+_0$,
\eqntag{
	\FF_+ (z, \hbar) = \cal{S}_+ \big[ \, \hat{\FF}_+ \, \big] (z, \hbar)
\fullstop
}
\item [\textup{(P5)}]
If the coefficients $\AA_0, \AA_1, \AA_2$ are periodic in $z$ with period $\omega \in \Complex$, then so is $\FF_+$.
\end{enumerate}
\end{lem}

\begin{proof}
We use the Borel-Laplace method to construct the exact solution $\FF_+$.
Namely, we first apply the Borel transform to obtain a first-order nonlinear PDE which is easy to rewrite as an integral equation.
Then most of the heavy-lifting is constrained to solving this integral equation, which we do via the method of successive approximations.
The desired solution $\FF_+$ is then obtained by applying the Laplace transform.

\paragraph*{Uniqueness.}
Suppose $\FF_+,\FF'_+$ are two such exact solutions defined on $H^+_0 \times S^+_0$.
Their difference $\FF_+ - \FF'_+$ is a holomorphic function on $H^+_0 \times S^+_0$ which is uniformly Gevrey asymptotic to $0$ as $\hbar \to 0$ along $\bar{A}_+$.
By the uniqueness claim in Nevanlinna's Theorem (\autoref{210617120300}), there can be only one holomorphic function on $S^+_0$ (namely, the constant function $0$) which is Gevrey asymptotic to $0$ as $\hbar \to 0$ along $\bar{A}_+$.
Thus, $\FF_+ - \FF'_+$ is must be the zero function.

\paragraph*{Step 1: The analytic Borel transform.}
Since $\AA_i \in \cal{G} \big(H_+; \bar{A}_+ \big)$, it follows from Nevanlinna's Theorem (\autoref{210617120300}) that there is some tubular neighbourhood 
\eqntag{\label{210726172442}
	\Xi_+ \coleq \set{ \xi ~\big|~ \op{dist} (\xi, \Real_+) < \epsilon}
}
for some $\epsilon > 0$ such that the analytic Borel transforms
\eqntag{\label{210726172655}
	\alpha_i (z, \xi) \coleq \Borel_+ [ \: \AA_i \: ] (z, \xi)
}
are holomorphic functions on $H_+ \times \Xi_+ \subset \Complex^2_{z\xi}$ with uniformly at most exponential growth as $|\xi| \to + \infty$.
Moreover, for all $(z,\xi) \in H_+ \times \Xi_+$,
\eqntag{\label{210726173018}
	\AA_i (z, \hbar) = a_i (z) + \Laplace_+ [ \: \alpha_i \: ]
\fullstop
}

Dividing \eqref{210714172412} through by $\hbar$ and applying the analytic Borel transform $\Borel_+$, we obtain the following PDE with convolution product:
\eqntag{\label{200117154820}
	\del_z \varphi - \del_\xi \varphi
		= 	\alpha_0
			+ a_1 \varphi
			+ \alpha_1 \ast \varphi
			+ a_2 \varphi \ast \varphi
			+ \alpha_2 \ast \varphi \ast \varphi
\fullstop{,}
}
where the unknown variables $\varphi$ and $\FF$ are related by 
\eqn{
	\varphi = \Borel_+ [\, \FF \,]
\qqtext{and}
	\FF = \Laplace_+ [\, \varphi \,]
\fullstop
}
We have thus reformulated the problem of solving the Riccati equation \eqref{210714172412} as the problem of solving the PDE \eqref{200117154820}: if $\varphi$ is a solution of \eqref{200117154820}, then its Laplace transform $\FF = \Laplace_+ [\, \varphi \,]$ is a solution of \eqref{210714172412}, provided that $\varphi$ has at most exponential growth at infinity in $\xi$ uniformly in $z$.
See \autoref{200223091621} for more details.

\paragraph*{Step 2: The integral equation.}
The principal part of the PDE \eqref{200117154820} has constant coefficients, so it is easy to rewrite it as an equivalent integral equation as follows.
Consider the holomorphic change of variables
\eqn{
	(z, \xi) \overset{\TT}{\mapstoo} (w, t) \coleq (z + \xi, \xi)
\qtext{and its inverse}
	(w, t) \overset{\TT^{-1}}{\mapstoo} (z, \xi) = (w - t, t)
\fullstop
}
For any function $\alpha = \alpha (z, \xi)$ of two variables, introduce the following notation:
\eqn{
	\TT^\ast \alpha (z, \xi) \coleq \alpha \big( \TT (z, \xi) \big) = \alpha (z + \xi, \xi)
\qtext{and}
	\TT_\ast \alpha (w, t) \coleq \alpha \big( \TT^{-1} (w, t) \big) = \alpha (w - t, t)
\fullstop
}
Note that $\TT^\ast \TT_\ast \alpha = \alpha$.
Under this change of coordinates, the differential operator $\del_z - \del_\xi$ transforms into $- \del_t$, and so the lefthand side of the \eqref{200117154820} becomes $- \del_t \big( \TT_\ast \varphi)$.
Integrating from $0$ to $t$, and imposing the initial condition
\eqntag{\label{210726174317}
	\varphi (z, 0)
		= \varphi_0 (z) \coleq a_0 (z)
\fullstop{,}
}
the lefthand side of the PDE \eqref{200117154820} becomes $-\TT_\ast \varphi$.
Applying $\TT^\ast$, we therefore obtain the following integral equation for $\varphi = \varphi (z, \xi)$:
\eqntag{\label{190312202445}
	\varphi
		= \varphi_0 - \TT^\ast \int_0^t \TT_\ast
			\Big( \alpha_0
			+ a_1 \varphi
			+ \alpha_1 \ast \varphi
			+ a_2 \varphi \ast \varphi
			+ \alpha_2 \ast \varphi \ast \varphi
			\Big) \dd{u}
\fullstop
}
Explicitly, this integral equation reads as follows:
{\scriptsize
\eqns{
	\varphi (z, \xi)
		&= a_0 (z) - \int_0^\xi
			\Bigg( 
				\alpha_0 (z + \xi - u, u)
				~+~ a_1 (z + \xi - u) \cdot \varphi (z + \xi - u, u)
				~+~ (\alpha_1 \ast \varphi) (z + \xi - u, u)
		\\		&\phantom{=}\:\qqqqqqqqqquad
				~+~ a_2 (z + \xi - u) \cdot (\varphi \ast \varphi) (z + \xi - u, u)
				~+~ (\alpha_2 \ast \varphi \ast \varphi) (z + \xi - u, u)
					\Bigg) \dd{u}
\fullstop
}
}%
Here, the integration is done along a straight line segment from $0$ to $\xi$.
Note also that the convolution products are with respect to the second argument; i.e.,
\eqns{
	(\alpha_1 \ast \varphi) (t_1, t_2)
		&= \int_0^{t_2} \alpha_1 (t_1, t_2 - y) \varphi (t_1, y) \dd{y}
\fullstop{,}
\\
	(\alpha_2 \ast \varphi \ast \varphi) (t_1, t_2)
		&= \int_0^{t_2} \alpha_2 (t_1, t_2 - y) \int_0^{y} \varphi (t_1, y - y') \varphi (t_1, y') \dd{y'} \dd{y}
\fullstop
}
Introduce the following notation: for any function $\alpha = \alpha (z, \xi)$ of two variables,
\eqntag{\label{200223171526}
\small
	\II_+ \big[ \alpha \big] (z, \xi) 
		\coleq - \TT^\ast \int_0^t \TT_\ast \alpha \dd{u}
		= - \int_0^\xi	\alpha (z + \xi - u, u) 	\dd{u}
		= \int_0^{\xi} \alpha (z + t, \xi - t) \dd{t}
\fullstop{,}
}
where the integration path is again the straight line segment connecting $0$ to $\xi$.
Then the integral equation \eqref{190312202445} can be written more succinctly as:
\eqntag{\label{190312202516}
	\varphi = \varphi_0 + \II_+ \Big[ \, \alpha_0
					+ a_1 \varphi
					+ \alpha_1 \ast \varphi
					+ a_2 \varphi \ast \varphi
					+ \alpha_2 \ast \varphi \ast \varphi \, \Big]
\fullstop
}

\paragraph*{Step 3: Method of successive approximations.}
To solve \eqref{190312202516}, we use the method of successive approximations.
Consider a sequence of functions $\set{\varphi_n}_{n=0}^\infty$ defined recursively by
\eqn{
	\varphi_0 = a_0
\qqtext{and}
	\varphi_1 \coleq \II_+ \big[ \alpha_0 + a_1 \varphi_0 \big]
\fullstop{,}
}
and for $n \geq 2$ by the following formula:
\eqntag{\label{180824203055}
	\varphi_n 
		\coleq \II_+ \left[ 
			a_1 \varphi_{n-1} 
			+ \alpha_1 \ast \varphi_{n-2} 
			+ \sum_{\substack{i,j \geq 0 \\ i + j = n-2}}
				a_2 \varphi_i \ast \varphi_j
			+ \sum_{\substack{i,j \geq 0 \\ i + j = n-3}}
				\alpha_2 \ast \varphi_i \ast \varphi_j
			\right]
\fullstop
}
Explicitly, the first few members of the sequence $\set{\varphi_n}_n$ are:
{\small
\eqns{
	\varphi_2 &= \II_+ \big[ a_1 \varphi_1 + \alpha_1 \ast \varphi_0 + a_2 \varphi_0 \ast \varphi_0 \big]
\fullstop{,}
\\	\varphi_3 &= \II_+ \big[ a_1 \varphi_2 + \alpha_1 \ast \varphi_1 + a_2 \varphi_1 \ast \varphi_0 + a_2 \varphi_0 \ast \varphi_1 + \alpha_2 \ast \varphi_0 \ast \varphi_0 \big]
\fullstop{,}
\\	\varphi_4 &= \II_+ \big[ a_1 \varphi_3 + \alpha_1 \ast \varphi_2 + a_2 \varphi_2 \ast \varphi_0 + a_2 \varphi_1 \ast \varphi_1 + a_2 \varphi_0 \ast \varphi_2 + \alpha_2 \ast \varphi_1 \ast \varphi_0 + \alpha_2 \ast \varphi_0 \ast \varphi_1 \big]
\fullstop
}}%
Of course, $\varphi_0$ is independent of $\xi$, so $\varphi_0 \ast \varphi_0 = \varphi_0^2 \xi$.

\paragraph*{Main Technical Claim.}
\textit{Let $\epsilon$ be so small that $\epsilon < r - r_0$. Then the infinite series}
\eqntag{\label{190306115025}
	\varphi_+ (z, \xi) \coleq \sum_{n=0}^\infty \varphi_n (z, \xi)
}
\textit{defines a holomorphic solution of the integral equation \eqref{190312202516} on the domain
\eqn{
	\mathbf{H}_+ \coleq \set{ (z, \xi) \in H_+ \times \Xi_+ ~\big|~ z + \xi \in H_+}
}
with at most exponential growth at infinity in $\xi$; more precisely, it satisfies the following uniform exponential bound: there are real constants $\AA, \KK > 0$ such that}
\eqntag{\label{200119155725}
	\big| \varphi_+ (z, \xi) \big| \leq \AA e^{\KK |\xi|}
\qqquad
	\text{$\forall (z, \xi) \in \mathbf{H}_+$}
\fullstop
}
\textit{In particular, $\varphi_+$ is a well-defined holomorphic solution on $H^+_0 \times \Xi_+ \subset \mathbf{H}_+$ where it satisfies the exponential estimate above.}

Assuming this claim, only one step remains in order to complete the proof of the \hyperlink{210714172405}{Main Technical Lemma}, which is to take the Laplace transform of $\varphi_+$.

\paragraph*{Step 4: The Laplace transform.}
Let
\eqntag{\label{200603190110}
	\FF_+ (z, \hbar) 
		\coleq \Laplace_+ \big[ \varphi_+ \big] (z, \hbar)
		= \int_{0}^{+\infty} e^{- \xi / \hbar} \varphi (z, \xi) \dd{\xi}
\fullstop
}
This integral is uniformly convergent for all $z \in H^+_0$ provided that ${\Re (\hbar^{-1}) > \CC_2}$.
Thus, if we take $d_0 \in (0, d]$ strictly smaller than $1/\CC_2$, then formula \eqref{200603190110} defines a holomorphic solution of the Riccati equation \eqref{210714172412} on the domain $H^+_0 \times S^+_0$ where $S^+_0 \coleq \set{ \hbar ~\big|~ \Re (\hbar^{-1}) > 1 / d_0 }$.
Furthermore, Nevanlinna's Theorem (\autoref{210617120300}) implies that $\FF_+$ admits a uniform Gevrey asymptotic expansion on $H^+_0$ as $\hbar \to 0$ along $\bar{A}_+$, and this asymptotic expansion is necessarily the formal solution $\hat{\FF}_+$.

\enlargethispage{1cm}
\paragraph*{Proof of the Main Technical Claim.}
If we assume for the moment that the series \eqref{190306115025} is uniformly convergent on $\mathbf{H}_+$, it is easy to check by direct substitution that it satisfies the integral equation \eqref{190312202516} (see \autoref{190907145933} for full details).
To demonstrate uniform convergence of the series $\varphi_+$, we first note the following exponential estimates on the coefficients of \eqref{190312202516}: there are constants $\CC, \LL > 0$ such that for each $i = 0,1,2$ and for all $(z, \xi) \in \mathbf{H}_+$,
\eqntag{\label{190907144006}
	|a_i| \leq \CC
	\qqtext{and}
	|\alpha_i| \leq \CC e^{\LL |\xi|}
\fullstop
}
We prove the Main Technical Claim by showing that there are constants $\AA, \MM > 0$ such that for all $n$ and for all $(z, \xi) \in \mathbf{H}_+$,
\eqntag{\label{200220115129}
	\big| \varphi_n (z, \xi) \big| \leq \AA \MM^n \frac{|\xi|^n}{n!} e^{\LL |\xi|}
\fullstop
}
This is enough to deduce the uniform convergence of the infinite series $\varphi_+$ as well as the exponential estimate \eqref{200119155725} by taking $\KK \coleq \MM + \LL$ because
\eqn{
	\big| \varphi_+ (z, \xi) \big|
		\leq \sum_{n=0}^\infty |\varphi_n|
		\leq \sum_{n=0}^\infty \AA \MM^n \frac{|\xi|^n}{n!} e^{\LL |\xi|}
		= \AA e^{ (\MM + \LL) |\xi|}
\fullstop
}

To show \eqref{200220115129}, we will first recursively construct a sequence of positive real numbers $(\MM_n)_{n=0}^\infty$ such that for all $n$ and for all $(z, \xi) \in \mathbf{H}_+$,
\eqntag{\label{180824195140}
	\big| \varphi_n (z, \xi) \big| \leq \MM_n \frac{|\xi|^n}{n!} e^{\LL |\xi|}
\fullstop
}
We will then show that there are $\AA, \MM > 0$ such that $\MM_n \leq \AA \MM^n$ for all $n$.

\textsc{Construction of $\MM_0, \MM_1$.}
We can take $\MM_0 \coleq \CC$ because $\big| \varphi_0 (z, \xi) \big| = \big| a_0 (z) \big| \leq \CC$.
For clarity, let us also say that we can take $\MM_1 \coleq \CC ( 1 + \CC)$ because \autoref{180824194855} gives the estimate
\eqn{
	\big| \varphi_1 (z, \xi) \big|
		\leq \int_0^{\xi} |\alpha_0| |\dd{u}| + \int_0^{\xi} |a_1| |\varphi_0| |\dd{u}|
		\leq \CC ( 1 + \CC) \int_0^{|\xi|} e^{\LL r} \dd{r}
		\leq \CC ( 1 + \CC) |\xi| e^{\LL |\xi|}
\fullstop
}

\textsc{Construction of $\MM_n$ for $n \geq 2$.}
We assume that the estimate \eqref{180824195140} holds for $\varphi_0, \ldots, \varphi_{n-1}$ and derive an estimate for $\varphi_n$.
Using \autoref{180824194855}, \ref{180824192144}, and \ref{180824195940}, we obtain the following bounds on the terms in the recursive formula \eqref{180824203055}:
\eqns{
	\Big| \II_+ \big[ a_1 \varphi_{n-1} \big] \Big|
		&\leq \CC \MM_{n-1} \frac{|\xi|^n}{n!} e^{\LL |\xi|}
\fullstop{,}
\\	\Big| \II_+ \big[ \alpha_1 \ast \varphi_{n-2} \big] \Big|
		&\leq \CC \MM_{n-2} \frac{|\xi|^n}{n!} e^{\LL |\xi|}
\fullstop{,}
\\	\Big| \II_+ \big[ a_2 \varphi_i \ast \varphi_j \big] \Big|
		&\leq \CC \MM_i \MM_j \frac{|\xi|^n}{n!} e^{\LL |\xi|}
\quad \text{if $i + j = n-2$ \fullstop{,}}
\\	\Big| \II_+ \big[ \alpha_2 \ast \varphi_i \ast \varphi_j \big]  \Big|
		&\leq \CC \MM_i \MM_j \frac{|\xi|^n}{n!} e^{\LL |\xi|}
\quad \text{if $i + j = n-3$ \fullstop}
}
Using these estimates in \eqref{180824203055}, we find:
\eqns{
	\big| \varphi_n \big|
		&\leq 
		\Big| \II_+ \big[ a_1 \varphi_{n-1} \big] \Big|
		+ \Big| \II_+ \big[ \alpha_1 \ast \varphi_{n-2} \big] \Big|
\\		&\phantom{\leq \Big| \II \big[ a_1 \varphi_{n-1} \big] \Big|~}
		+ \sum_{\substack{i,j \geq 0 \\ i + j = n-2}}
				\Big| \II_+ \big[ a_2 \varphi_i \ast \varphi_j \big] \Big|
		+ \sum_{\substack{i,j \geq 0 \\ i + j = n-3}}
				\Big| \II_+ \big[ \alpha_2 \ast \varphi_i \ast \varphi_j \big] \Big|
\\
		&\leq \CC 
			\left( \MM_{n-1} + \MM_{n-2}
				+ \sum_{\substack{i,j \geq 0 \\ i + j = n-2}} 
					\MM_i \MM_j
				+ \sum_{\substack{i,j \geq 0 \\ i + j = n-3}}
					\MM_i \MM_j
			\right)
			\frac{|\xi|^n}{n!} e^{\LL |\xi|}
\fullstop
}
We can therefore define, for all $n \geq 2$,
\eqntag{\label{180824204834}
	\MM_n \coleq 
			\CC 
			\left( \MM_{n-1} + \MM_{n-2}
				+ \sum_{\substack{i,j \geq 0 \\ i + j = n-2}} 
					\MM_i \MM_j
				+ \sum_{\substack{i,j \geq 0 \\ i + j = n-3}}
					\MM_i \MM_j
			\right)
\fullstop
}
Thus, for example, 
$\MM_2 = \CC ( \MM_1 + \MM_0 + \MM_0^2 )$,
$\MM_3 = \CC ( \MM_2 + \MM_1 + 2 \MM_0 \MM_1 + \MM_0^2)$, 
$\MM_4 = \CC ( \MM_3 + \MM_2 + 2 \MM_0 \MM_2 + \MM_1^2 + 2 \MM_0 \MM_1)$,
and so on.

\textsc{Bounds on $\MM_n$.}
Consider the following power series in an abstract variable $t$:
\eqn{
	\hat{p} (t) \coleq \sum_{n=0}^\infty \MM_n t^n \in \Complex \bbrac{t}
\fullstop
}
We will show that $\hat{p} (t)$ is in fact a convergent power series.
First, we observe that $\hat{p} (0) = \MM_0 = \CC$ and that $\hat{p} (t)$ satisfies the following algebraic equation:
\eqntag{\label{180824205336}
	\hat{p} = \CC \Big( 1 + t + \hat{p} t + \hat{p} t^2 + \hat{p}^2 t^2 + \hat{p}^2 t^3 \Big)
\fullstop{,}
}
which can be seen by expanding and comparing the coefficients using the defining formula \eqref{180824204834} for $\MM_n$.
Consider the holomorphic function $\GG = \GG (p,t)$ of two variables, defined by
\eqn{
	\GG (p,t) \coleq - p + \CC \Big( 1 + t + pt + pt^2 + p^2t^2 + p^2t^3 \Big)
\fullstop
}
It has the following properties:
\eqn{
	\GG (\CC,0) = 0,
\qquad
	\evat{\frac{\del \GG}{\del p}}{(p,t) = (\CC,0)} = -1 \neq 0
\fullstop
}
Thus, by the Holomorphic Implicit Function Theorem, there exists a function $p (t)$, holomorphic at $t = 0$, satisfying $p (0) = \CC$ and $\GG \big(p (t), t \big) = 0$ for all $t$ sufficiently close to $0$.
Since $\hat{p} (0) = \CC$ and $\GG \big(\hat{p} (t), t \big) = 0$, the power series $\hat{p} (t)$ must be the Taylor expansion of $p (t)$ at $t = 0$.
As a result, $\hat{p} (t)$ is in fact a convergent power series, which means its coefficients grow at most exponentially: there are constants $\AA, \MM > 0$ such that $\MM_n \leq \AA \MM^n$ for all $n$.
This completes the proof of the Main Technical Claim and therefore of the \hyperlink{210714172405}{Main Technical Lemma}.
\end{proof}

\subsubsection{Proof of \autoref*{211026151811}}
\label{211117174240}

We can now finish the proof of the main result in this paper.

\begin{proof}[Proof of \autoref{211026151811}.]
We immediately restrict attention to a Borel disc in the $\hbar$-plane of some diameter $d > 0$ bisected by the direction $\theta$; i.e., without loss of generality, assume that $S = \set{ \hbar ~\big|~ \Re (e^{i\theta}/\hbar) > 1/ d}$.
Note that the rotation $\hbar \mapsto e^{-i\theta} \hbar$ sends $S$ to the Borel disc $S_+ = \set{ \hbar ~\big|~ \Re (1/\hbar) > 1/ d}$ from \eqref{211103174239}.

Next, let $r > 0$ be such that $W \coleq W^\alpha_\theta = W^\alpha_\theta (x_0; r) = \Phi^{-1} ( H )$ is the WKB halfstrip from the hypothesis, where $H \coleq H^\alpha_\theta = H^\alpha_\theta (0, r) = \set{z ~\big|~ \op{dist} (z, e^{i\theta} \Real_\alpha) < r }$.
Recall that $\Phi^{-1} : H \to W$ is a local biholomorphism.
Put $H_+ \coleq \set{z ~\big|~ \op{dist} (z, \Real_+) < r }$ and $\Phi_\alpha^\theta \coleq \varepsilon_\alpha e^{-i \theta} \Phi$, so that $W = (\Phi_\alpha^\theta)^{-1} (H_+)$.
Furthermore, $\Phi_\alpha^\theta$ transforms the differential operator $\frac{\varepsilon_\alpha}{\sqrt{\DD_0}} \del_x$ into $e^{-i\theta} \del_z$.

Consider now holomorphic functions of $(x, \hbar) \in W \times S$ denoted by $a_\ast, a_{\ast\ast}$, $b_\ast, b_{\ast\ast}$, $c_\ast, c_{\ast\ast}$, obtained from $a,b,c$ by removing the leading and the next-to-leading order terms respectively; i.e., they are defined by the following relations:
\eqntag{\label{200525200550}
\begin{aligned}
	a 	&= a_0 + \hbar a_\ast
&\text{and}\qquad
	a_\ast &= a_1 + \hbar a_{\ast\ast}
\fullstop{,}
\\
	b 	&= b_0 + \hbar b_\ast
&\text{and}\qquad
	b_\ast &= b_1 + \hbar b_{\ast\ast}
\fullstop{,}
\\
	c 	&= c_0 + \hbar c_\ast
&\text{and}\qquad
	c_\ast &= c_1 + \hbar c_{\ast\ast}
\fullstop
\end{aligned}
}
Recall that the leading and the next-to-leading orders $f_0^\alpha, f_1^\alpha$ of the formal solution $\hat{f}_\alpha$ are holomorphic functions on $W$ that satisfy the following identities:
\eqntag{\label{200522162901}
\begin{gathered}
	a_0 (f_0^\alpha)^2 + b_0 f_0^\alpha + c_0 = 0
\qqtext{and}
	\varepsilon_\alpha \sqrt{\DD_0} = 2a_0 f_0^\alpha + b_0
\fullstop{,}
\\
	\del_x f_0^\alpha = \varepsilon_\alpha \sqrt{\DD_0} f_1^\alpha + a_1 (f_0^\alpha)^2 + b_1 f_0^\alpha + c_1
\fullstop{,}
\end{gathered}
}
where $\varepsilon_\pm = \pm 1$.
Using these expressions, a straightforward calculation shows that the change of the unknown variable $f \mapsto \tilde{f}$ given by
\eqntag{\label{200522163534}
	f = f_0^\alpha + \hbar (f_1^\alpha + \tilde{f})
}
transforms the Riccati equation \eqref{211118123850} into the following Riccati equation on $W \times S$:
\eqntag{\label{200526092203}
	\varepsilon_\alpha \frac{\hbar}{\sqrt{\DD_0}} \del_x \tilde{f}
		- \tilde{f}
	= \hbar \big( \tilde{a} \tilde{f}^2 + \tilde{b} \tilde{f} + \tilde{c} \big)
\fullstop{,}
}
where
\eqntag{\label{200529110135}
\begin{gathered}
	\tilde{a} \coleq \frac{\varepsilon_\alpha}{\sqrt{\DD_0}} a
\fullstop{,}
\qqquad
	\tilde{b} \coleq \frac{\varepsilon_\alpha}{\sqrt{\DD_0}} \Big( b_\ast +  2a f_1^\alpha + 2 a_\ast f_0^\alpha \Big)
\fullstop{,}
\\
	\tilde{c} \coleq \frac{\varepsilon_\alpha}{\sqrt{\DD_0}} \bigg(
		- \del_x f_1^\alpha
		+ a (f_1^\alpha)^2 + \big(2 a_{\ast} f_0^\alpha + b_{\ast} \big) f_1^\alpha + \big( a_{\ast\ast} (f_0^\alpha)^2 + b_{\ast\ast} f_0^\alpha + c_{\ast\ast} \big)
		\bigg)
\fullstop
\end{gathered}
}

Finally, transforming equation \eqref{200526092203} by $\Phi_\alpha^\theta$ and applying a rotation $\hbar \mapsto e^{-i\theta} \hbar$, we obtain a Riccati equation on $H_+ \times S_+$ of the form \eqref{210714172412} where the coefficients $\AA_0, \AA_1, \AA_2$ are given by
\eqntag{
	\AA_2 (z, \hbar) \coleq \tilde{a} \big( x(z), e^{i\theta} \hbar \big),
\quad
	\AA_1 (z, \hbar) \coleq \tilde{b} \big( x(z), e^{i\theta} \hbar \big),
\quad
	\AA_0 (z, \hbar) \coleq \tilde{c} \big( x(z), e^{i\theta} \hbar \big),
}
where $x(z) = (\Phi^\theta_\alpha)^{-1} (z)$ and the unknown variables $\tilde{f}$ and $\FF$ are related by
\eqntag{
	\FF (z, \hbar) = \tilde{f} \big( x(z), e^{i\theta} \hbar \big)
\fullstop
}
\autoref{211026151811} now follows from the \hyperref[210714172405]{Main Technical Lemma \ref*{210714172405}}.
\end{proof}

\subsection{Borel Summability of Formal Solutions}

In this subsection, we translate \autoref{211026151811} and its method of proof into the language of Borel-Laplace theory, the basics of which are briefly recalled in \autoref{210616130753}.
Namely, it follows directly from our construction that the canonical exact solutions are the Borel resummation of the corresponding formal solutions.
Let us make this statement precise and explicit.
The following theorem is a direct consequence of the proof of \autoref{211026151811}.

\begin{thm}[\textbf{Borel summability of formal solutions}]{211027140644}
Assume all the hypotheses of \autoref{211026151811}.
Then the local formal solution $\hat{f}_\alpha$ is Borel summable in the direction $\theta$ uniformly near $x_0$.
Namely, the canonical local exact solution $f_\alpha^\theta$ is the uniform Borel resummation of $\hat{f}_\alpha$ in the direction $\theta$: for all $\hbar \in S_0$ and uniformly for all $x \in U_0$,
\begin{equation}\label{211027165527}
	f_\alpha^\theta (x, \hbar) 
		= f^\alpha_0 (x) + \cal{S}_\theta \big[ \: \hat{f}_\alpha \: \big] (x, \hbar)
\fullstop
\end{equation}
\end{thm}

A lot of information is packed into \autoref{211027140644}.
Let us unpack it into the following series of explicit statements, all of which are deduced immediately from the proof of \autoref{211026151811}.

\begin{lem}{211027165927}
Assume all the hypotheses of \autoref{211026151811} and let $f_\alpha^\theta$ be the canonical exact solution defined on $U_0 \times S_0$.
\begin{enumerate}
\item The formal Borel transform $\hat{\phi}_\alpha$ of $\hat{f}_\alpha$, given by
\eqntag{
\mbox{}\hspace{-10pt}
	\hat{\phi}_\alpha (x, \xi) =
	\hat{\Borel} [ \, \hat{f}_\alpha \, ] (x, \xi)
		= \sum_{n=0}^\infty \phi_n^\alpha (x) \xi^n
\qtext{where}
	\phi_k^\alpha (x) \coleq \tfrac{1}{k!} f_{k+1}^\alpha (x)
\fullstop{,}
}
is a uniformly convergent power series in $\xi$.
In other words, there is an open disc $\Disc_0 \subset \Complex_\xi$ centred at the origin such that $\hat{\phi}_\alpha$ defines a holomorphic function on $U_0 \times \Disc_0$.
In symbols, $\hat{\phi}_\alpha \in \cal{O} (U_0) \set{\xi}$.
\item In particular, the power series coefficients of the formal solution $\hat{f}_\alpha$ grow at most factorially in $k$: there are real constants $\CC, \MM > 0$ such that
\eqntag{\label{211104172301}
	\big| f_k^\alpha (x) \big|
	\leq \CC \MM^{k} k!
\qqqqqqquad
	(\forall k \geq 0, \forall x \in U_0)
\fullstop
}
\item There is some $\epsilon > 0$ such that the analytic Borel transform $\phi_\alpha^\theta$ of $f^\theta_\alpha$ in the direction $\theta$, given by
\eqntag{
	\phi_\alpha^\theta (x, \xi) = 
	\Borel_\theta [\, f^\theta_\alpha \,] (x, \xi)
		= \frac{1}{2\pi i} \oint\nolimits_\theta f^\theta_\alpha (x, \hbar) e^{\xi / \hbar} \frac{\dd{\hbar}}{\hbar^2}
\fullstop{,}
}
is uniformly convergent for all $(x, \xi) \in U_0 \times \Xi_\theta$ where 
\eqn{
	\Xi_\theta \coleq \set{ \xi \in \Complex_\xi ~\Big|~ \op{dist} (\xi, e^{i\theta} \Real_+) < \epsilon }
\fullstop
}
Furthermore, $\phi_\alpha^\theta$ defines the analytic continuation of the formal Borel transform $\hat{\phi}_\alpha$ along the ray $e^{i\theta} \Real_+ \subset \Complex_\xi$.
In particular, there are no singularities in the Borel plane $\Complex_\xi$ along the ray $e^{i\theta} \Real_+$.
\item The Laplace transform of $\phi_\alpha^\theta$ in the direction $\theta$, given by
\eqn{
	\Laplace_\theta \big[\: \phi_\alpha^\theta \:\big] (x, \hbar)
		= \int\nolimits_{e^{i\theta} \Real_+} e^{-\xi/\hbar} \phi_\alpha^\theta (x, \xi) \dd{\xi}
\fullstop{,}
}
is uniformly convergent for all $(x, \hbar) \in U_0 \times S_0$ and satisfies the following identity:
\begin{equation}\label{211027181734}
	f_\alpha^\theta (x, \hbar)
		= f_0^\alpha (x) + \Laplace_\theta \big[\: \phi_\alpha^\theta \:\big] (x, \hbar)
\fullstop
\end{equation}
\end{enumerate}
\end{lem}


\paragraph{}
The fact that identity \eqref{211027165527} holds uniformly for all $x \in U_0$ means in particular that operations such as differentiation and integration with respect to $x$ can be exchanged with the operation of Borel resummation.
Thus, we have the following corollary.

\begin{cor}{211103142954}
Assume all the hypotheses of \autoref{211026151811} and let $f_\alpha^\theta$ be the canonical exact solution defined on $U_0 \times S_0$.
Then the formal power series on $U_0$ given by the derivative $\del_x \hat{f}_\alpha$ and the integral $\int_{x'_0}^x \hat{f}_\alpha$ from any basepoint $x'_0 \in U_0$ are uniformly Borel summable on $U_0$, and the following identities hold uniformly for all $x \in U_0$:
\eqnstag{
\label{211115112356}
	\del_x f_\alpha^\theta (x, \hbar)
		&= \del_x f_0^\alpha (x) + \cal{S}_\theta \big[ \: \del_x \hat{f}_\alpha \: \big] (x, \hbar)
	\fullstop{,}
\\
\label{211115112403}
	\int\nolimits_{x'_0}^x f_\alpha^\theta (t, \hbar) \dd{t}
		&= \int\nolimits_{x'_0}^x f^\alpha_0 (t) \dd{t}
			+ \cal{S}_\theta 
				\left[ \: \int\nolimits_{x'_0}^x \hat{f}_\alpha \dd{t} \: \right] (x, \hbar)
\fullstop
}
\end{cor}

\paragraph{}
Thanks to the tighter control on the asymptotics of canonical exact solutions on WKB halfstrips, all of the above statements extend uniformly over strictly smaller WKB halfstrips.
Explicitly, we have the following corollary.

\begin{prop}[\textbf{uniform Borel summability on WKB halfstrips}]{211115112318}
Assume all the hypotheses of \autoref{211026151811}.
Let $r > 0$ be such that $W = W_\theta^\alpha (x_0, r)$.
For any $r_0 \in (0, r)$, let $W_0 \coleq W_\theta^\alpha (x_0, r_0)$.
Then the formal solution $\hat{f}_\alpha$ is Borel summable in the direction $\theta$ uniformly for all $x \in W_0$.
Also, the formal power series given by the derivative $\del_x \hat{f}_\alpha$ and the integral $\int_{x'_0}^x \hat{f}_\alpha$ from any basepoint $x'_0 \in U_0$ are uniformly Borel summable on $W_0$, and identities \eqref{211115112356} and \eqref{211115112403} hold uniformly for all $x \in W_0$.
\end{prop}

\paragraph{Explicit recursion for the Borel transform.}
\label{200812193708}
The analytic Borel transform $\phi_\alpha^\theta$ can be given a reasonably explicit presentation as follows.
Define an integral operator $\II$ acting on holomorphic functions $\phi = \phi (x, \xi)$ by the following formula, wherever it makes sense:
\eqntag{\label{200623145936}
	\II \big[ \, \phi \, \big] (x, \xi)
		\coleq \int_0^\xi \phi ( x_t, \xi - t ) \dd{t}
\qqtext{where}
	x_t \coleq \Phi^{-1} \big( \Phi (x) + t \big)
\fullstop{,}
}
and the integration contour is the straight line segment from $0$ to $\xi \in \Complex$.
In particular, for any $x \in U_0$ and any sufficiently small $\xi \in e^{i\theta} \Real_\alpha$, the path $\set{x_t}_{t=0}^\xi$ is a segment of the WKB $(\theta, \alpha)$-ray emanating from $x$.

Recall functions $\tilde{a}, \tilde{b}, \tilde{c}$ defined by the identities \eqref{200529110135}.
Their leading-order parts in $\hbar$ are respectively
\eqntag{\label{211104194016}
	\tilde{a}_0 = \frac{\varepsilon_\alpha}{\sqrt{\DD_0}} a_0
\fullstop{,}
\qqquad
	\tilde{b}_0 = \frac{\varepsilon_\alpha}{\sqrt{\DD_0}} 
	\Big( b_1 + 2a_0 f_1^\alpha + 2 a_1 f_0^\alpha \Big)
\fullstop{,}
\qqquad
	\tilde{c}_0 = f_2^\alpha
\fullstop{,}
}
where last identity was obtained by comparing with \eqref{191207212042}.
Finally, denote their analytic Borel transforms in direction $\theta$ as follows:
\eqn{
\begin{aligned}
	\beta_0 &\coleq \Borel_\theta [ \, \tilde{c} \, ]
\fullstop{,}
\\
	\beta_1 &\coleq \Borel_\theta [ \, \tilde{b} \, ]
\fullstop{,}
\\
	\beta_2 &\coleq \Borel_\theta [ \, \tilde{a} \, ]
\fullstop{;}
\end{aligned}
\qqqqtext{so that}
\begin{aligned}
	{\tilde{a}} &= \tilde{a}_0 + \Laplace_\theta [\, \beta_2 \,]
\fullstop{,}
\\
	{\tilde{b}} &= \tilde{b}_0 + \Laplace_\theta [\, \beta_1 \,]
\fullstop{,}
\\
	{\tilde{c}} &= \tilde{c}_0 + \Laplace_\theta [\, \beta_0 \,]
\fullstop
\end{aligned}
}

\begin{prop}[\textbf{Recursive Formula for the Borel Transform}]{200623141629}
Assume all the hypotheses of \autoref{211026151811} and let $r > 0$ be such that $W = W_\theta^\alpha (x_0; r)$.
For any $r_0 \in (0, r)$ and any $\epsilon \in (0, r - r_0)$, let $W_0 \coleq W_\theta^\alpha (x_0; r_0)$ and $\Xi_\theta \coleq \set{ \xi ~\big|~ \op{dist} (\xi, e^{i\theta} \Real_+) < \epsilon }$.
Then the analytic Borel transform $\phi_\alpha^\theta$ can be expressed more explicitly as follows: uniformly for all $(x, \xi) \in W_0 \times \Xi_\theta$,
\eqntag{\label{200625202831}
	\phi_\alpha^\theta (x, \xi) = f^\alpha_1 (x) + \int_0^\xi \sigma_\alpha^\theta (x, t) \dd{t}
\qqtext{with}
	\sigma_\alpha^\theta (x, \xi) \coleq \sum_{k=0}^\infty \sigma_k (x, \xi)
\fullstop{,}
}
where $\sigma_0 \coleq \tilde{c}_0 = f_2^\alpha$, ~ $\sigma_1 \coleq \II \big[ \beta_0 + \tilde{b}_0 \sigma_0 \big]$, and for $k \geq 2$,
\eqntag{\label{200704101223}
	\sigma_k
		\coleq \II \left[ 
			\tilde{b}_0 \sigma_{k-1}
			+ \beta_1 \ast \sigma_{k-2}
			+ \sum_{\substack{i,j \geq 0 \\ i + j = k-2}}
				\tilde{a}_0 \sigma_i \ast \sigma_j
			+ \sum_{\substack{i,j \geq 0 \\ i + j = k-3}}
				\beta_2 \ast \sigma_i \ast \sigma_j
			\right]
\fullstop
\rlap{\qqquad\qedhere}
}
\end{prop}

\begin{proof}
The proof is straightforward and amounts to transforming some of the main constructions in the proof of the \hyperref[210714172405]{Main Technical Lemma \ref*{210714172405}} using the inverse Liouville transformation $\Phi$ back to the $x$-variable.
Let $\Phi_\alpha^\theta \coleq \varepsilon_\alpha e^{-i \theta} \Phi$, and let $H_+$ be the image of $W$ under $\Phi_\alpha^\theta$.
Since $\Phi^{-1}$ is a local biholomorphism $H_+ \to W$, for a fixed $x \in W$ and every $t \in [0, \xi]$, there is a unique point $x_t \coleq \Phi^{-1} \big( \Phi (x) + t \big) \in W$ such that $\Phi (x) + t = \Phi (x_t)$.
Note in particular that, since $\Phi$ may be multivalued on $W$, the point $x_t$ does not depend on the choice of branch of $\Phi (x)$.
Thus, the integral operator $\II$ from \eqref{200623145936} is defined by using $\Phi$ to transform the integral operator $\II_+$ from \eqref{200223171526} defined in the proof of \autoref{210714172405}.
%
%
Likewise, the sequence $\set{\varphi_k (z, \xi)}_{k=0}^\infty$ defined in the proof of \autoref{210714172405} by the recursive formula \eqref{180824203055} transforms under $\Phi$ to give the sequence $\sigma_k (\Phi^{-1} (z), \xi) \coleq \varphi_k \big(z, e^{-i\theta}\xi\big)$.
\end{proof}

\subsection{Monic Riccati Equations}
\label{211103172432}

In many situations, such as those arising in the context of the exact WKB analysis of second-order ODEs, the coefficients of the Riccati equation satisfy hypothesis (2) in \autoref{211026151811} only after an additional transformation.

\begin{example}{211106162450}
For example, consider the Riccati equation $\hbar \del_x f = f^2 - x$, which is encountered in the WKB analysis of the deformed Airy differential equation.
The coefficients are $a = 1, b = 0, c = -x$ and the leading-order discriminant $\DD_0$ is $4x$.
In this case, a WKB halfstrip $W$ is necessarily an unbounded domain, and the asymptotic condition \eqref{211028195400} reduces to requiring that $c = -x$ is bounded on $W$ by $\sqrt{\DD_0} = 2\sqrt{x}$, which is not the case.
Therefore, \autoref{211026151811} cannot be applied to this Riccati equation directly.

However, this can be remedied by making a change of the unknown variable $f \mapsto g$ given by $f = \sqrt{\DD_0} g$ for $x \in W$.
It transforms the Riccati equation $\hbar \del_x f = f^2 - x$ into
\eqntag{\label{211112191634}
	\hbar \del_x g = 2\sqrt{x} g^2 - \tfrac{1}{2}\hbar x^{-1} g + \tfrac{1}{2} \sqrt{x}
\fullstop
}
Notice that the leading-order discriminant of this Riccati equation remains $\DD_0 = 4x$ and that its are now bounded at infinity by $\sqrt{\DD_0} = 2 \sqrt{x}$.
Therefore \autoref{211026151811} can be applied to \eqref{211112191634}.
\end{example}

\paragraph{}
More generally, this transformation is necessary when dealing with \textit{monic} Riccati equations; i.e., whenever the coefficient $a$ of the Riccati equation is identically $1$.
This is always the case in the exact WKB analysis of second-order ODEs \cite{MY210623112236}.
Spelled out, we have the following version of our results.

\begin{thm}[\textbf{Exact Existence and Uniqueness for Monic Equations}]{211112183728}
\mbox{}\newline
Consider the following monic Riccati equation
\eqntag{\label{211112182040}
	\hbar \del_x f = f^2 + pf + q
}
where $p, q$ are holomorphic functions of $(x, \hbar) \in X \times S$ which admit locally uniform asymptotic expansions $\hat{p}, \hat{q} \in \cal{O} (X) \bbrac{\hbar}$ as $\hbar \to 0$ along $A$.
Assume that the leading-order discriminant $\DD_0 = p_0^2 - 4q_0$ is not identically zero.
Fix a regular point $x_0 \in X$, a square-root branch $\sqrt{\DD_0}$ near $x_0$, and a sign $\alpha \in \set{+, -}$.
In addition, we assume the following hypotheses:
\begin{enumerate}
\item there is a WKB $(\theta, \alpha)$-halfstrip domain $W = W_\theta^\alpha \subset X$ containing $x_0$;
\item the asymptotic expansions of the coefficients $p,q$ are valid with Gevrey bounds as $\hbar \to 0$ along the closed arc $\bar{A}_\theta = [\theta - \tfrac{\pi}{2}, \theta + \tfrac{\pi}{2}]$, with respect to the asymptotic scales $\sqrt{\DD_0}$ and $\DD_0$ respectively, uniformly for all $x \in W$:
\eqnstag{\label{211112183652}
	p &\simeq \hat{p}
\qquad
	\text{as $\hbar \to 0$ along $\bar{A}_\theta$, wrt $\sqrt{\DD_0}$, unif. $\forall x \in W$\fullstop{,}}
\\\label{211112183655}
	q &\simeq \hat{q}
\qquad
	\text{as $\hbar \to 0$ along $\bar{A}_\theta$, wrt $\DD_0$, unif. $\forall x \in W$\fullstop{;}}
}
\item the logarithmic derivative $\del_x \log \DD_0$ is bounded by $\DD_0$ on $W$.
\end{enumerate}
Then all the conclusions of \autoref{211026151811}, as well as \autoref{211027140644}, \autoref{211027165927}, and \autoref{211103142954} hold verbatim.
Also, all the conclusions of \autoref{211118140510} hold verbatim with the only exception that the asymptotic statement \eqref{211113110026} must be replaced with the following asymptotic statement with respect to the asymptotic scale $\sqrt{\DD_0}$:
\eqntag{\label{211118140804}
	f^\theta_\alpha \simeq \hat{f}_\alpha
\qqqquad
	\text{as $\hbar \to 0$ along $\bar{A}_\theta$, wrt $\sqrt{\DD_0}$, unif. $\forall x \in W_0$\fullstop}
}
\end{thm}

\begin{proof}
The change of the unknown variable $f \mapsto g$ given by $f = \sqrt{\DD_0} g$ transforms \eqref{211112182040} into
\eqntag{\label{211112192221}
	\hbar \del_x g = \sqrt{\DD_0} g^2 + \big( p - \hbar \del_x \log \sqrt{\DD_0} \big) g + \tfrac{1}{\sqrt{\DD_0}} q
\fullstop
}
This transformation is well-defined for all $x \in W$ because $W$ necessarily supports the univalued square-root branch $\sqrt{\DD_0}$.
Notice that the leading-order discriminant of this Riccati equation remains $\DD_0$.
Note that $\del_x \log \sqrt{\DD_0}$ is bounded by $\sqrt{\DD_0}$ if and only if $\del_x \log \DD_0$ is bounded by $\DD_0$, as provided by hypothesis (3).
Now it is obvious that the hypotheses of \autoref{211112183728} imply that the Riccati equation \eqref{211112192221} satisfies all the hypotheses of \autoref{211026151811}.
It yields the canonical local exact solution $g_\alpha^\theta$ near $x_0$, and therefore the canonical local exact solution $f_\alpha^\theta = \sqrt{\DD_0} g_\alpha^\theta$ near $x_0$.
\end{proof}

Of course, a general Riccati equation \eqref{211118123850} can always be put into the monic form \eqref{211112183728} via the change of the unknown variable $f \mapsto g = af$ which yields
\eqntag{
	\hbar \del_x g = g^2 + (b + \hbar \del_x \log a) g + ac
\fullstop
}
If $a$ is nowhere-vanishing, then this transformation makes sense globally on $X$.

Likewise, \autoref{211102175415} is also true for the monic Riccati equation \eqref{211112183728}, though with slightly simplified hypotheses as follows.

\begin{cor}[\textbf{extension to larger domains}]{211118133610}
Consider the Riccati equation \eqref{211112182040} with $\DD_0 \not\equiv 0$.
Fix a sign $\alpha \in \set{+, -}$ and let $U \subset X$ be a domain free of turning points that supports a univalued square-root branch $\sqrt{\DD_0}$.
In addition, assume that hypotheses (1-3) in \autoref{211112183728} are satisfied for every point $x_0 \in U$.
Then the conclusions of \autoref{211102175415} hold verbatim.
\end{cor}

\paragraph{}
However, \autoref{211115112318} and \autoref{200623141629} are no longer true for the canonical exact solutions of the monic Riccati equation \eqref{211112183728}.
In this case, one needs either to factorise $\sqrt{\DD_0}$ out of $f_\alpha^\theta$ and apply these propositions to the regularised Riccati equation \eqref{211112192221}, or identify and remove the `principal part' of $f_\alpha^\theta$ in the limit along the WKB rays.
The latter procedure can be explicitly formalised when the WKB rays limit to a pole of $\DD_0$; this will be explained in detail elsewhere.

\begin{rem}{211121133355}
The same trick as above can help us tackle more general situations as follows.
If $\chi = \chi (x)$ is any holomorphic function, then the change of the unknown variable $f \mapsto g$ given by $f = \chi g$ transforms the Riccati equation \eqref{211118123850} into
\eqntag{\label{200711153318}
	\hbar \del_x g = a' g^2 + b' g + c'
\fullstop{,}
}
where
\eqntag{\label{200711153341}
	a' \coleq \chi a
\qquad
	b' \coleq b - \hbar \del_x \log \chi
\qquad
	c' \coleq \chi^{-1} c
\fullstop
}
Note that since $\chi$ is independent of $\hbar$, the leading-order discriminant remains unchanged: $\DD'_0 = (b'_0)^2 - 4a'_0 c'_0 = b_0^2 - 4a_0c_0 = \DD_0$.
In summary, we have the following proposition which in particular recovers \autoref{211026151811} by taking $\chi = 1$ and \autoref{211112183728} by taking $\chi = \sqrt{\DD_0}$.
\end{rem}

\begin{prop}{211106163055}
Assume all the hypotheses of \autoref{211026151811}, except hypothesis (2) is replaced with the following:
\begin{enumerate}
\item [\textup{(2${}^\prime$)}]
There is a nonvanishing holomorphic function $\chi = \chi (x)$ on $W$ such that $a',b',c'$ defined by \eqref{200711153341} admit Gevrey asymptotics as $\hbar \to 0$ along the closed arc $\bar{A}_\theta = [\theta - \tfrac{\pi}{2}, \theta + \tfrac{\pi}{2}]$, with respect to the asymptotic scale $\sqrt{\DD_0}$, uniformly for all $x \in W$:
\end{enumerate}
\begin{equation}
\label{211106163648}
	a' \simeq \hat{a}',\quad
	b' \simeq \hat{b}',\quad
	c' \simeq \hat{c}'
\qquad
	\text{as $\hbar \to 0$ along $\bar{A}_\theta$, wrt $\chi$, unif. $\forall x \in W$\fullstop}
\end{equation}

Then all the conclusions of \autoref{211026151811}, as well as \autoref{211027140644}, \autoref{211027165927}, and \autoref{211103142954} hold verbatim.
Also, all the conclusions of \autoref{211118140510} hold verbatim with the only exception that the asymptotic statement \eqref{211113110026} must be replaced with the following asymptotic statement with respect to the asymptotic scale $\chi$:
\eqntag{\label{211121134253}
	f^\theta_\alpha \simeq \hat{f}_\alpha
\qqqquad
	\text{as $\hbar \to 0$ along $\bar{A}_\theta$, wrt $\chi$, unif. $\forall x \in W_0$\fullstop}
}
\end{prop}

\subsection{Exact Solutions in Wider Sectors}

The last general result we prove is about extending canonical exact solutions to sectors with wider openings.
However, we not address here the question of extending canonical exact solutions to radially larger sectorial domains in $\Complex_\hbar$ or discussing the relationship between unequal canonical exact solutions for different values of $\theta$.
These questions will be examined in detail elsewhere.

\paragraph{}
Thus, suppose that $\pi \leq |A| \leq 2\pi$.
Let $\Theta \coleq [\theta_-, \theta_+]$ be the closed arc such that $A = (\theta_- - \tfrac{\pi}{2}, \theta_+ + \tfrac{\pi}{2})$; i.e., $\theta_\pm \coleq \vartheta_\pm \mp \tfrac{\pi}{2}$.
See \autoref{210618101426}.
This arc $\Theta$ is sometimes called the arc of \textit{copolar directions} of $A$.
For every $\theta \in \Theta$, let $A_\theta \coleq (\theta -\tfrac{\pi}{2}, \theta +\tfrac{\pi}{2}) \subset A$ be the semicircular arc bisected by $\theta$.
See \autoref{210618105212}.
Note that $A = \Cup_{\theta \in \Theta} A_\theta$.

\begin{figure}[t]
\centering
\begin{subfigure}[b]{0.47\textwidth}
    \centering
    \includegraphics{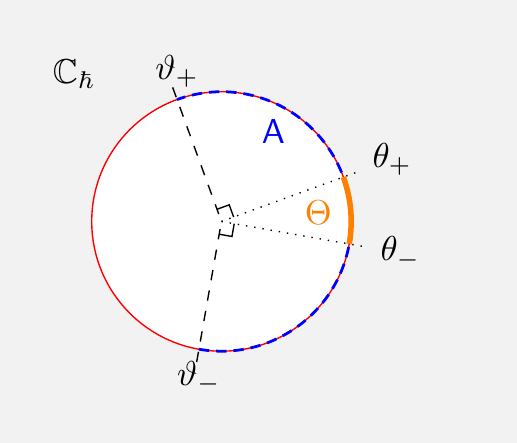}
    \caption{The arc $\Theta$ of copolar directions of $A$.}
    \label{210618101426}
\end{subfigure}
\quad
\begin{subfigure}[b]{0.47\textwidth}
    \centering
    \includegraphics{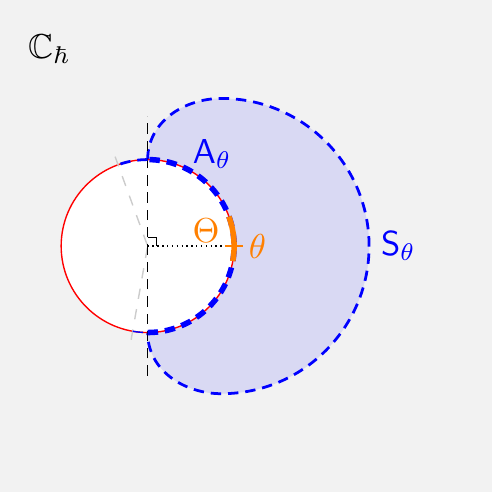}
    \caption{A Borel disc $S_\theta$ with opening $A_\theta$.}
    \label{210618105212}
\end{subfigure}
\caption{}
\end{figure}

\begin{prop}{211103152513}
Consider the Riccati equation \eqref{211118123850} whose coefficients $a,b,c$ are holomorphic functions of $(x,\hbar) \in X \times S$ admitting locally uniform asymptotic expansions $\hat{a}, \hat{b}, \hat{c} \in \cal{O} (X) \bbrac{\hbar}$ as $\hbar \to 0$ along $A$.
Assume that the leading-order discriminant $\DD_0 = b_0^2 - 4 a_0 c_0$ is not identically zero.
Fix a regular point $x_0 \in X$, a square-root branch $\sqrt{\DD_0}$ near $x_0$, and a sign $\alpha \in \set{+, -}$.
Fix a regular point $x_0 \in X$, a square-root branch $\sqrt{\DD_0}$ near $x_0$, and a sign $\alpha \in \set{+, -}$.
In addition, assume that hypotheses (1) and (2) in \autoref{211026151811} are satisfied for every $\theta \in \Theta$.

Then the Riccati equation \eqref{211118123850} has a canonical local exact solution $f^\Theta_\alpha$ near $x_0$ which is asymptotic to the formal solution $\hat{f}_\alpha$ as $\hbar \to 0$ in every direction $\theta \in \Theta$.
Namely, there is a neighbourhood $U_0 \subset X$ of $x_0$ and a sectorial domain $S_0 \subset S$ with the same opening $A$ such that the Riccati equation \eqref{211118123850} has a unique holomorphic solution $f^\Theta_\alpha$ on $U_0 \times S_0$ which is Gevrey asymptotic to $\hat{f}_\alpha$ as $\hbar \to 0$ along the closed arc $\bar{A}$ uniformly for all $x \in U_0$:\begin{equation}
\label{211118134012}
	f^\Theta_\alpha \simeq \hat{f}_\alpha
\qqqquad
	\text{as $\hbar \to 0$ along $\bar{A}$, unif. $\forall x \in U_0$\fullstop}
\end{equation}
\end{prop}

\begin{proof}
By \autoref{211027140644} (or more specifically by part (4) of \autoref{211027165927}), for every $\theta \in \Theta$, the canonical exact solution in the direction $\theta$ exists and can be written as
\eqn{
		f_\alpha^\theta (x, \hbar)
		= f_0^\alpha (x) + \Laplace_\theta \big[\: \phi_\alpha^\theta \:\big] (x, \hbar)
		= f_0^\alpha (x) + \int\nolimits_{e^{i\theta} \Real_+} e^{-\xi/\hbar} \phi_\alpha^\theta (x, \xi) \dd{\xi}
\fullstop{,}
}
where $\phi_\alpha^\theta$ is the analytic Borel transform of $f_\alpha^\theta$ in the direction $\theta$.
The explicit formula from \autoref{200623141629} reveals that these analytic Borel transforms $\phi_\alpha^\theta$ for each $\theta \in \Theta$ together define a holomorphic function $\phi_\alpha^\Theta$ on an $\epsilon$-neighbourhood $\Xi_\Theta = \set{\xi ~\big|~ \op{dist} (\xi, \Sigma_\Theta) < \epsilon}$ of the sector $\Sigma_\Theta \coleq \set{\xi ~\big|~ \arg (\xi) \in \Theta}$ for a sufficiently small $\epsilon > 0$.
Furthermore, $\phi_\alpha^\Theta$ has at most exponential growth at infinity in $\Xi_\Theta$, which means that Cauchy-Goursat's theorem yields the identity $\mathfrak{L}_{\theta_+} \big[ \phi_\alpha^{\theta_+} \big] = \mathfrak{L}_{\theta_+} \big[ \phi_\alpha^\Theta \big] = \mathfrak{L}_{\theta_-} \big[ \phi_\alpha^\Theta \big] = \mathfrak{L}_{\theta_-} \big[ \phi_\alpha^{\theta_-} \big]$.
\end{proof}

\section{Examples and Applications}
\label{211107122319}

The somewhat obscure technical hypotheses in \autoref{211026151811} and \autoref{211112183728} can be made more transparent in a number of special situations which we describe in this section.
We also present the simplest explicit example in \autoref{211027171318} where we construct a pair of canonical exact solutions by following all the steps in the proof of the main theorem.
Finally, in \autoref{211119204333}, we give a very important application of our result in the context of the exact WKB analysis of Schrödinger equations.

\subsection{Equations with Mildly Deformed Coefficients}
\label{211121154157}

\subsubsection{Undeformed Coefficients}
\label{211119201901}

The simplest yet ubiquitous situation is when the coefficients of the Riccati equation \eqref{211118123850} are independent of $\hbar$, in which case the asymptotic hypotheses in \autoref{211026151811} dramatically simplify.

\paragraph{}
Thus, let us consider both the general Riccati equation \eqref{211118123850} as well as a monic Riccati equation \eqref{211112182040} on $X \times \Complex_\hbar$ with coefficients $a,b,c,p,q \in \cal{O} (X)$.
Their leading-order discriminants $\DD_0$ are simply $b^2 - 4ac$ and $p^2 - 4q$.
The sectorial domain $S$ can be taken to be any halfplane bisected by some direction in $\Complex_\hbar$.

\paragraph{Polynomial coefficients.}
The simplest case is when the coefficients are polynomials in $x$; i.e., $a,b,c,p,q \in \Complex [x]$.
Then $X = \Complex_x$ and there are only finitely many turning points and singular WKB trajectories.
All singular WKB trajectories either connect a turning point to infinity, or two turning points together.
All WKB trajectories can be easily plotted using a computer, or even by hand in simple examples.
See \autoref{211027171318} where we examine the simplest example in detail.

The biggest advantage of this simple situation is that in order to check hypothesis (1) in either \autoref{211026151811} or \autoref{211112183728} that a given regular point $x_0$ is contained in a WKB halfstrip, it is sufficient to only examine the WKB ray emanating from $x_0$ and check that it does not hit a turning point.
If so, this WKB ray is necessarily either a closed WKB trajectory (i.e., a simple closed curve in the complement of the turning points) or it escapes to infinity.
In either case, there is necessarily a WKB halfstrip $W$ containing $x_0$.
For instance, we can take a small disc $\Disc$ centred at $x_0$ which is compactly contained in the complement of all singular WKB rays (for the same phase and sign), and let $W$ be the union of all WKB rays emanating from $\Disc$.
This disc $\Disc$ should be chosen small enough that $W$ is compactly contained in the complement of the turning points; it ensures in particular that hypothesis (3) in \autoref{211112183728} is satisfied.

If the WKB ray emanating from $x_0$ is a closed WKB trajectory, then hypothesis (2) in both theorems is automatic.
On the other hand, if this WKB ray escapes to infinity, then hypothesis (2) in \autoref{211026151811} is equivalent to saying that the polynomials $a,b,c$ are all bounded at infinity by $\sqrt{\DD_0}$, and hypothesis (2) in \autoref{211112183728} is equivalent to saying that $p$ is bounded at infinity by $\sqrt{\DD_0}$ and $q$ is bounded at infinity by $\DD_0$.
As these are all polynomials in $x$, hypothesis (2) in both theorems therefore boils down to a condition on their degrees.
In summary, we have the following.

\begin{prop}{211121120847}
Consider either the general Riccati equation \eqref{211118123850} or a monic Riccati equation \eqref{211112182040} on $\Complex_x \times \Complex_\hbar$ with polynomial $\hbar$-independent coefficients $a,b,c$ or $p,q \in \Complex [x]$.
Fix a regular point $x_0 \in \Complex_x$ and a square-root branch $\sqrt{\DD_0}$ near $x_0$.
Fix a sign $\alpha \in \set{+, -}$, a phase $\theta \in \Real$, and let $S \coleq \set{\hbar ~\big|~ \Re (e^{-i\theta} \hbar) > 0}$.
Assume that
\textit{
\begin{itemise}
\item[\textup{(1)}] the WKB $(\theta, \alpha)$-ray $\Gamma_\theta^\alpha = \Gamma_\theta^\alpha (x_0)$ emanating from $x_0$ does not hit a turning point; assume in addition that $a_0$ is nonvanishing on $\Gamma_\theta^\alpha$ if $\alpha = -$.
\end{itemise}
}
If $\Gamma_\theta^\alpha$ is a closed trajectory, then all the hypotheses of \autoref{211026151811} or \autoref{211112183728} are satisfied, and therefore their conclusions hold verbatim.
If, on the other hand, $\Gamma_\theta^\alpha$ escapes to infinity, then in addition we assume that either
\textit{
\begin{itemise}
\item[\textup{(2)}] $\deg (a), \deg (b), \deg (c) \leq \tfrac{1}{2} \deg (\DD_0)$ \qquad or\\
$\deg (p) \leq \tfrac{1}{2} \deg (\DD_0)$ \quad and \quad $\deg (q) \leq \deg (\DD_0)$.
\end{itemise}
}
Then all the hypotheses of \autoref{211026151811} or \autoref{211112183728} are satisfied.
\end{prop}

\paragraph{Rational coefficients.}
More generally, the coefficients $a,b,c$ and $p,q$ may be arbitrary rational functions of $x$.
Then $X = \Complex_x \setminus \set{\text{poles}}$ and again there are only finitely many turning points and singular WKB trajectories, all of which can be easily plotted using a computer.
All singular WKB trajectories end either on a turning point or a simple pole of $\DD_0$.

As in the polynomial scenario, checking hypothesis (1) in both theorems boils down to examining a single WKB ray emanating from a regular point $x_0$.
If this ray does not hit a turning point, it must be either a closed trajectory or limit to one of the poles $x_\infty$ of $\DD_0$ including possibly the one at infinity.
This WKB ray is infinite if and only if the pole order of $\DD_0$ at $x_\infty$ is $2$ or greater.
Note that if $\op{ord} (\DD_0) \geq 2$ at $x_\infty$, then hypothesis (3) in \autoref{211112183728} is automatic.
Finally, similar to the polynomial scenario, the asymptotic hypothesis (2) in both theorems boils down to a boundedness condition near the pole $x_\infty$ on the coefficients by an appropriate power of $\DD_0$.
In summary, we have the following.

\enlargethispage{1cm}
\begin{prop}{211121122547}
Consider either the general Riccati equation \eqref{211118123850} or a monic Riccati equation \eqref{211112182040} on $\Complex_x \times \Complex_\hbar$ with rational $\hbar$-independent coefficients $a,b,c$ or $p,q \in \Complex (x)$.
Fix a sign $\alpha \in \set{+, -}$, a phase $\theta \in \Real$, and let $S \coleq \set{\hbar ~\big|~ \Re (e^{-i\theta} \hbar) > 0}$.
Fix a regular point $x_0 \in X$ and a square-root branch $\sqrt{\DD_0}$ near $x_0$.
Assume that
\textit{
\begin{itemise}
\item[\textup{(1)}] the WKB $(\theta, \alpha)$-ray $\Gamma_\theta^\alpha = \Gamma_\theta^\alpha (x_0)$ emanating from $x_0$ does not hit a turning point; assume in addition that $a_0$ is nonvanishing on $\Gamma_\theta^\alpha$ if $\alpha = -$.
\end{itemise}
}
If $\Gamma_\theta^\alpha$ is a closed trajectory, then all the hypothesis of \autoref{211026151811} or \autoref{211112183728} are satisfied, and therefore their conclusions hold verbatim.
If, on the other hand, $\Gamma_\theta^\alpha$ tends a pole $x_\infty \in \Complex_x \cup \set{\infty}$ of $\DD_0$ of order $2$ or higher, then we also assume that either
\textit{
\begin{itemise}
\item[\textup{(2)}] $\op{ord} (a), \op{ord} (b), \op{ord} (c) \leq \tfrac{1}{2} \op{ord} (\DD_0)$ at $x_\infty$ \qquad or\\
$\op{ord} (p) \leq \tfrac{1}{2} \op{ord} (\DD_0)$ and $\op{ord} (q) \leq \op{ord} (\DD_0)$ at $x_\infty$.
\end{itemise}
}
Then all the hypotheses of \autoref{211026151811} or \autoref{211112183728} are satisfied.
\end{prop}

\paragraph{General meromorphic coefficients.}
When $a,b,c$ or $p,q$ are more general not necessarily rational meromorphic functions, the WKB geometry is far more complicated to describe in general.
However, if the WKB ray $\Gamma^\alpha_\theta (x_0)$ is closed or limits to a second- or higher-order pole of $\DD_0$, then a general but simplified version of both theorems can be stated as follows.

\begin{prop}{211121155120}
Consider either the general Riccati equation \eqref{211118123850} or a monic Riccati equation \eqref{211112182040} on $X \times \Complex_\hbar$ with  coefficients $a,b,c$ or $p,q \in \cal{O} (X)$.
Fix a sign $\alpha \in \set{+, -}$, a phase $\theta \in \Real$, and let $S \coleq \set{\hbar ~\big|~ \Re (e^{-i\theta} \hbar) > 0}$.
Fix a regular point $x_0 \in X$ and a square-root branch $\sqrt{\DD_0}$ near $x_0$.
Then we make the following hypotheses:
\begin{enumerate}
\item the WKB $(\theta, \alpha)$-ray $\Gamma_\theta^\alpha = \Gamma_\theta^\alpha (x_0)$ emanating from $x_0$ does not hit a turning point but instead limits to a pole $x_\infty \in \del X$ of $\DD_0$ of order $2$ or greater.
If $\alpha = -$, we also assume that $a_0$ is nonvanishing on $\Gamma_\theta^\alpha$ and $x_\infty$ is not an accumulation point of zeroes of $a_0$.
\item $\op{ord} (a), \op{ord} (b), \op{ord} (c) \leq \tfrac{1}{2} \op{ord} (\DD_0)$ at $x_\infty$ \qquad or\\
$\op{ord} (p) \leq \tfrac{1}{2} \op{ord} (\DD_0)$ and $\op{ord} (q) \leq \op{ord} (\DD_0)$ at $x_\infty$.
\end{enumerate}
Then all the hypotheses of \autoref{211026151811} or \autoref{211112183728} are satisfied.
\end{prop}

\subsubsection{Polynomially Deformed Coefficients}

The next simplest situation is when the equations coefficients depend on $\hbar$ at most polynomially.

\paragraph{}
Thus, let us consider both the general Riccati equation \eqref{211118123850} as well as a monic Riccati equation \eqref{211112182040} on $X \times \Complex_\hbar$ where functions $a,b,c$ in \eqref{211118123850} or $p,q$ in \eqref{211112182040} are at most polynomials in $\hbar$ with holomorphic coefficients.
The sectorial domain $S$ can still be taken to be a halfplane in $\Complex_\hbar$.

The WKB geometry is fully determined by the leading-order part of the equation, so all the same considerations apply as explained in \autoref{211119201901}.
The advantage of being given the coefficients $a,b,c$ or $p,q$ as polynomials in $\hbar$ rather than more general functions of $\hbar$ is that the assumptions on the $\hbar$-asymptotics reduce to simple bounds on the $\hbar$-polynomial coefficients of $a,b,c$ or $p,q$ of the kind we have already seen.
In summary, we have the following.

\begin{prop}{211121161720}
Consider either the general Riccati equation \eqref{211118123850} or a monic Riccati equation \eqref{211112182040} on $X \times \Complex_\hbar$ with  coefficients $a,b,c$ or $p,q \in \cal{O} (X) [\hbar]$.
Fix a sign $\alpha \in \set{+, -}$, a phase $\theta \in \Real$, and let $S \coleq \set{\hbar ~\big|~ \Re (e^{-i\theta} \hbar) > 0}$.
Fix a regular point $x_0 \in X$ and a square-root branch $\sqrt{\DD_0}$ near $x_0$.
Assume hypothesis (1) from \autoref{211121155120} and instead of hypothesis (2), assume that
\begin{enumerate}
\setcounter{enumi}{1}
\item $\op{ord} (a_k), \op{ord} (b_k), \op{ord} (c_k) \leq \tfrac{1}{2} \op{ord} (\DD_0)$ at $x_\infty$ for every $k$ \qquad or\\
$\op{ord} (p_k) \leq \tfrac{1}{2} \op{ord} (\DD_0)$ and $\op{ord} (q_k) \leq \op{ord} (\DD_0)$ at $x_\infty$ for every $k$.
\end{enumerate}
Then all the hypotheses of \autoref{211026151811} or \autoref{211112183728} are satisfied.
\end{prop}

\subsection{The Simplest Explicit Example: Deformed Airy}
\label{211027171318}

In this subsection, we illustrate the main constructions in this paper in the following explicit example.
Consider the following Riccati equation on the domain $\Complex_x \times \Complex_\hbar$:
\eqntag{\label{200701111940}
	\hbar \del_x f = f^2 - x
\fullstop
}
Thus, in this example, $a = 1, b = 0, c = -x$ and $X = \Complex_x$.
Let us fix $\theta = 0$, so we will search for canonical exact solutions of \eqref{200701111940} with prescribed asymptotics as $\hbar \to 0$ along the positive real axis $\Real_+ \subset \Complex_\hbar$.
Then the sectorial domain $S$ can be taken as the complement of the negative real axis $\Real_- \subset \Complex_\hbar$.

This Riccati equation arises in the WKB analysis of the Airy differential equation $\hbar^2 \del^2_x \psi (x, \hbar) = x \psi (x, \hbar)$ upon considering the WKB ansatz $\psi = \exp \big( - \int f / \hbar \big)$, see \cite{MY210623112236} for more details.
For this Riccati equation, it is known that exact solutions exist; see, e.g., \cite[\S2.2]{MR2182990}.
There, instead of solving the Riccati equation directly, the Borel-Laplace method is applied to the Airy differential equation.
Here, we take a different approach by solving the Riccati equation directly.
We have included this example because it is the simplest and most explicit standard example that nicely illustrates most constructions encountered in our paper.

\paragraph{Leading-order analysis.}
Following \autoref{200601224547}, the leading-order equation \eqref{200521115322} for the Riccati equation \eqref{200701111940} is simply $f_0^2 - x = 0$.
The leading-order discriminant given by formula \eqref{200219161938} is $\DD_0 (x) = 4x$.
There is a single turning point at $x = 0$.
Let $\sqrt{x}$ be the principal square-root branch (i.e., positive on the positive real axis) in the complement of a branch cut along, say, the negative real axis.
Label the two leading-order solutions as 
\eqntag{
	f^\pm_0 (x) \coleq \pm \sqrt{x}
\qqtext{so that}
	\sqrt{\DD_0} = 2 \sqrt{x}
\fullstop
}
The leading-order solutions $f^\pm_0$ are holomorphic on any simply connected domain $U \subset \Complex_x^\ast$.
However, $f^\pm_0$ are unbounded if $U$ is unbounded: the coefficient $c = -x$ of \eqref{200701111940} is not bounded by $\sqrt{\DD} = 2 \sqrt{x}$, so not all the hypotheses of \autoref{200614161900} are satisfied.
This will be rectified later by regularising the coefficients following \autoref{211103172432}.

\paragraph{Formal perturbation theory.}
Now we study the formal aspects of this Riccati equation following \autoref{200521121035}.
By \autoref{200118111737}, the Riccati equation \eqref{200701111940} has a pair of formal solutions $\hat{f}_\pm \in \cal{O} (U) \bbrac{\hbar}$ with leading-order terms $f_0^\pm$.
Their coefficients $f_k^\pm$ for $k \geq 1$ are given by the recursive formula \eqref{191207215212}, which in this example reduces to
\eqntag{
	f_k^\pm = \pm \tfrac{1}{2 \sqrt{x}} \del_x f_{k-1}^\pm \mp \tfrac{1}{2 \sqrt{x}} \sum_{i+j = k}^{i,j\neq k} f_i^\pm f_j^\pm
\fullstop
}
The first few coefficients are 
\eqntag{
	f_0^\pm = \pm \sqrt{x},
\qquad
	f_1^\pm = +\tfrac{1}{4} x^{-1},
\qquad
	f_2^\pm = \mp\tfrac{5}{32} x^{\!\!\nicefrac{-5}{2}},
\qquad
	f_3^\pm = + \tfrac{15}{64} x^{-4},
\quad
	\ldots
\fullstop
}
In fact, if we set $d_0^\pm = \pm 1$, it is easy to show by induction that for all $k \geq 1$,
\eqntag{
	f_k^\pm (x) = d_k^\pm x^{\!\nicefrac{-3k}{2}} \sqrt{x}
\qtext{where}
	d_k^\pm \coleq \frac{1}{2} \left( (2 - 3k/2) d_{k-1}^\pm - \sum_{i + j = k}^{i,j \neq k} d_i^\pm d_j^\pm \right)
\fullstop
}
Note that all $d^\pm_k$ are rational numbers.

\paragraph{The Liouville transformation and WKB trajectories.}
\label{200814090913}
Following \autoref{211028202110}, let us describe the geometry of WKB trajectories on $\Complex_x$ emerging from this Riccati equation.
For any basepoint $x_0 \in \Complex_x$, the Liouville transformation is given, on the complement of the branch cut, by the simple formula 
\eqntag{
	z = \Phi (x) 
		= \int_{x_0}^x 2\sqrt{t} \dd{t} 
		= \tfrac{4}{3} (x^{\!\nicefrac{3}{2}}_\phantomindex - x_0^{\!\nicefrac{3}{2}})
\fullstop
}
It follows that, for example, the WKB $(0,+)$-ray emanating from any point $x_0$ with $\arg (x_0) \neq \pm 3\pi /2$ is complete.
If, on the other hand, $\arg (x_0) = \pm 3\pi/2$, then the WKB $(0,+)$-ray emanating from $x_0$ hits the turning point in finite time.
Likewise, the WKB $(0,-)$-ray of every point $x_0$ with $\arg (x_0) \neq 0$ is complete.
See \autoref{200814122327}.
\begin{figure}[t]
\centering
\includegraphics[scale=1.25]{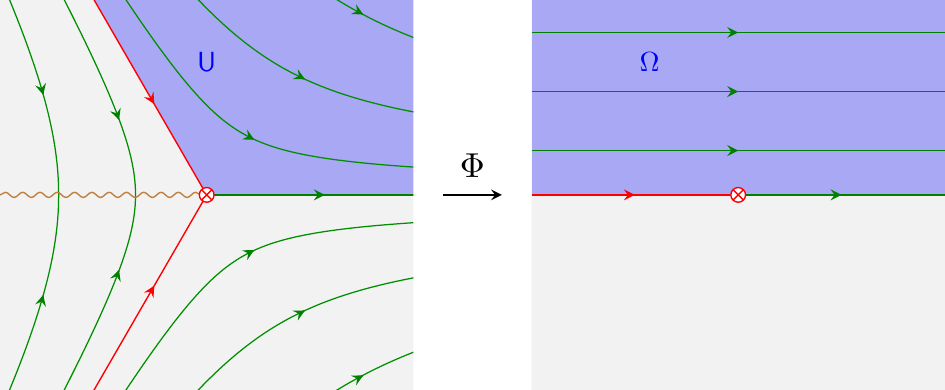}
\caption{
Pictured are the complex planes $\Complex_x$ (left) and $\Complex_z$ (right) with $\Phi$ being the Liouville transformation with basepoint $x_0 = 0$.
In $\Complex_x$, there is a turning point at the origin, indicated by a red circled cross.
A few complete WKB trajectories on $\Complex_x$ are drawn in green, with arrows indicating the orientation with respect to the chosen square root branch $\sqrt{\DD_0} = 2 \sqrt{x}$, for which the branch cut is taken along the negative real axis.
There are two special trajectories, indicated in red, which are not complete: they flow into the turning point in finite time.
The domain $U$ from \eqref{200814122922} is shaded in blue.
}
\label{200814122327}
\end{figure}
We focus our attention now on the domain 
\eqntag{\label{200814122922}
	U \coleq \set{ x ~\big|~ 0 < \arg (x) < +3\pi/2}
\fullstop
}
Its image under the Liouville transformation $\Phi_0 (x) \coleq \tfrac{4}{3} x^{\nicefrac{3}{2}}$ is the upper halfplane $H = \set{z ~\big|~ \Im (z) > 0}$.
Clearly, the domain $U$ is swept out by complete WKB trajectories and every point is contained in a WKB strip.
Thus, for example, let us take $x_0 \coleq \tfrac{3}{4} e^{i\pi/3}$, so $\Phi_0 (x_0) = i$.
Let $U_0$ be the preimage under $\Phi_0$ of the horizontal strip $H_0 \coleq \set{z ~\big|~ \tfrac{1}{2} <  \Im (z) < \tfrac{3}{2}}$.
The Liouville transformation based at $x_0$ is simply $\Phi = \Phi_0 - i$, and the image of $U_0$ is the WKB strip $\set{z ~\big|~ \op{dist} (z, \Real) < \tfrac{1}{2}}$.
However, in this example, it is more convenient to work with the Liouville transformation $\Phi_0$.

\paragraph{Regularising the coefficients.}
\label{200814065244}
The first observation is that the Riccati equation \eqref{200701111940} does not satisfy hypothesis (3) of \autoref{211026151811}.
This is because the rescaled coefficients $a/\sqrt{\DD_0}$, $b/\sqrt{\DD_0}$, $c/\sqrt{\DD_0}$ are respectively $\tfrac{1}{2} x^{\!\nicefrac{-1}{2}}, 0, - \tfrac{1}{2} x^{\!\nicefrac{+1}{2}}$ which are unbounded on $U$.
This unboundedness is caused by two separate problems: one is that $x^{\!\nicefrac{+1}{2}}$ is unbounded at infinity, and the other is that $x^{\!\nicefrac{-1}{2}}$ is unbounded near the turning point at the origin.
The latter problem is remedied by restricting to $U_0$.
In order to remedy the first problem and proceed according to our method, it is necessary to regularise the coefficients of this Riccati equation as in \autoref{211103172432}.

If we make a change of the unknown variable $f \mapsto g = x^{\!\!\nicefrac{-1}{2}} f$ over $U$, then the Riccati equation \eqref{200701111940} gets transformed into
\eqntag{\label{200709152427}
	\hbar \del_x g = \sqrt{x} g^2 - \tfrac{\hbar}{2x} g - \sqrt{x}
\fullstop
}
This is equation \eqref{200711153318} from \autoref{211103172432} with \eqn{
	a' = \sqrt{x},
\qquad
	b' = \tfrac{1}{2} \hbar x^{-1},
\qquad
	c' = - \sqrt{x},
\qquad
	\chi = \sqrt{x} = \tfrac{1}{2} \sqrt{\DD}
\fullstop
}
The Riccati equation \eqref{200709152427} now satisfies hypothesis (3') from \autoref{211106163055} with regularising factor $\chi = \sqrt{x}$.

The coefficients of the formal solutions $\hat{g}_\pm$ of \eqref{200709152427} are given by $g^\pm_k (x) = d_k^\pm x^{\!\!\nicefrac{-3k}{2}}$.
Explicitly, the first few coefficients are
\eqn{
	g_0^\pm = \pm 1,
\qquad
	g_1^\pm = +\tfrac{1}{4} x^{\!\!\nicefrac{-3}{2}},
\qquad
	g_2^\pm = \mp\tfrac{5}{32} x^{-3},
\qquad
	g_3^\pm = + \tfrac{15}{64} x^{\!\!\nicefrac{-9}{2}},
\quad
	\ldots
\fullstop
}
We now follow the step-by-step procedure in the proof of \autoref{211026151811} in \autoref{211107195800}.

\paragraph{Step 0: Preliminary transformation.}
We begin by performing the preliminary transformations (one for each of $\pm$) of the unknown variable $g \mapsto \tilde{g}$ given by
\eqntag{\label{200814082235}
	g 
		= g_0^\pm + \hbar \big( g_1^\pm + \tilde{g} \big)
		= d_0^\pm + \hbar d_1^\pm x^{\!\!\nicefrac{-3}{2}} + \hbar \tilde{g}
		= \pm 1 + \tfrac{1}{4} \hbar x^{\!\!\nicefrac{-3}{2}} + \hbar \tilde{g}
\fullstop
}
They transform the regularised Riccati equation \eqref{200709152427} into a pair of Riccati equations
\eqntag{\label{200711161709}
	\tfrac{\hbar}{2\sqrt{x}} \del_x \tilde{g} - \tilde{g}
		= \hbar \big( \tfrac{1}{2} \tilde{g}^2 - d_2^\pm x^{-3} \big)
		= \hbar \big( \tfrac{1}{2} \tilde{g}^2 \pm \tfrac{5}{32} x^{-3} \big)
\fullstop
}
This is equation \eqref{200526092203} from \autoref{211107195800} with
\eqn{
	\tilde{a} = \tfrac{1}{2},
\qquad
	\tilde{b} = 0,
\qquad
	\tilde{c} = \pm \tfrac{5}{32} x^{-3}
\fullstop
}
Applying the Liouville transformation $\Phi_0$ to the Riccati equations \eqref{200711161709}, we get
\eqntag{\label{200711161928}
	\hbar \del_z \FF - \FF
		= \hbar \Big( \tfrac{1}{2} \FF^2 - d_2^\pm \big( 3 z /4 \big)^{-2} \Big)
		= \hbar \Big( \tfrac{1}{2} \FF^2 \pm \tfrac{5}{18} z^{-2} \Big)
\fullstop
}
The unknown variables $\tilde{g}$ and $\FF$ are related by $\tilde{g} (x, \hbar) = \FF \big( \tfrac{4}{3} x^{\!\nicefrac{4}{3}}, \hbar \big)$.
Equation \eqref{200711161928} is equation \eqref{210714172412} from \autoref{211107195800} with 
\eqn{
	\AA_0 = \pm \tfrac{5}{18} z^{-2},
\qquad
	\AA_1 = 0,
\qquad
	\AA_2 = \tfrac{1}{2}
\fullstop
}

\paragraph{Step 1: The analytic Borel transform.}
Since $\AA_i$ are independent of $\hbar$, it follows that their Borel transforms $\alpha_i$ are zero, and so the Borel transform of \eqref{200711161928} is the following PDE:
\eqntag{\label{200814080723}
	\del_z \varphi - \del_\xi \varphi = \tfrac{1}{2} \varphi \ast \varphi
\fullstop
}
This is equation \eqref{200117154820} from \autoref{211107195800} with
\eqn{
	\alpha_0 = \alpha_1 = \alpha_2 = 0,
\qquad
	a_0 = \pm \tfrac{5}{18} z^{-2},
\qquad
	a_1 = 0,
\qquad
	a_2 = \tfrac{1}{2}
\fullstop
}
In this case, the tubular neighbourhood $\Xi_+$ can be taken arbitrarily large.

\paragraph{Step 2: The integral equation.}
The PDE \eqref{200814080723} is easy to transform into an integral equation:
\eqntag{\label{200711164712}
	\varphi (z, \xi) = \varphi_0^\pm (z) + \frac{1}{2} \int_0^\xi \varphi \ast \varphi (z + t, \xi - t) \dd{t}
\fullstop{,}
}
where $\varphi (x, 0) = \varphi_0^\pm (z) \coleq a_0 (z) = \pm \tfrac{5}{18} z^{-2}$.
This is equation \eqref{190312202445} from \autoref{211107195800}.

\paragraph{Step 3: Method of successive approximations.}
The integral equation \eqref{200711164712} is solved by the method of successive approximations.
This method yields a sequence $\set{\varphi_n^\pm}_{n=0}^\infty$ of holomorphic functions given by $\varphi_0^\pm = a_0 = \pm \tfrac{5}{18} z^{-2}$, $\varphi_1^\pm = 0$, and for $n \geq 2$,
\eqns{
	\varphi_n^\pm 
		&= \frac{1}{2} \sum_{i + j = n-2} 
			\int_0^\xi \varphi_i^\pm \ast \varphi_j^\pm (z + t, \xi - t) \dd{t}
\\		&= \frac{1}{2} \sum_{i + j = n-2} 
			\int_0^\xi \int_0^{\xi - t} 
				\varphi_i^\pm (z + t, \xi - t - y)
				\varphi_j^\pm (z + t, y) \dd{y} \dd{t}
\fullstop
}
This is equation \eqref{180824203055} from \autoref{211107195800}.
It is easy to see that $\varphi_n = 0$ for all $n$ odd because $\varphi_1 = 0$.
The first few even terms of this sequence are:
{\small
\eqns{
	\varphi_0^\pm &= \pm \tfrac{5}{18} z^{-2}
\fullstop{;}
\\
	\varphi_2^\pm &= \tfrac{1}{12} \big(\pm \! \tfrac{5}{18}\big)^2 \tfrac{3z + 2\xi}{z^3 (z + \xi)^2} \xi^2
	\sim \tfrac{1}{6} \big(\pm \! \tfrac{5}{18}\big)^2 z^{-3} \xi
\rlap{\qquad\text{as $\xi \to + \infty$\fullstop{;}}}
\\
	\varphi_4^\pm &= \tfrac{1}{48} \big(\pm \! \tfrac{5}{18}\big)^3 \tfrac{\xi^4}{z^4 (z + \xi)^2}
	\sim \tfrac{1}{48} \big(\pm \! \tfrac{5}{18}\big)^3 z^{-4} \xi^2
\rlap{\qquad\text{as $\xi \to + \infty$\fullstop{;}}}
\\
	\varphi_6^\pm
	&= \tfrac{1}{2} \big(\pm \! \tfrac{5}{18}\big)^4 \left( \tfrac{\xi  \left(16 \xi^5+810 z^5+1650 \xi  z^4+915 \xi ^2 z^3+70 \xi ^3 z^2-8 \xi ^4 z\right)}{4320 z^5 (z + \xi)^2 (2z + \xi)}
+
\log \left(\tfrac{z}{z + \xi}\right)
\tfrac{7 \xi^2 + 27 z^2 + 28 \xi  z }{72 z^2 (2 z + \xi)^2}
\right)
\\
	&\sim \tfrac{1}{270} \big(\pm \! \tfrac{5}{18}\big)^4
		z^{-5} \xi^3
\rlap{\qquad\text{as $\xi \to + \infty$\fullstop{;}}}
\\
	\varphi_8^\pm
	&\sim \tfrac{7}{35 \, 640} \big(\pm \! \tfrac{5}{18}\big)^5
		z^{-6} \xi^4
\rlap{\qquad\text{as $\xi \to + \infty$\fullstop}}
}
}%
An exact expression for $\varphi_8^\pm$ involves logarithms and dilogarithms; it is very long and not very useful, occupying almost half of this page.
But the pattern is clear:
\eqn{
	\varphi_{2n}^\pm \in \OO \left( \tfrac{\LL^n}{n!} z^{-2} \big( \xi / z^2 \big)^n \right)
\rlap{\qquad\text{as $\xi \to + \infty$\fullstop{,}}}
}
for some constant $\LL > 0$ independent of $n$ and $z$.
It follows that the solution to the integral equation \eqref{200711164712} satisfies
\eqn{
	\varphi_\pm (z, \xi)
		= \sum_{n=0}^\infty \varphi_n (z, \xi)
		\quad\subdomeq\quad
			\sum_{n=0}^\infty \tfrac{1}{n!} z^{-2} \big( \LL \xi / z^2 \big)^n
		= z^{-2} e^{\LL \xi / z^2}
		\qqquad
{\text{as $\xi \to + \infty$\fullstop}}
}
This asymptotic inequality yields the exponential estimate \eqref{200119155725} from \autoref{211107195800} with $\AA = 4$ and $\KK = 0$.

\paragraph{Step 4: Laplace transform.}
Applying the Laplace transform to $\varphi_\pm$, we obtain exact solutions $\FF_\pm$ of the two Riccati equations \eqref{200711161928}:
\eqn{
	\FF_\pm (z, \xi) \coleq \int_0^{+\infty} e^{-\xi/\hbar} \varphi_\pm (z, \xi) \dd{\xi}
\fullstop
}
It is evident from the asymptotic behaviour of $\varphi_\pm$ as $\xi \to +\infty$ that this Laplace integral is uniformly convergent for all $z \in H_0$ and all $\hbar \in S_0 \coleq \set{ \Re (1/\hbar) > \LL }$.
Note that it is not uniformly convergent for $z \in H$, because the constant $\AA$ in the estimate for $\varphi_\pm$ grows like $|z|^{-2}$.

Finally, using the inverse Liouville transformation $\Phi^{-1}_0: z \mapsto \big( \frac{3}{4} z \big)^{2/3}$ to go back to the $x$-variable, we obtain two exact solutions of the Riccati equation \eqref{200709152427}:
\eqn{
	g_\pm (x, \hbar)
		\coleq \pm 1 + \tfrac{1}{4} \hbar x^{\!\!-\nicefrac{3}{2}}
			+ \hbar \int_0^{+\infty} e^{-\xi/\hbar}
				\varphi_\pm \big( \tfrac{4}{3} x^{\!\nicefrac{3}{2}}, \xi \big) 
				\dd{\xi}
\fullstop
}
Transforming back to the original Riccati equation \eqref{200701111940} via the identities \eqref{200814082235} and $f = x^{-\nicefrac{1}{2}} g$, we obtain two exact solutions of the original Riccati equation \eqref{200701111940}:
\eqn{
	f_\pm (x, \hbar) = \pm \sqrt{x} + \tfrac{1}{4x} \hbar + \hbar \sqrt{x} \int_0^{+\infty} e^{-\xi/\hbar} \varphi_\pm \big( \tfrac{4}{3} x^{\!\nicefrac{3}{2}}, \xi \big) \dd{\xi}
\fullstop
}
These are the two canonical exact solutions on $U_0$.

\subsection{Exact WKB Solutions of Schrödinger Equations}
\label{211119204333}

In this final subsection, we give an application of our existence and uniqueness result to deduce existence and uniqueness of the so-called \textit{exact WKB solutions} of the complex one-dimensional stationary Schrödinger equation
\eqntag{
\label{200228140445}
	\Big( \hbar^2 \del_x^2 - q (x, \hbar) \Big) \psi (x, \hbar) = 0
\fullstop
}
We keep the discussion here very brief; the details can be found in \cite{MY210623112236}.
The \textit{potential} function $q (x, \hbar)$ is defined on a domain in $\Complex_x \times \Complex_\hbar$ and usually assumed to have polynomial or even constant dependence on $\hbar$.
In view of the work done in this article, we can assume a much more general $\hbar$-dependence, but for simplicity of presentation let us suppose that $q$ is a polynomial in $\hbar$:
\eqn{
	q (x, \hbar) = q_0 (x) + q_1 (x) \hbar + \cdots + q_n (x) \hbar^n
\fullstop
}
The WKB method begins by searching for a solution in the form of the \textit{WKB ansatz}:
\eqntag{
\label{200228142121}
	\psi (x, \hbar) = \exp \left( - \frac{1}{\hbar} \int_{x_\ast}^x f (t, \hbar) \dd{t} \right)
\fullstop{,}
}
where $x_0$ is a chosen basepoint, and $f = f(x, \hbar)$ is the unknown function to be solved for.
Substituting this expression back into the Schrödinger equation, we find that the WKB ansatz \eqref{200228142121} is a solution if the function $f$ satisfies the singularly perturbed Riccati equation
\eqntag{
\label{200228143858}
	\hbar \del_x f = f^2 - q
\fullstop
}
Locally in $x$, this Riccati equation has two formal solutions $\hat{f}_\pm$ with locally holomorphic leading-orders $f^\pm_0 = \pm \sqrt{q_0}$.
They give rise to a pair of \dfn{formal WKB solutions}, which by definition are the following formal expressions:
\eqntag{
	\hat{\psi}_\pm (x, \hbar) \coleq \exp \left( - \frac{1}{\hbar} \int_{x_0}^x \hat{f}_\pm (t, \hbar) \dd{t} \right)
\fullstop
}
An \dfn{exact WKB solution} is any analytic solution $\psi (x, \hbar)$ to the \Schrodinger equation \eqref{200228140445} which is asymptotic as $\hbar \to 0$ in the right halfplane to a formal WKB solution.

\begin{thm}[\textbf{Local Existence of Exact WKB Solutions}]{200309160221}
\mbox{}\\
Consider a \Schrodinger equation \eqref{200228140445} with potential $q = q (x, \hbar)$ which is a polynomial in $\hbar$ whose coefficients are rational functions on $\Complex_x$.
We make the following two assumptions:
\begin{enumerate}
\item Suppose that the poles of $q$ have order at least $2$ and that they are completely specified in the leading-order term $q_0$.
More precisely, if $D \subset \Complex_x \cup \set{\infty}$ is the set of poles of $q_0$, we assume that every pole $x_\infty \in D$ has order $\op{ord} (q_0) \geq 2$; we assume furthermore that every $q_k$ has no poles other than $D$ and that for every $x_\infty \in D$ we have $\op{ord} (q_i) \leq \op{ord} (q_0)$ .
\item Fix a basepoint point $x_0 \in \Complex_x$ which is neither a pole nor a zero of $q_0$, and assume that the real one-dimensional curve 
\eqntag{
	\hspace{-1cm}
	\Gamma (x_0) \coleq \set{ x \in \Complex_x ~\Big|~ \Im \left( ~\int\nolimits_{x_0}^x \sqrt{q_0 (t)} \dd{t} \right) = 0 }
\fullstop{,}
}
limits at both ends into points of $D$ (not necessarily distinct).
\end{enumerate}

Then the Schrödinger equation \eqref{200228140445} has a canonical local basis of exact WKB solutions $\psi_\pm$ normalised at $x_0$:
\eqntag{\label{200229142913}
	\psi_\pm (x_0, \hbar) = 1
\qtext{and}
	\psi_\pm (x, \hbar)
		\sim
		\hat{\psi}_\pm (x, \hbar)
\text{\quad as $\hbar \to 0$ in the right halfplane\fullstop}
}
\end{thm}

\begin{proof}
We consider the corresponding Riccati equation \eqref{200228143858}.
Its leading-order discriminant is simply $\DD_0 = 4q_0$.
The assumptions on the pole orders of $q$ and the fact that $\Gamma (x_0)$ flows into $D$ at both ends imply that $\Gamma (x_0)$ is a generic WKB trajectory.
Thus, all the hypotheses of \autoref{211112183728} (or more specifically of \autoref{211121161720}) are satisfied, so \eqref{200228143858} has a canonical pair of exact solutions $f_\pm$ defined for $x$ near $x_0$ and asymptotic to $\hat{f}_\pm$ as $\hbar \to 0$ in the right halfplane.
The exact WKB solutions $\psi_\pm$ are then defined as $\psi_\pm (x, \hbar) \coleq \exp \left( - \hbar^{-1} \int_{x_0}^x f_\pm (t, \hbar) \dd{t} \right)$.
They form a basis of solutions near $x_0$ because the Wronskian of $\psi_+$ and $\psi_-$ evaluated at $x = x_0$ is $f_+ (x_0, \hbar) - f_- (x_0, \hbar) \neq 0$.
\end{proof}

\begin{appendices}
\appendixsectionformat
\section{Basic Notions from Asymptotic Analysis}
\setcounter{section}{1}
\setcounter{paragraph}{0}
\label{210714114437}

\paragraph{Sectorial domains.}
\label{210217114252}
Fix a circle $\Sphere^1 \coleq \Real / 2\pi \Integer$ once and for all.
We refer to its points as \dfn{directions}, and we think of it as the set of directions at the origin in $\Complex_\hbar$ when $\hbar$ is written in polar coordinates.
More precisely, we consider the \textit{real-oriented blowup} of the complex plane $\Complex_\hbar$ at the origin, which by definition is the bordered Riemann surface $[\Complex_\hbar : 0] \coleq \Real_+ \times \Sphere^1$ with coordinates $(r, \theta)$, where $\Real_+$ is the nonnegative reals.
The projection $[\Complex_\hbar : 0] \to \Complex_\hbar$ sends $(r, \theta) \mapsto r e^{i\theta}$, which is a biholomorphism away from the circle of directions.
See \autoref{210618080005} for an illustration.
\begin{figure}[t]
\centering
\includegraphics{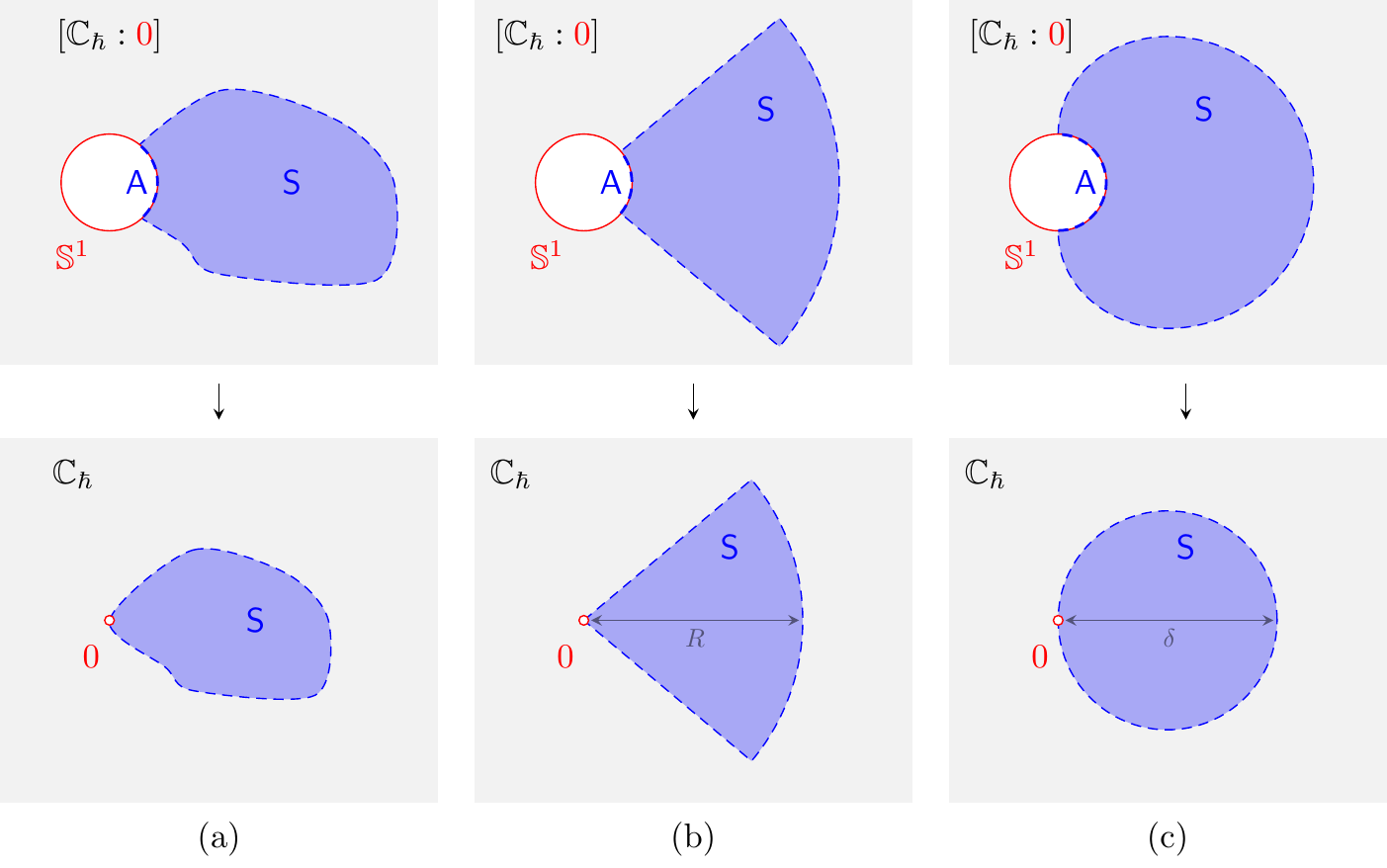}
\caption{Examples of sectorial domains.
(a): an arbitrary sectorial domain with opening $A$.
(b): straight sector with opening $A$ and some radius $\RR > 0$.
(c): a Borel disc with diameter $\delta > 0$.}
\label{210618080005}
\end{figure}

A \dfn{sectorial domain} near the origin in $\Complex_\hbar$ is a simply connected domain $S \subset \Complex_\hbar^\ast = \Complex_\hbar \setminus \set{0}$ whose closure $\bar{S}$ in $[\Complex_\hbar : 0]$ intersects the boundary circle $\Sphere^1$ in a closed arc $\bar{A} \subset \Sphere^1$ with nonzero length.
In this case, the open arc $A$ is called the \dfn{opening} of $S$, and its length $|A|$ is called the \dfn{opening angle} of $S$.
A \dfn{proper subsectorial domain} $S_0 \subset S$ is one whose closure $\bar{S}_0$ in $[\Complex_\hbar : 0]$ is contained in $S$.
This means in particular that the opening $A_0$ of $S_0$ is compactly contained in $A$; i.e., $\bar{A}_0 \subset A$.

The simplest example of a sectorial domain $S$ is of course a \textit{straight sector} of radius $\RR > 0$ and opening $A$, which is the set of points $\hbar \in \Complex_\hbar$ satisfying $\arg (\hbar) \in A$ and $0 < |\hbar| < \RR $.
The most typical example of a sectorial domain encountered in this paper is a \dfn{Borel disc} of \dfn{diameter} $\delta >0$:
\eqntag{
	S (\delta) \coleq \set{ \hbar \in \Complex_\hbar ~\big|~ \Re (1/\hbar) > 1/\delta }
\fullstop
}
Its opening is $A = (-\pi/2, +\pi/2)$.
Notice that any straight sector with opening $A$ contains a Borel disc, but a Borel disc contains no straight sectors with opening $A$.
It is also not difficult to see that any sectorial domain with opening $A$ contains a Borel disc.
More generally, we will consider Borel discs bisected by some direction $\theta \in \Sphere^1$:
\eqntag{
	S_\theta (\delta) \coleq \set{ \hbar \in \Complex_\hbar ~\big|~ \Re (e^{i\theta}/\hbar) > 1/\delta }
\fullstop
}

\subsection{Poincaré Asymptotics in One Dimension}
\label{210220155700}

First, for the benefit of the reader and to fix some notation, let us briefly recall some basic notions from asymptotic analysis in one complex variable.
We denote by $\Complex \bbrac{\hbar}$ the ring of formal power series in $\hbar$, and by $\Complex \set{\hbar}$ the ring of convergent power series in $\hbar$.
Fix an arc of directions $A \subset \Sphere^1$.
We denote by $\cal{O} (S)$ the ring of holomorphic functions on a sectorial domain $S \subset \Complex_\hbar$ with opening $A$.

\paragraph{Sectorial germs.}
\label{210225113610}
For the purpose of asymptotic behaviour as $\hbar \to 0$, the actual nonzero radial size of $S$ is irrelevant.
So it is better to consider \textit{germs} of holomorphic functions defined on sectorial domains with opening $A$, formally defined next.

\begin{defn}{210616134514}
A \dfn{sectorial germ} on $A$ is an equivalence class of pairs $(f,S)$ where $S \subset \Complex_\hbar$ is a sectorial domain with opening $A$ and $f$ is a holomorphic function on $S$.
Two such pairs $(f,S)$ and $(f',S')$ are considered equivalent if the intersection ${S \cap S'}$ contains a sectorial domain $S''$ with opening $A$ on which $f$ and $f'$ are equal.
\end{defn}

Sectorial germs on $A$ form a ring which we denote by $\cal{O} (A)$.
For any sectorial domain $S \subset \Complex_\hbar$ with opening $A$, there is a map $\cal{O} (S) \to \cal{O} (A)$ that sends a holomorphic function $f$ to the corresponding sectorial germ.

\emph{Terminology:} For the benefit of the reader who is uncomfortable with the language of germs, we stress that every sectorial germ on $A$ can be represented by an actual holomorphic function $f \in \cal{O} (S)$ on some sectorial domain $S$.
In fact, we normally denote the equivalence class of any pair $(f, S)$ simply by ``$f$'' and we often even refer to it as a \textit{holomorphic function on $A$}.

\paragraph{Poincaré asymptotics.}
\label{210225114212}
Recall that a holomorphic function $f \in \cal{O} (S)$ defined on a sectorial domain $S$ is said to admit (\dfn{Poincaré}) \dfn{asymptotics as $\hbar \to 0$ along $A$} if there is a formal power series $\hat{f} (\hbar) \in \Complex \bbrac{\hbar}$ such that the order-$n$ remainder
\eqntag{\label{200720152534}
	\RR_n (\hbar) \coleq f(\hbar) - \sum_{k=0}^{n-1} f_k \hbar^k
}
is bounded by $\hbar^n$ for all sufficiently small $\hbar \in S$.
That is, for every ${n \geq 0}$, and every compactly contained subarc $A_0 \Subset A$, there is a sectorial subdomain $S_0 \subset S$ with opening $A_0$ and a real constant $\CC_{n,0} > 0$ such that
\eqntag{\label{200720153758}
	\big| \RR_n (\hbar) \big| \leq \CC_{n,0} |\hbar|^n
}
for all $\hbar \in S_0$.
The constants $\CC_{n,0}$ may depend on $n$ and the opening $A_0$.
If this is the case, we write
\eqntag{\label{200720175735}
	f (\hbar) \sim \hat{f} (\hbar)
\qqqquad
	\text{as $\hbar \to 0$ along $A$\fullstop}
}
Sectorial germs with this property form a subring $\cal{A} (A) \subset \cal{O}(A)$, and the asymptotic expansion map defines a ring homomorphism $\ae : \cal{A}(A) \to \Complex \bbrac{\hbar}$.

\paragraph{Asymptotics along a closed arc.}
\label{210225134124}
Furthermore, we will write
\eqntag{\label{210220160756}
	f (\hbar) \sim \hat{f} (\hbar)
\qqqquad
	\text{as $\hbar \to 0$ along $\bar{A}$\fullstop{,}}
}
if the constants $\CC_{n,0}$ in \eqref{200720153758} can be chosen uniformly for all compactly contained subarcs $A_0 \Subset A$ (i.e., independent of $A_0$ so that $\CC_{n,0} = \CC_n$ for all $n$).
Obviously, if a function $f$ admits asymptotics along $A$, then it admits asymptotics along $\bar{A}_0$ for any $A_0 \Subset A$.
Sectorial germs with this property form a subring $\cal{A} (\bar{A}) \subset \cal{A}(A)$.

For example, the function $e^{-1/\hbar}$ admits asymptotics as $\hbar \to 0$ along the open arc $A = (-\pi/2, +\pi/2)$ (where it is asymptotic to $0$), but \textit{not} along the closed arc $\bar{A} = [-\pi/2, +\pi/2]$, because the constants $\CC_{n,0}$ in the asymptotic estimates \eqref{200720153758} blow up as $A_0$ approaches $A$.
Thus, $e^{-1/\hbar}$ is an element of $\cal{A} (A)$ but not of $\cal{A} (\bar{A})$.

\subsection{Uniform Poincaré Asymptotics}
\label{210225141606}

Now, suppose we also have a domain $U \subset \Complex_x$.

\paragraph{Power series with holomorphic coefficients.}
\label{200721171954}
We denote by $\cal{O} (U) \bbrac{\hbar}$ the set of formal power series in $\hbar$ with holomorphic coefficients on $U$:
\eqntag{\label{200624150851}
	\hat{f} (x, \hbar) = \sum_{k=0}^\infty f_k (x) \hbar^k 
	\quad \in \quad
	\cal{O} (U) \bbrac{\hbar}
\fullstop
}
Let us also introduce the subsets $\cal{O} (U) \set{\hbar}$ and $\cal{O}_\text{loc} (U) \set{\hbar}$ of $\cal{O} (U) \bbrac{\hbar}$ consisting of, respectively, uniformly and locally-uniformly convergent power series on $U$.
We will not have any use for pointwise convergence (or other pointwise regularity statements), so we do not introduce any special notation for those.

\paragraph{Semisectorial germs.}
\label{210226112200}
Since we are only interested in keep track of the asymptotic behaviour as $\hbar \to 0$ along $A$, we focus our attention on holomorphic functions of $(x, \hbar)$ which behave like sectorial germs in $\hbar$.
More precisely, we introduce the following definition.


\begin{defn}{210226161539}
A \dfn{semisectorial germ} on $(U; A)$ is an equivalence class of pairs $(f,\mathbb{U})$ where $f$ is a holomorphic function on the domain $\mathbb{U} \subset U \times \Complex_\hbar^\ast$ with the following property: for every point $x_0 \in U$, there is a neighbourhood $U_0 \subset U$ of $x_0$ and a sectorial domain $S_0$ with opening $A$ such that $U_0 \times S_0 \subset \mathbb{U}$.
Any two such pairs $(f,\mathbb{U})$ and $(f',\mathbb{U}')$ are considered equivalent if, for every $x_0 \in U$, the intersection $S_0 \cap S'_0$ contains a sectorial domain $S''_0$ with opening $A$ such that the restrictions of $f$ and $f'$ to $U_0 \times S''_0$ are equal.
\end{defn}

\emph{Terminology:} We will often abuse terminology and refer to semisectorial germs $f \in \cal{O} (U;A)$ as \textit{holomorphic functions on $(U;A)$}.

Typically, $\mathbb{U}$ is a product domain $U \times S$ for some $S$ or a (possibly countable) union of product domains.
Notice that in particular the projection of $\mathbb{U}$ onto the first component is necessarily $U$.
Semisectorial germs form a ring which we denote by $\cal{O} (U; A)$.
For any domain $\UUU \subset \Complex_{x\hbar}^2$ as above, there is a map $\cal{O} (\UUU) \to \cal{O} (U;A)$ sending a holomorphic function $f$ on $\UUU$ to the corresponding semisectorial germ; i.e., the equivalence class of $f$ in $\cal{O} (U;A)$.
There is also a map $\cal{O}_\text{loc}  (U) \set{\hbar} \to \cal{O} (U;A)$ for any $A$ given by restriction.

\paragraph{Uniform Poincaré asymptotics.}
\label{200721171635}
Consider the product space $\UUU = U \times S$ where $S$ is a sectorial domain with opening $A$, or more generally let $\UUU$ be a domain of the form described in \autoref{210226161539}.
Recall that a holomorphic function $f \in \cal{O} (\UUU)$ is said to admit (pointwise Poincaré) \dfn{asymptotics as $\hbar \to 0$ along $A$} if there is a formal power series $\hat{f} (x, \hbar) \in \cal{O} (U) \bbrac{\hbar}$ such that for every ${n \geq 0}$, every $x \in U$, and every compactly contained subarc $A_0 \Subset A$, there is a sectorial domain $S_0$ with opening $A_0$ and a real constant $\CC_{n,x,0} > 0$ which satisfies the following inequality:
\eqntag{\label{200305151735}
	\Big| \RR_n (x, \hbar) \Big|
		= \left| f(x, \hbar) - \sum_{k=0}^{n-1} f_k (x) \hbar^k \right|
		\leq \CC_{n,x,0} |\hbar|^n
}
for all $\hbar \in S_0$.
The constant $\CC_{n,x,0}$ may depend on $n,x$, and the opening $A_0$.
We say that $f$ admits \dfn{uniform asymptotics} \textit{on $U$ as $\hbar \to 0$ along $A$} if $\CC_{n,x,0}$ can be chosen to be independent of $x$ (i.e., so that $\CC_{n,x,0} = \CC_{n,0}$).
We also say that $f$ admits \dfn{locally uniform asymptotics} \textit{on $U$ as $\hbar \to 0$ along $A$} if every point in $U$ has a neighbourhood on which $f$ admits uniform asymptotics.
In these cases, we write, respectively,
\eqnstag{\label{210225121434}
	f (x,\hbar) &\sim \hat{f} (x,\hbar)
\qqqquad
	\text{as $\hbar \to 0$ along $A$, unif. $\forall x \in U$\fullstop{;}}
\\
\label{210225121453}
	f (x,\hbar) &\sim \hat{f} (x,\hbar)
\qqqquad
	\text{as $\hbar \to 0$ along $A$, loc.unif. $\forall x \in U$\fullstop}
}
We denote the subrings consisting of semisectorial germs satisfying \eqref{210225121434} or \eqref{210225121453} respectively by $\cal{A} (U; A)$ and $\cal{A}_\text{loc} (U; A)$.
The asymptotic expansion map defines a ring homomorphism $\ae : \cal{A}_\text{loc}  (U;A) \to \cal{O} (U) \bbrac{\hbar}$.
An elementary application of the Cauchy integral formula shows that the ring $\cal{A}_\text{loc}  (U; A)$ (but not $\cal{A} (U; A)$) is preserved under differentiation with respect to $x$; that is, $\del_x \big( \cal{A}_\text{loc}  (U; A) \big) \subset \cal{A}_\text{loc}  (U; A)$.

\paragraph{Uniform Poincaré asymptotics along a closed arc.}
\label{210225142705}
If in addition to \eqref{210225121434} or \eqref{210225121453}, the constants $\CC_{n,x,0}$ can be chosen uniformly for all $A_0 \Subset A$ (i.e., so that $\CC_{n,x,0} = \CC_{n,x}$ for all $n$ and $x$), we will write, respectively,
\eqnstag{\label{210225123426}
	f (x,\hbar) &\sim \hat{f} (x,\hbar)
\qqqquad
	\text{as $\hbar \to 0$ along $\bar{A}$, unif. $\forall x \in U$\fullstop{;}}
\\
\label{210225123428}
	f (x,\hbar) &\sim \hat{f} (x,\hbar)
\qqqquad
	\text{as $\hbar \to 0$ along $\bar{A}$, loc.unif. $\forall x \in U$\fullstop}
}
Holomorphic functions satisfying these conditions form subrings which we denote respectively by $\cal{A} (U; \bar{A}) \subset \cal{A} (U; A) $ and $\cal{A}_\text{loc} (U; \bar{A}) \subset \cal{A}_\text{loc} (U; A)$.
Again, the ring $\cal{A}_\text{loc} (U; \bar{A})$ (but not the ring $\cal{A} (U; \bar{A})$) is preserved under differentiation by $x$.

\subsection{Gevrey Asymptotics}
\label{210224181219}
\enlargethispage{1cm}

For the purposes of the main construction in this paper, the notion of Poincaré asymptotics is too weak.
A powerful and systematic way to refine Poincaré asymptotics is known as \textit{Gevrey asymptotics}.
The basic principle behind it is to strengthen the asymptotic requirements by specifying the dependence on $n$ of the constants $\CC_{n,0}$ in \eqref{200720153758} and $\CC_{n,x,0}$ in \eqref{200305151735}.
In this paper, we use only the simplest Gevrey regularity class (more properly known as \textit{1-Gevrey asymptotics}) which requires the asymptotic bounds to grow essentially like $n!$.
See, for example, \cite[\S1.2]{MR3495546} for a more general introduction to Gevrey asymptotics.
As before, for the benefit of the reader we first recall Gevrey asymptotics in one complex dimension.

\paragraph{Gevrey asymptotics in one dimension.}
\label{200722160857}
A holomorphic function $f \in \cal{O} (S)$ defined on a sectorial domain $S$ with opening $A$ is said to admit \dfn{Gevrey asymptotics as $\hbar \to 0$ along $A$} if the constants $\CC_{n,0}$ in \eqref{200720153758} depend on $n$ like $\CC_0 \MM_0^n n!$.
More explicitly, there is a formal power series $\hat{f} (\hbar) \in \Complex \bbrac{\hbar}$ such that for every compactly contained subarc $A_0 \Subset A$, there is a sectorial domain $S_0 \subset S$ with opening $A_0 \Subset A$ and real constants $\CC_0, \MM_0 > 0$ which for all $n \geq 0$ give the bounds
\eqntag{\label{200722160158}
	\big| \RR_n (\hbar) \big| \leq \CC_0 \MM_0^n n! |\hbar|^n
}
for all $\hbar \in S_0$.
We will use the symbol ``$\simeq$'' to distinguish Gevrey asymptotics from Poincaré asymptotics.
Thus, if $f$ admits Gevrey asymptotics as $\hbar \to 0$ along $A$, we will write
\eqntag{\label{210225131044}
	f (\hbar) \simeq \hat{f} (\hbar)
\qqqquad
	\text{as $\hbar \to 0$ along $A$\fullstop}
}
We denote the subring of sectorial germs satisfying this property by $\cal{G} (A) \subset \cal{A} (A)$.
If in addition to \eqref{200722160158}, the constants $\CC_0, \MM_0$ can be chosen uniformly for all $A_0 \Subset A$, then we will write
\eqntag{\label{210225134416}
	f (\hbar) \simeq \hat{f} (\hbar)
\qqqquad
	\text{as $\hbar \to 0$ along $\bar{A}$\fullstop}
}
We denote the subring of sectorial germs satisfying this property by $\cal{G} (\bar{A}) \subset \cal{G} (A)$.
Explicitly, $f \in \cal{O} (\bar{A})$ if and only if there is a sectorial domain $S_0$ with opening $A$ and real constants $\CC, \MM > 0$ which give the following bounds for all $n \geq 0$ and $\hbar \in S_0$:
\eqntag{\label{210225134829}
	\big| \RR_n (\hbar) \big| \leq \CC \MM^n n! |\hbar|^n
\fullstop
}

\paragraph{Gevrey series in one dimension.}
\label{200722151439}
It is easy to see that if $f \in \cal{G} (A)$ then the coefficients of its asymptotic expansion also grow essentially like $n!$.
By definition, a formal power series $\hat{f} (\hbar) = \sum f_n \hbar^n \in \Complex \bbrac{\hbar}$ is a \dfn{Gevrey power series} if there are constants $\CC, \MM > 0$ such that for all $n \geq 0$,
\eqntag{\label{200723182724}
	| f_n | \leq \CC \MM^n n!
\fullstop
}
Gevrey series form a subring $\Gevrey \bbrac{\hbar} \subset \Complex \bbrac{\hbar}$, and the asymptotic expansion map restricts to a ring homomorphism $\ae : \cal{G}(A) \to \Gevrey \bbrac{\hbar}$.
The ring $\cal{G} (\bar{A})$ for an arc with opening $|A| = \pi$ plays a central role in Gevrey asymptotics because it is possible to identify and describe a subclass $\bar{\Gevrey} \bbrac{\hbar} \subset \Gevrey \bbrac{\hbar}$ such that the asymptotic expansion map restricts to a bijection $\ae : \cal{G} (\bar{A}) \iso \bar{\Gevrey} \bbrac{\hbar}$.
This identification is done using a theorem of Nevanlinna and requires techniques from the Borel-Laplace theory, see \autoref{210616130753}.

\paragraph{Examples in one dimension.}
\label{210617074256}
Any function which is holomorphic at $\hbar = 0$ automatically admits Gevrey asymptotics as $\hbar \to 0$ along any arc: its asymptotic expansion is nothing but its convergent Taylor series at $\hbar =  0$ whose coefficients necessarily grow at most exponentially.
Also, if a function admits Gevrey asymptotics as $\hbar \to 0$ along a strictly larger arc $A'$ which contains the closed arc $\bar{A}$, then it automatically admits Gevrey asymptotics along $\bar{A}$.

The function $e^{-1/\hbar}$ admits Gevrey asymptotics as $\hbar \to 0$ in the right halfplane where it is asymptotic to $0$.
So if $A = (-\pi/2, +\pi/2)$, then $e^{-1/\hbar} \in \cal{G} (A)$, but $e^{-1/\hbar} \not\in \cal{G} (\bar{A})$ as discussed before.
On the other hand, the function $e^{-1/\hbar^\alpha}$ for any real number $0 < \alpha < 1$ does not admit Gevrey asymptotics as $\hbar \to 0$ along any arc.

\paragraph{Gevrey series with holomorphic coefficients.}
\label{210225135829}
Now, suppose again that in addition we have a domain $U \subset \Complex_x$.
A formal power series $\hat{f} (x, \hbar) \in \cal{O} (U) \bbrac{\hbar}$ on $U$ is called a \dfn{Gevrey series} if, for every $x \in U$, there are constants $\CC_x, \MM_x > 0$ such that, for all $n \geq 0$,
\eqntag{\label{190303174313}
	\big| f_n (x) \big| \leq \CC_x \MM_x^n n!
\fullstop
}
Furthermore, $\hat{f}$ is a \dfn{uniformly Gevrey series} on $U$ if the constants $\CC_x, \MM_x$ can be chosen to be independent of $x \in U$.
$\hat{f}$ is a \dfn{locally uniformly Gevrey series} on $U$ if every $x_0 \in U$ has a neighbourhood $U_0 \subset U$ where $\hat{f}$ is uniformly Gevrey.
Such power series form subrings $\cal{G} (U) \bbrac{\hbar}$ and $\cal{G}_\text{loc}  (U) \bbrac{\hbar}$ of $\cal{O} (U) \bbrac{\hbar}$ respectively.
The ring $\cal{G}_\text{loc} (U) \bbrac{\hbar}$ is preserved by $\del_x$.

\paragraph{Uniform Gevrey asymptotics.}
\label{210225135947}
Again, consider the product space $\UUU = U \times S$ where $S$ is a sectorial domain with opening $A$, or more generally let $\UUU$ be a domain of the form described in \autoref{210226161539}.
A holomorphic function $f \in \cal{O} (\UUU)$ is said to admit (pointwise) \dfn{Gevrey asymptotics on $U$ as $\hbar \to 0$ along $A$} if for every $x \in U$ and every compactly contained subarc $A_0 \Subset A$, there is a sectorial domain $S_0$ with opening $A_0$ and constants $\CC_{x,0}, \MM_{x,0} > 0$ such that for all $n \geq 0$ and all $\hbar \in S_0$,
\eqntag{\label{200624162243}
	\big| \RR_n (x, \hbar) \big|
		= \left| f(x, \hbar) - \sum_{k=0}^{n-1} f_k (x) \hbar^k \right|
		\leq \CC_{x,0} \MM^n_{x,0} n! |\hbar|^n
\fullstop
}
We say that $f$ admits \dfn{uniform Gevrey asymptotics} on $U$ as $\hbar \to 0$ along $A$ if the constants $\CC_{x,0}, \MM_{x,0}$ can be chosen to be independent of $x \in U$ (i.e., so that $\CC_{x,0} = \CC_0, \MM_{x,0} = \MM_0$).
We also say $f$ admits \dfn{locally uniform Gevrey asymptotics} on $U$ if every point $x_0 \in U$ has a neighbourhood $U_0 \subset U$ on which $f$ has uniform Gevrey asymptotics.
In these cases, we write, respectively,
\eqnstag{\label{210225142615}
	f (x,\hbar) &\simeq \hat{f} (x,\hbar)
\qqqquad
	\text{as $\hbar \to 0$ along $A$, unif. $\forall x \in U$\fullstop{;}}
\\
\label{210225142617}
	f (x,\hbar) &\simeq \hat{f} (x,\hbar)
\qqqquad
	\text{as $\hbar \to 0$ along $A$, loc.unif. $\forall x \in U$\fullstop}
}
Such functions form subrings $\cal{G} (U; A) \subset \cal{A} (U; A)$ and $\cal{G}_\text{loc} (U; A) \subset \cal{A}_\text{loc} (U; A)$ respectively.
The asymptotic expansion map $\ae$ restricts to ring homomorphisms $\cal{G} (U;A) \to \cal{G} (U) \bbrac{\hbar}$ and $\cal{G}_\text{loc} (U;A) \to \cal{G}_\text{loc} (U) \bbrac{\hbar}$.
As in the case of uniform Poincaré asymptotics, an application of the Cauchy integral formula shows that the ring $\cal{G}_\text{loc} (U; A)$ (but not the ring $\cal{G} (U; A)$) is preserved under differentiation by $x$.

\paragraph{Uniform Gevrey asymptotics along a closed arc.}
\label{210225142900}
If in addition to \eqref{210225142615} or \eqref{210225142617}, the constants $\CC_{x,0}, \MM_{x,0}$ can be chosen uniformly for all $A_0 \Subset A$ (i.e., so that $\CC_{x,0} = \CC_x$ and $\MM_{x,0} = \MM_x$), we will write, respectively,
\eqnstag{\label{210226091301}
	f (x,\hbar) &\simeq \hat{f} (x,\hbar)
\qqqquad
	\text{as $\hbar \to 0$ along $\bar{A}$, unif. $\forall x \in U$\fullstop{;}}
\\
\label{210226091307}
	f (x,\hbar) &\simeq \hat{f} (x,\hbar)
\qqqquad
	\text{as $\hbar \to 0$ along $\bar{A}$, loc.unif. $\forall x \in U$\fullstop}
}
Such functions form subrings $\cal{G} (U; \bar{A}) \subset \cal{G} (U; \bar{A})$ and $\cal{G}_\text{loc} (U; \bar{A}) \subset \cal{G}_\text{loc} (U; A)$ respectively.
Again, we have $\del_x \big( \cal{G}_\text{loc} (U; \bar{A}) \big) \subset \cal{G}_\text{loc} (U; \bar{A})$.

\paragraph{Example.}
\label{200708191952}
The following example is related to what is sometimes called the \textit{Euler series} \cite[Example 1.1.4]{MR3495546}.
Consider the following formal series on $U = \Complex_x^\ast$:
\eqntag{\label{200722153024}
	\hat{\EE} (x, \hbar) 
		\coleq - \sum_{k=1}^\infty (-x)^{-k} (k-1)! \hbar^k
		= \frac{\hbar}{x} \sum_{k=0}^\infty \frac{k!}{(-x)^k} \hbar^k
	~\in~ \cal{O} (U) \bbrac{\hbar}
\fullstop
}
Incidentally, $\hat{\EE}$ is a formal solution of the differential equation $\hbar^2 \del_\hbar \EE + x \EE = \hbar$, but this fact is not important for the discussion in this example.

It is easy to see that the power series $\hat{\EE}$ has zero radius of convergence for any fixed nonzero $x$, but it is a Gevrey series for which the bounds \eqref{190303174313} can be satisfied by taking $\CC_x = 1$ and $\MM_x = |x|^{-k}$.
This demonstrates that $\hat{\EE}$ is a locally uniform Gevrey series on $\Complex_x^\ast$.
One can show that $\hat{\EE}$ is not a uniform Gevrey series.
Thus, $\hat{\EE} \in \cal{G}_\text{loc} (U) \bbrac{\hbar}$ but $\not\in \cal{G} (U) \bbrac{\hbar}$.

Consider the function
\eqn{
	\EE (x, \hbar)
		\coleq \int_0^{+\infty} \frac{e^{- \xi/\hbar}}{x + \xi} \dd{\xi}
\fullstop
}
It is well-defined and holomorphic for all $x$ in the cut plane $U' \coleq \Complex_x \setminus \Real_-$ and all $\hbar$ with ${\Re (\hbar) > 0}$.
It is not holomorphic at $\hbar = 0$ for any $x \in U$, but it is bounded as $\hbar \to 0$ in the right halfplane.
In fact, $\EE$ admits the power series $\hat{\EE}$ as its locally uniform Gevrey asymptotics:
\eqntag{\label{210225150455}
	\EE (x, \hbar) \simeq \hat{\EE} (x, \hbar)
\qquad
	\text{as $\hbar \to 0$ along $[-\nicefrac{\pi}{2}, +\nicefrac{\pi}{2}]$, loc.unif. $\forall x \in U'$\fullstop}
}
In symbols, $\EE \in \cal{G}_\text{loc} (U'; A)$.
To see this, we can write:
\eqn{
	\frac{1}{x + \xi}
	= \frac{1}{x} \frac{1}{1 + \xi/x}
	= \frac{1}{x} \sum_{k=0}^{n-2} 
		\frac{1}{(-x)^k} \xi^k
		+ \frac{1}{(-x)^{n-1}} \frac{\xi^{n-1}}{x + \xi}
\fullstop
}
Therefore, we obtain the relation
\eqn{
	\EE (x, \hbar) 
		= \frac{\hbar}{x} \sum_{k=0}^{n-2} \frac{k!}{(-x)^k} \hbar^k
			+ \frac{1}{(-x)^{n-1}} \int_0^{+\infty} 
				\frac{\xi^{n-1} e^{- \xi / \hbar}}{x + \xi} \dd{\xi}
\fullstop
}
So to demonstrate \eqref{210225150455}, one can find a locally uniform bound on the integral which is valid uniformly for all directions in the halfplane arc $(-\pi/2, +\pi/2)$.

\newpage
\section{Basics of the Borel-Laplace Theory}
\label{210616130753}

In this appendix section, we recall some basic definitions from the theory of Borel-Laplace transforms.

\paragraph{}
Let $U \subset \Complex_x$ be a domain.
Fix a direction $\theta \in \Sphere^1$, let $A_\theta$ be the halfplane arc bisected by $\theta$, and let $S_\theta$ be the Borel disc bisected by $\theta$ of some diameter $\delta > 0$:
\eqntag{\label{210616181227}
	A_\theta \coleq (\theta -\tfrac{\pi}{2}, \theta + \tfrac{\pi}{2})
\qtext{and}
	S_\theta \coleq \set{ \Re \big( \smash{e^{i\theta}} / \hbar \big) > 1/\delta}
\fullstop
}
Introduce another complex plane $\Complex_\xi$, sometimes called the \dfn{Borel plane}.
In the same vein, the complex plane $\Complex_\hbar$ is sometimes called the \dfn{Laplace plane}.
Let $e^{i\theta} \Real_+ \subset \Complex_\xi$ be the nonnegative real ray in the direction $\theta$.
By a \dfn{tubular neighbourhood} of $e^{i\theta} \Real_+$ of some \textit{thickness} $\epsilon > 0$ we mean a domain of the form 
\eqntag{
	\Xi_\theta \coleq \set{ \xi \in \Complex_\xi ~\big|~ \op{dist} (\xi, e^{i\theta}\Real_+) < \epsilon}
\fullstop
}

\subsection{The Laplace Transform}

\paragraph{}
Let us first recall some well-known properties of the Laplace transform.
Suppose $\Xi_\theta \subset \Complex_\xi$ is a tubular neighbourhood of $e^{i \theta} \Real_+$, and $\phi = \phi (x, \xi)$ is a holomorphic function on $U \times \Xi_\theta$.
Its \dfn{Laplace transform} in the direction $\theta$ is defined by the formula:
\eqntag{\label{200624181217}
	\Laplace_\theta [\, \phi \,] (x, \hbar)
		\coleq \int\nolimits_{e^{i\theta} \Real_+} \phi (x, \xi) e^{-\xi/\hbar} \dd{\xi}
\fullstop
}
When $\theta = 0$, we often write $\Laplace_+$ instead.
The function $\phi$ is called \dfn{uniformly} or \dfn{locally uniformly Laplace transformable} in the direction $\theta$ if this integral is respectively uniformly or locally uniformly convergent.
Clearly, $\phi$ is uniformly Laplace transformable in the direction $\theta$ if $\phi$ has uniform \dfn{at-most-exponential growth} as $|\xi| \to + \infty$ along the ray $e^{i\theta} \Real_+$.
Explicitly, this means there are constants $\AA, \LL > 0$ such that for all $(x,\xi) \in U \times \Xi_\theta$,
\eqntag{
	\big| \phi (x, \xi) \big| \leq \AA e^{\LL |\xi|}
\fullstop
}

\paragraph{Properties of the Laplace transform.}
Recall that the Laplace transform converts (a) the convolution product of functions into the pointwise multiplication of their Laplace transforms, and (b) differentiation by $\xi$ into multiplication by $\hbar^{-1}$.
Thus, if $\phi, \varphi$ are two uniformly Laplace transformable holomorphic functions on $U \times \Xi_\theta$, then:
\eqn{
	\Laplace_\theta [\, \phi \ast \varphi \,] (x, \hbar)
		= \Laplace_\theta [\, \phi \, ] (x, \hbar) \cdot \Laplace_\theta [\, \varphi \, ] (x, \hbar)
\qtext{and}
	\Laplace_\theta [\, \del_\xi \phi \, ] (x, \hbar)
		= \hbar^{-1} \Laplace_\theta [\, \phi \,] (x, \hbar)
\fullstop
}
Let us also note that the convolution product is taken with respect to the variable $\hbar$ and recall that it is defined by the following formula:
\eqntag{
	\phi \ast \varphi (x, \xi)
	\coleq 
	\int\nolimits_0^\xi \phi (x, \xi - y) \varphi (x, y) \dd{y}
\fullstop{,}
}
where the path of integration is a straight line segment from $0$ to $\xi$.
Finally, if $\phi$ is a uniformly Laplace transformable holomorphic function on $U \times \Xi_\theta$, then for any $x_0 \in U$, and all $(x, \xi) \in U \times \Xi_\theta$,
\eqntag{
	\Laplace_\theta \left[ \, \int\nolimits_{x_0}^x \phi (t, \xi) \dd{t} \, \right]
		= \int\nolimits_{x_0}^x \Laplace_\theta \big[ \, \phi \, \big] (t, \xi) \dd{t}
\fullstop{,}
}
where the path of integration is assumed to lie entirely in $U$.

\subsection{The Borel Transform}

\paragraph{}
Let $f = f(x, \hbar)$ be a holomorphic function on a product domain $U \times S_\theta$, or more generally a holomorphic function on the pair $(U, A_\theta)$ as described in \autoref{210226161539}.
In the latter situation, given $x \in U$, take a sufficiently small Borel disc $S_\theta$ such that $U_0 \times S_\theta$ is contained in $\UUU$ for some neighbourhood $U_0$ of $x$.

The \dfn{analytic Borel transform} (a.k.a., the \dfn{inverse Laplace transform}) of $f$ in the direction $\theta$ is defined by the following formula:
\eqntag{\label{210617101748}
	\Borel_\theta [\, f \,] (x, \xi)
		\coleq \frac{1}{2\pi i} \oint\nolimits_\theta f(x, \hbar) e^{\xi / \hbar} \frac{\dd{\hbar}}{\hbar^2}
\fullstop
}
When $\theta = 0$, we often write $\Borel_+$ instead.
Here, the notation ``$\oint_\theta$'' means that the integration is done along the boundary $\wp_\theta \coleq \set{ \Re (e^{i\theta}/\hbar) = 1 / \delta'}$ of a Borel disc $S'_\theta \subsetneq S_\theta$ of strictly smaller diameter $\delta' < \delta$, traversed anticlockwise (i.e., emanating from the singular point $\hbar = 0$ in the direction $\theta - \pi/2$ and reentering in the direction $\theta + \pi/2$).
Observe that the integral kernel $e^{\xi / \hbar} \hbar^{-2}$ has an essential singularity at $\hbar = 0$ for every nonzero $\xi$.
We therefore interpret this improper integral as the Cauchy principal value, which means it is defined as the limit of an integral over a sequence of path segments on the boundary $\del S'_\theta$ approaching the singular point $\hbar = 0$ in both directions at the same rate.
Explicitly, parameterise the boundary path $\wp_\theta$ by $t \in \Real$ as $\hbar (t) = e^{i \theta} ( \delta' + it )^{-1}$.
Then for every $\TT > 0$, we take the path segment $\wp_\theta (\TT)$ for $t \in [-\TT,+\TT]$ and define
\eqntag{\label{210617080526}
	\oint\nolimits_\theta \coleq \lim_{\TT \to +\infty} \int\nolimits_{\wp_\theta (\TT)}
}
Finally, since the integrand $f(x, \hbar) e^{\xi / \hbar} \hbar^{-2}$ is holomorphic in the interior of the Borel disc $S_\theta$, it is not difficult to see that the integral in \eqref{210617101748} is independent of $\delta'$, provided that it exists.
If the integral \eqref{210617101748} converges for all $\xi \in e^{i\theta} \Real_+$ and uniformly (resp. locally uniformly) for all $x \in U$, then we say $f$ is \dfn{uniformly} (resp. \dfn{locally uniformly}) \dfn{Borel transformable} in the direction $\theta$.
The following lemma gives a criterion for Borel transformability.

\begin{lem}{210617101734}
If $f \in \cal{O} (U; A_\theta)$ admits uniform (resp. locally uniform) Gevrey asymptotics as $\hbar \to 0$ along the closed arc $\bar{A}_\theta$ (in symbols, $f \in \cal{G}_\textup{(loc)} (U; \bar{A}_\theta)$), then it is uniformly (resp. locally uniformly) Borel transformable in the direction $\theta$.
\end{lem}

\begin{proof}
It is enough to consider the situation $\theta = 0$.
Let $f_0, f_1$ be the leading and the next-to-leading order terms in the asymptotic expansion of $f$.
We look at the second remainder term 
\eqn{
	\RR_2 (z) = f (z) - \big( f_0 + z f_1 \big)
\fullstop
}
If the Borel transforms $\Borel_+ [ \, f_0 \, ], \Borel_+ [ \, z f_1 \, ], \Borel_+ [ \, \RR_2 \, ]$ are well-defined, it will follow that $\phi (\xi)$ given by \eqref{200617120419} is a well-defined function on $\Real_+$.
The first two are well-defined and independent of the choice of a Borel path $\wp$, because by definition the integral is interpreted as the Cauchy principal value.
So we just need to examine $\Borel_+ [ \, \RR_2 \, ]$.
The asymptotic condition \eqref{200305191443} with $n = 2$ reads $\big| \RR_2 (z) \big| \leq 2! \CC \MM^2 |z|^2$.
So if $\wp$ is the Borel circle with any radius $\RR$ such that $\RR > \smash{\hat{\RR}}$, then the integral over $\wp$ is well-defined because
\eqn{
	\frac{1}{2\pi} \int_\wp \big| \RR_2 (z) \big| e^{\xi \Re (z^{-1})} \left| \frac{\dd{z}}{z^2} \right|
	\leq \frac{1}{\pi} \CC \MM^2 e^{\RR \xi}
		\int_\wp |\dd{z}|
	\leq \CC \MM^2 e^{\RR\xi} \RR^{-1}
\fullstop
}
Moreover, the integral is independent of the particular choice of the Borel path essentially because the integrand is holomorphic in the sectorial domain and the integral over any Borel circle $\wp$ decays as the radius $\RR$ increases.
\end{proof}

A proof of this proposition (which is a simple complex analytic argument using the Gevrey bounds \eqref{210225134829}) follows from the much stronger result known as \hyperref[200711095033]{Nevanlinna's Theorem \ref*{200711095033}}.
We stress the importance of having Gevrey asymptotics along the \textit{closed} arc $\bar{A}_\theta$ and not just $A_\theta$.
For example, recall from \autoref{210617074256} that the function $f(\hbar) = e^{-1/\hbar}$ is in $\cal{G} (A)$, where $A = (-\pi/2, +\pi/2)$, but not in $\cal{G} (\bar{A})$, and we can see that its Borel transform in the direction $\theta = 0$ is not well-defined when $\xi = 1$.

\paragraph{}
If $f$ is holomorphic at $\hbar = 0$ (i.e., if $f \in \cal{O}_\textup{loc} (U) \set{\hbar}$), then it is necessarily Borel transformable in every direction, and all these Borel transforms agree and define a holomorphic function $\Borel [\, f \,] (x, \xi)$ of $(x, \xi) \in U \times \Complex_\xi$.
In fact, in this case the integration contour in \eqref{210617101748} can be deformed to a circle around the origin, so that the Borel transform of $f$ for any $\theta$ is nothing but the residue integral:
\eqntag{\label{210617101920}
	\Borel [\, f \,] (x, \xi)
		= \underset{\hbar = 0}{\Res}\: \frac{f(x, \hbar) e^{\xi / \hbar}}{\hbar^2}
\fullstop
}

\paragraph{}
Using the residue calculus expression \eqref{210617101920}, it is easy to deduce the following helpful formulas:
\eqntag{\label{200704112520}
	\Borel [\, 1 \,] = 0
\qqqtext{and}
	\Borel [\, \hbar^{k+1} \,] = \frac{\xi^k}{k!}
\qquad \text{for all $k \geq 0$.}
}
They can be used to extend the Borel transform to formal power series in $\hbar$ by defining the \dfn{formal Borel transform}:
for any $\hat{f} (x, \hbar) \in \cal{O} (U) \bbrac{\hbar}$,
\eqntag{
\mbox{}\hspace{-10pt}
	\hat{\phi} (x, \xi) =
	\hat{\Borel} [ \, \hat{f} \, ] (x, \xi)
		\coleq \sum_{n=0}^\infty \phi_n (x) \xi^n
		\in \cal{O} (U) \bbrac{\xi}
\qtext{where}
	\phi_k (x) \coleq \tfrac{1}{k!} f_{k+1} (x)
\GREY{.}
}
Thus, the formal Borel transform essentially `divides' the coefficients of the power series by $n!$.
The following lemma follows immediately from this formula and the Gevrey power series bounds \eqref{190303174313}.

\begin{lem}{210617102534}
If $\hat{f}$ is a uniformly or locally-uniformly Gevrey series on $U$ (in symbols, $\hat{f} \in \cal{G} (U) \bbrac{\hbar}$ or $\hat{f} \in \cal{G}_\textup{loc} (U) \bbrac{\hbar}$, respectively), then its formal Borel transform $\hat{\phi}$ is respectively a uniformly or locally uniformly convergent series in $\xi$.
In symbols, $\hat{\phi} \in \cal{O} (U) \set{\xi}$ or $\hat{\phi} \in \cal{O}_\textup{loc} (U) \set{\xi}$, respectively.
\end{lem}

\paragraph{Properties of the Borel transform.}
Finally, we mention a few important properties of the Borel transform.
Given two uniformly Borel $\theta$-transformable holomorphic functions $f,g$ on $(U, A_\theta)$, the following identities hold for all $(x, \xi) \in U \times \Xi_\theta$:
\begin{gather}
	\Borel_\theta [ \, fg \, ]
		= \big( \Borel_\theta [ \, f \, ] \ast \Borel_\theta [ \, g \, ] \big)
\GREY{;}
\\
	\Borel_\theta [ \, \hbar^{-1} f \, ] = \del_\xi \Borel_\theta [ \, f \, ]
\qtext{and}
	\Borel_\theta [ \, \del_x f \, ] = \del_x \Borel_\theta [ \, f \, ]
\fullstop{;}
\\
	\Borel_\theta \left[ \, \int\nolimits_{x_0}^x f (t, \hbar) \dd{t} \, \right] (x, \xi)
		= \int\nolimits_{x_0}^x \Borel_\theta \big[ \, f \, \big] (t, \xi) \dd{t}
\fullstop{,}
\end{gather}
for any $x_0 \in U$ where the path of integration is assumed to lie entirely in $U$.

\subsection{Borel Resummation}

One of the most fundamental theorems in Gevrey asymptotics is a theorem of Nevanlinna \cite[pp.44-45]{nevanlinna1918theorie} which was rediscovered and clarified decades later by Sokal \cite{MR558468}; see also \cite[p.182]{zbMATH00797135}, and \cite[Theorem 5.3.9]{MR3495546}.
It identifies a subclass of Gevrey formal power series with a class of sectorial germs admitting Gevrey asymptotics in a halfplane.
This identification is often called \textit{Borel resummation} and it should be thought of in complete analogy with the familiar identification of convergent power series with germs of holomorphic functions.
The exceptions are that the identification for Gevrey series depends on a direction $\theta$ and the operation of converting a formal power serious into a holomorphic function is much more involved.
For clarity, we also provide definitions and statements in the single-variable case.

\begin{defn}{210617122738}
A Gevrey power series $\hat{f} (\hbar) \in \Gevrey \bbrac{\hbar}$ is a \dfn{Borel summable series} in the direction $\theta$ if its formal Borel transform $\hat{\phi} (\xi) \in \Complex \set{\xi}$ admits an analytic continuation $\phi (\xi) \coleq \rm{AnCont}_\theta [\, \hat{\phi} \,] (\xi)$ to a tubular neighbourhood $\Xi_\theta$ of the ray $e^{i\theta} \Real_+$ with at-most-exponential growth as $|\xi| \to + \infty$ in $\Xi_\theta$.
The subring of Borel summable series in the direction $\theta$ will be denoted by $\bar{\Gevrey}_\theta \bbrac{\hbar} \subset \Gevrey \bbrac{\hbar}$.

More generally in the parametric situation, a uniformly Gevrey power series $\hat{f} (x, \hbar) \in \cal{G} (U) \bbrac{\hbar}$ is called a \dfn{uniformly Borel summable series} in the direction $\theta$ if its (necessarily uniformly convergent) formal Borel transform $\hat{\phi} (x, \xi) = \hat{\Borel}_\theta [\, \hat{f} \,] \in \cal{O} (U) \set{\xi}$ admits an analytic continuation $	\phi (x, \xi) \coleq \rm{AnCont}_\theta [\, \hat{\phi} \,] (x, \xi)$ to a domain $U \times \Xi_\theta$ for some tubular neighbourhood $\Xi_\theta$ of the ray $e^{i\theta} \Real_+$ with uniformly at-most-exponential growth as $|\xi| \to + \infty$ in $\Xi_\theta$.
\dfn{Locally uniformly Borel summable series} on $U$ are defined the same way by allowing the thickness of the tubular neighbourhood $\Xi_\theta$ to have a locally constant dependence on $x$.
We denote the subalgebras of uniformly and locally uniformly Borel summable series in the direction $\theta$ by $\bar{\cal{G}}_{\theta} (U) \bbrac{\hbar} \subset \cal{G} (U) \bbrac{\hbar}$ and $\bar{\cal{G}}_{\theta, \textup{loc}} (U) \bbrac{\hbar} \subset \cal{G}_\textup{loc} (U) \bbrac{\hbar}$, respectively.
\end{defn}

\begin{thm}[\textbf{Nevanlinna's Theorem}]{210617120300}
In both the one-dimensional and the parametric cases, for any direction $\theta$, the asymptotic expansion map $\op{\ae}$ along the halfplane arc $A_\theta = (\theta - \pi/2, \theta + \pi/2)$ bisected by $\theta$ restricts to an algebra isomorphism
\eqntag{
	\op{\ae}: \cal{G} (\bar{A}_\theta) \iso \bar{\Gevrey}_{\theta} \bbrac{\hbar}
\qtext{or}
	\op{\ae}: \cal{G}_\textup{loc} (U; \bar{A}_\theta) \iso \bar{\cal{G}}_{\theta, \textup{loc}} (U) \bbrac{\hbar}
\fullstop
}
\end{thm}

See \autoref{200624185107} for a detailed proof.
Note that $\op{\ae}$ restricts further to an isomorphism $\cal{G} (U; \bar{A}_\theta) \iso \bar{\cal{G}}_{\theta} (U) \bbrac{\hbar}$.

\begin{defn}{210617132122}
The inverse algebra isomorphism
\eqntag{
	\cal{S}_\theta \coleq \op{\ae}^{-1} : \bar{\Gevrey}_{\theta} \bbrac{\hbar} \iso \cal{G} (\bar{A}_\theta)
\qtext{or}
	\cal{S}_\theta \coleq \op{\ae}^{-1} : \bar{\cal{G}}_{\theta} (U) \bbrac{\hbar} \iso \cal{G}_\textup{loc} (U; \bar{A}_\theta)
}
is called the \dfn{Borel resummation} in the direction $\theta$.
In particular, $\cal{S}_\theta$ restricts to be the identity map on the subalgebra of convergent power series: i.e., if $\hat{f} (\hbar) \in \Complex \set{\hbar}$ or $\hat{f} (x, \hbar) \in \cal{O}_\textup{loc} (U) \set{\hbar}$, then $\cal{S}_\theta (\hat{f}) = \hat{f}$.
\end{defn}

\paragraph{Properties of Borel resummation.}
As an algebra isomorphism, Borel resummation respects sum, products, and scalar multiplication: if $\hat{f}, \hat{g} \in \bar{\Gevrey}_{\theta} \bbrac{\hbar}$ (or $\hat{f}, \hat{g} \in \bar{\cal{G}}_{\theta, \textup{loc}} \bbrac{\hbar}$) are any two (respectively locally uniformly) Borel summable series in the direction $\theta$, and $c \in \Complex$ (or respectively $c = c(x) \in \cal{O} (U)$), then
\eqntag{
	\cal{S}_\theta [\, \hat{f} + c\hat{g} \,]
		= \cal{S}_\theta [\, \hat{f} \,] + c\cal{S}_\theta [\, \hat{g} \,]
\qtext{and}
	\cal{S}_\theta [\, \hat{f} \cdot \hat{g} \,]
		= \cal{S}_\theta [\, \hat{f} \,] \cdot \cal{S}_\theta [\, \hat{g} \,]
\fullstop
}
It is also compatible with composition with convergent power series.
We also note two more properties with respect to calculus in the $x$-variable.
If $\hat{f} (x, \hbar) \in \bar{\cal{G}}_{\theta} (U) \bbrac{\hbar}$ is a locally uniformly Borel summable series in the direction $\theta$, then
\eqntag{
	\del_x \cal{S}_\theta [\, \hat{f} \,]
		= \cal{S}_\theta [\, \del_x \hat{f} \,]
\qtext{and}
	\int\nolimits_{x_0}^x \cal{S}_\theta [\, \hat{f} \,]
		= \cal{S}_\theta \left[\, \int\nolimits_{x_0}^x \hat{f} \,\right]
\fullstop{,}
}
for any $x_0 \in U$ where the path of integration is assumed to lie entirely in $U$.

From the proof of Nevanlinna's Theorem \ref*{210617120300} we can extract the following more explicit statement that yields a formula for the Borel resummation which is actually often taken as the definition of Borel resummation.

\begin{lem}[\textbf{Borel-Laplace identity}]{210617095806}
Let $f \in \cal{G} (\bar{A}_\theta)$ and $\hat{f} \in \bar{\Gevrey}_{\theta} \bbrac{\hbar}$ be such that $\op{\ae} (f) = \hat{f}$ and $\cal{S}_\theta [\, \hat{f} \,] = f$.
Let $\hat{\phi} (\xi) \coleq \hat{\Borel} [\, \hat{f} \,] (\xi) \in \Complex \set{\xi}$ be the (necessarily convergent) formal Borel transform of $\hat{f}$, and let $\phi (\xi) \coleq \rm{AnCont}_\theta [\, \hat{\phi} \,] (\xi)$ be its analytic continuation to a tubular neighbourhood $\Xi_\theta$ of the ray $e^{i\theta} \Real_+$ with at-most-exponential growth as $|\xi| \to + \infty$ in $\Xi_\theta$.
Then we have the following two identities:
\eqntag{
	f (\hbar) = f_0 + \Laplace_\theta [\, \phi \,] (\hbar)
\qtext{and}
	\phi (\xi) = \Borel_\theta [\, f \,] (\xi)
\fullstop
}
More generally in the parametric situation, let $f \in \cal{G} (U; \bar{A}_\theta)$ and $\hat{f} \in \bar{\cal{G}}_{\theta} (U; \bar{A}_\theta)$ be such that $\op{\ae} (f) = \hat{f}$ and $\cal{S}_\theta [\, \hat{f} \,] = f$.
Let $\hat{\phi} (x, \xi) \coleq \hat{\Borel} [\, \hat{f} \,] (\xi) \in \cal{O} (U) \set{\xi}$ be the (necessarily uniformly convergent) formal Borel transform of $\hat{f}$, and let $\phi (x, \xi) \coleq \rm{AnCont}_\theta [\, \hat{\phi} \,] (x,\xi)$ be its analytic continuation to the domain $U \times \Xi_\theta$, where $\Xi_\theta$ in a tubular neighbourhood of the ray $e^{i\theta} \Real_+$, with uniformly at-most-exponential growth as $|\xi| \to + \infty$ in $\Xi_\theta$.
Then we have the following two identities:
\eqntag{\label{210617162329}
	f (x, \hbar) = f_0 (x) + \Laplace_\theta [\, \phi \,] (x, \hbar)
\qtext{and}
	\phi (x, \xi) = \Borel_\theta [\, f \,] (x, \xi)
\fullstop
}
In other words, in both situations we have the following formula for the Borel resummation:
\eqntag{\label{210617162325}
	f
		= \cal{S}_\theta [\, \hat{f} \,]
		= f_0 + \Laplace_\theta \Big[ \, \op{AnCont}_\theta \big[ \, \hat{\Borel}_\theta [\, \hat{f} \,] \, \big] \, \Big]
\fullstop
}
\end{lem}

\subsection{Proof of Nevanlinna's Theorem}
\label{200624185107}

This subsection contains a detailed proof of \hyperref[210617120300]{Nevanlinna's Theorem \ref*{210617120300}}.
Thanks to the assumptions of uniformity in $x$, it is very easy to see how all the limiting and convergence arguments in the proof can be made uniform in $x$.
Therefore, in order to avoid obscuring the discussion with unnecessary verbiage, we omit the dependence on $x$ altogether and present the proof of the following version of Nevanlinna's theorem.
For the benefit of the reader, we have replaced the variable ``$\hbar$'' with ``$z$'' to give the statement a slightly more unadulterated appearance.

\begin{thm}[\textbf{Nevanlinna's Theorem~\cite{nevanlinna1918theorie}}]{200711095033}
Let $\hat{f} (z) \in \Complex \bbrac{z}$ be a formal power series, and let $\hat{\phi} (\xi) \in \Complex \bbrac{\xi}$ be its formal Borel transform:
\eqn{
	\hat{f} (z) = \sum_{k=0}^\infty f_k z^k
\qtext{and}
	\hat{\phi} (\xi)
		\coleq \hat{\Borel} [ \hat{f} ] (\xi)
		= \sum_{k=0}^\infty \phi_k \xi^k
\qtext{where}
	\phi_k \coleq \tfrac{1}{k!} f_{k+1}
\fullstop
}
Then the following two statements are equivalent.
\begin{enumerate}
\item $\hat{\phi} (\xi)$ is a convergent power series which admits an analytic continuation $\phi (\xi)$ to a tubular neighbourhood $\Xi \subset \Complex_\xi$ of the positive real axis $\Real_+ \subset \Complex_\xi$ with at most exponential growth at infinity: i.e., there are constants $\AA, \LL > 0$ such that
\eqntag{\label{200615204240}
	\big| \phi (\xi) \big| \leq \AA e^{\LL |\xi|}
\rlap{\qquad$(\xi \in \Xi)$\fullstop}
}
\item There is a holomorphic function $f = f(z)$ defined on a sectorial domain $S \subset \Complex_z$ of the origin with halfplane opening $(-\tfrac{\pi}{2}, +\tfrac{\pi}{2})$ which admits $\hat{f}$ as its Gevrey asymptotic expansion as $z \to 0$ along the closed right halfplane arc: i.e., there are constants ${\CC, \MM > 0}$ such that for all $n \geq 0$,
\eqntag{\label{200305191443}
	\RR_n (z) \coleq f(z) - \sum_{k=0}^{n-1} f_k z^k
		~\subdomeq~ \CC \MM^n n! z^n
\quad\text{as $z \to 0$ in the right halfplane\fullstop}
}
\end{enumerate}
If these conditions are satisfied, $f$ and $\phi$ are unique and given as the Laplace and Borel transforms of each other:
\eqnstag{
\label{200615203126}
\llap{$(\forall z \in S)$\qqquad}
	f(z) &= f_0 + \Laplace_+ [ \, \phi \, ] (z) 
		= f_0 + \int_0^{+ \infty} e^{- \xi / z} \phi (\xi) \dd{\xi}
\fullstop{,}
\\
\label{200617120419}
\llap{$(\forall \xi \in \Real_+)$\qqquad}
	\phi (\xi) &= \Borel_+ [\, f \,] (\xi) 
		= \frac{1}{2\pi i} \int_{\wp} e^{\xi / z} f(z) \frac{\dd{z}}{z^2}
\fullstop
\rlap{\qqqqqqqqqquad\quad~\qedhere}
}
\end{thm}

\subsubsection[Proof of (1)$\Rightarrow$(2)]{Proof of \hyperref[200711095033]{(1)$\bm{\Rightarrow}$(2)}}

Fix any constant $\RR > \LL$.
Then the exponential estimate \eqref{200615204240} implies that the integral in \eqref{200615203126} is uniformly convergent for all $z$ in the Borel disc
\eqn{
	S \coleq \set{z \in \Complex_z ~\Big|~ \Re (z^{-1}) > \RR}
\fullstop
}
Therefore formula \eqref{200615203126} defines a holomorphic function $f$ on $S$.
It remains to show that this function satisfies the bounds \eqref{200305191443}.

To do so, we cover the sectorial domain $S$ with two smaller sectorial domains $S_\pm$ whose openings $A_\pm$ are strictly less than $\pi$; for example, we can take:
\eqn{
	S_\pm \coleq S \cap \hat{S}_\pm
\qtext{where}
	\hat{S}_\pm 
		\coleq \set{ \Big. |z| < \tfrac{2}{\RR}, \quad \big| \arg(z) - \theta_\pm \big| < \tfrac{\pi}{3} }
\qtext{and}
	\theta_\pm \coleq \pm \tfrac{\pi}{4}
\fullstop
}
The standard sectors $\hat{S}_\pm$ have openings $\hat{A}_\pm \coleq (\theta_\pm - \tfrac{\pi}{3}, \theta_\pm + \tfrac{\pi}{3})$ and the sectorial domain $S_\pm$ have openings $A_+ \coleq (\theta_+ - \tfrac{\pi}{3}, + \tfrac{\pi}{2})$ and $A_- \coleq (- \tfrac{\pi}{2}, \theta_- + \tfrac{\pi}{3})$.
The advantage of restricting to these subsectorial domain $S_\pm$ is that we get the following useful lower bound: if $c \coleq \sin (\tfrac{\pi}{3}) > 0$, then 
\eqntag{\label{200616105512}
	z \in S_\pm
\qtext{$\implies$}
	\Re (\omega_\pm / z) > c / |z|
\qqtext{where}
	\omega_\pm \coleq e^{i \theta_\pm} = e^{\pm i \pi/4}
\fullstop
}
At the same time, the tubular neighbourhood $\Xi$ contains a standard tubular neighbourhood $W_{2\delta} \coleq \set{ \xi ~\big|~ \op{dist} (\xi, \Real_+) < 2\delta}$ for some thickness $2\delta > 0$.
Let $\delta$ be so small that the power series $\hat{\phi} (\xi)$ is absolutely convergent in the disc around $\xi = 0$ of radius $2\delta$.
Mark two points
\eqn{
	\xi_\pm \coleq \delta \omega_\pm \in W_\delta
\fullstop{,}
}
and consider the following paths contained in $W_{2\delta}$:
\eqn{
	\gamma_\pm \coleq [0, \xi_\pm]
\qqtext{and}
	\ell_\pm \coleq \xi_\pm + \Real_+
\fullstop
}
Since $\phi$ is holomorphic on $W_\delta$, we can decompose $f$ on each subdomain $S_\pm$ as 
\eqn{
	f (z) = g_\pm (z) + h_\pm (z)
}
where
\eqn{
	g_\pm (z) \coleq f_0 + \int_{\gamma_\pm} e^{- \xi / z} \phi (\xi) \dd{\xi}
\qqtext{and}
	h_\pm (z) \coleq \int_{\ell_\pm} e^{- \xi / z} \phi (\xi) \dd{\xi}
\fullstop
}

\setcounter{claimx}{0}
\begin{claim}{200615210903}
Each function $g_\pm$ admits $\hat{f}$ as its asymptotic expansion as $z \to 0$ along $\bar{A}_\pm$ which is valid with Gevrey regularity: i.e., there is a constant $\MM_1 > 0$ (independent of the choice of $\pm$) such that, for all $n \geq 0$,
\eqn{
	g_\pm (z) - \sum_{k=0}^{n-1} f_k z^k
		\in \OO \big( \MM_1^n n! z^n \big)
\rlap{\qquad as $z \to 0$ along $A_\pm$\fullstop}
}
\end{claim}

\begin{claim}{200615211357}
Each function $h_\pm (z)$ is asymptotic to $0$ as $z \to 0$ along $\bar{A}_\pm$ with Gevrey regularity: i.e., there is a constant $\MM_2 > 0$ (independent of the choice of $\pm$) such that, for all $n \geq 0$,
\eqn{
	h_\pm (z) \in \OO \big( \MM_2^n n! z^n \big)
\rlap{\qquad as $z \to 0$ along $A_\pm$\fullstop}
}
\end{claim}

The conclusion of this part of the theorem now follows from these two claims by taking $\MM \coleq \max \set{\MM_1, \MM_2}$.
Now we prove these claims.

\begin{subproof}[Proof of \autoref{200615210903}.]
Since each interval $\gamma_\pm$ is contained in the disc of absolute convergence of the power series $\hat{\phi}$, we can make the following computation:
\eqnstag{\label{200616110355}
	g_\pm - \sum_{k=0}^{n-1} f_k z^k
	&= \int_{\gamma_\pm} e^{- \xi / z} \GREEN{\hat{\phi}} \dd{\xi}
		- \sum_{k=1}^{n-1} f_k \BLUE{z^k}
\\	\nonumber
	&= \int_0^{\xi_\pm} \GREEN{\sum_{k = 0}^\infty \phi_k \xi^k} e^{-\xi / z} \dd{\xi}
		- \sum_{k=1}^{n-1} \frac{f_{k}}{\BLUE{(k-1)!}} \BLUE{\int_0^{\xi_\pm \cdot \infty}
			\xi^{k-1} e^{- \xi/z} \dd{\xi}}
\\	\nonumber
	&= \sum_{k = 0}^\infty \phi_k \int_0^{\xi_\pm} \xi^k e^{-\xi / z} \dd{\xi}
		- \sum_{k=0}^{n-2} \phi_{k} \int_0^{\xi_\pm \cdot \infty}
			\xi^k e^{- \xi/z} \dd{\xi}
\\	\nonumber
	&= \sum_{k = 0}^\infty \phi_k \omega_\pm^{k+1} \int_0^{\delta} t^k e^{-t \omega_\pm / z} \dd{t}
		- \sum_{k=0}^{n-2} \phi_{k} \omega_\pm^{k+1} \int_0^{+\infty}
			t^k e^{-t \omega_\pm / z} \dd{t}
\\	\nonumber
	&= \sum_{k = n-1}^\infty \phi_k \omega_\pm^{k+1} \int_0^{\delta} t^k e^{-t \omega_\pm / z} \dd{t}
		- \sum_{k=0}^{n-2} \phi_{k} \omega_\pm^{k+1} \int_{\delta}^{+\infty}
			t^k e^{-t \omega_\pm / z} \dd{t}
\fullstop
}
In the third line we were allowed to interchange integration and summation because the series $\sum \phi_k \xi^k e^{-\xi / z}$ is absolutely convergent for all $\xi$ in the interval $[0, \xi_\pm]$.
(Indeed, from the inequality \eqref{200616105512}, $\Re (\xi / z) = |\xi| \Re (\omega_\pm / z) > 0$ because $z \in S_\pm$, and so $\big| \xi^k e^{- \xi / z} \big| \leq |\xi|^k \leq | \xi_\pm |^k$ for all $\xi \in [0, \xi_\pm]$.)
In the fourth line, we made the substitution $\xi = t \omega_\pm$ in both integrals.

The point of this computation is that the constraints on $t$ and $k$ in both the first and the second summation terms lead to the same bound on $t^k$:
\eqn{
\begin{rcases*}
	t \leq \delta \qtext{and} k \geq n-1
\\	t \geq \delta \qtext{and} k < n-1
\end{rcases*}
~\implies~
	\left(~ \frac{t}{\delta} ~\right)^{k - (n-1)} \leq 1
~\iff~
	t^k \leq t^{n-1} \delta^{k - n + 1}
\fullstop
}
So both integrals in the last line of \eqref{200616110355} can be bounded above by the same quantity:
\eqn{
\begin{rcases*}
	\int_0^{\delta} t^k e^{-t \omega_\pm / z} \dd{t}
\\	\int_{\delta}^{\infty}
			t^k e^{-t \omega_\pm / z} \dd{t}
\end{rcases*}
	\leq
	\delta^{k - n + 1} \int_0^{+\infty} t^{n-1} e^{-t c / |z|} \dd{t}
	\leq \delta^{k - n + 1} n! \frac{|z|^n}{c^n}
\fullstop{,}
}
where we again used the inequality \eqref{200616105512}.
As a result, for all $z \in S$, we obtain the following bound, valid for every $n \geq 1$:
\eqns{
	\left| g_\pm (z) - \sum_{k=0}^{n-1} f_k z^k \right|
	&\leq \sum_{k=0}^\infty 
		|\phi_k| \delta^{k - n + 1}
		n! \frac{|z|^n}{c^n}
	= |z|^n \CC_1 \MM_1^n n!
\fullstop{,}
}
where $\CC_1 \coleq \delta \sum_{k=0}^\infty |\phi_k| \delta^k$ and $\MM_1 \coleq 1/c\delta$.
Note that $\CC_1$ is a finite number because $\delta = |\xi_\pm|$ and $\xi_\pm$ is contained in the disc of uniform convergence of $\hat{\phi}$.
\end{subproof}

\begin{subproof}[Proof of \autoref{200615211357}.]
Parameterise the path $\ell_\pm$ as $\xi(t) = \xi_\pm + t$ for $t \in \Real_+$.
Then using the exponential estimate \eqref{200615204240} it is easy to show that for all $z \in S_\pm$, the function $h_\pm (z)$ is exponentially decaying:
{\small
\eqn{
	\big| h_\pm \big|
		\leq \int_0^{+ \infty} e^{-\Re (\xi_\pm / z)} e^{- t \Re (z^{-1})}
			\Big| \phi \big(\xi_\pm + t \big) \Big| \dd{t}
		\leq \AA e^{\delta \LL} e^{- c\delta /|z|}
			\int_0^{+\infty} e^{-t (\RR - \LL)}
		\leq \CC_2 e^{- \MM_2 /|z|}
}}%
where $\CC_2 \coleq \AA e^{\delta \LL}(\RR - \LL)^{-1}$ and $\MM_2 \coleq c \delta$.
In particular, it follows that for every $n \in \Integer_+$ we have the bound $\big| h_\pm (z) z^{-n} \big| \leq \CC_2 e^{- \MM_2 /|z|} |z|^{-n}$.
For $n = 0$, the claim is obviously true, so let us assume that $n \geq 1$ and analyse the real-valued function $r \mapsto r^{-n} e^{-\MM_2/r}$ on $\Real_{>0}$.
It achieves its maximum value at $r = \MM_2 / n$, so upon using Stirling's bounds, we obtain the desired estimate:
\eqntag{
	\big| h_\pm (z) z^{-n} \big| 
		\leq \CC_2 \KK^{-n} n^n e^{-n}
		\leq \CC_2 \MM_2^n n!
\fullstop
\tag*{\qquad\rlap{$\square\blacksquare$}}
}\noqed
\end{subproof}
\vspace{-\baselineskip}

\subsubsection[Proof of (2)$\Rightarrow$(1)]{Proof of \hyperlink{200711095033}{(2)$\bm{\Rightarrow}$(1)}}

Although the proof of this part of the theorem is long, the strategy is straightforward: we define $\phi$ by the formula \eqref{200617120419} and verify that it is analytic in a tubular neighbourhood of $\Real_+$ with at most exponential growth at infinity in $\xi$.
This verification is a combination of techniques from standard real and complex analysis which crucially rely on the fact that $f$ admits uniform Gevrey asymptotics as $z \to 0$ along the closed right halfplane arc; i.e., on the estimate \eqref{200305191443}.

First, we assume without loss of generality that the sectorial domain $S$ is the Borel disc $\set{\Re (z^{-1}) > \smash{\hat{\RR}}}$ with radius $\smash{\hat{\RR}} > 1$ so large that
\eqntag{\label{200619171900}
\llap{$(\forall x \in S)$\qqqqqqquad}
	\big| \RR_n (z) \big|
		~\leq~ \CC \MM^n n! |z|^n
}
By \autoref{200619145717}, the constants $\CC, \MM$ can be chosen such that the coefficients $f_k$ of the asymptotic expansion $\hat{f}$ (and therefore the coefficients $\phi_k$ of the formal Borel transform $\hat{\phi} = \hat{\Borel} [ \, \hat{f} \, ]$) satisfy the following bounds:
\eqntag{\label{200617121503}
\llap{$(\forall k \in \Integer_{\geq 0})$\qqqquad}
	| f_k | \leq \CC \MM^k k!
\qqtext{and}
	| \phi_k | \leq \CC \MM^k
\fullstop
}
Then the power series $\hat{\phi}$ is absolutely convergent on the disc $\smash{\hat{\Disc}} \coleq \set{ \xi \in \Complex ~\big|~ |\xi| < \hat{r}}$ of radius $\hat{r} \coleq 1/\MM$.
Fix any $0 < r < \hat{r}$.
For every point $\xi_0$ on the real line $\Real_\xi \subset \Complex_\xi$, consider the following interval and disc centred at $\xi_0$:
\eqnstag{\label{200619173137}
	\mathbb{I} (\xi_0) 
		&\coleq \set{ \xi \in \Real ~\big|~ |\xi - \xi_0| < r }
			\subset \Real_\xi
\fullstop{,}
\\
	\Disc (\xi_0) 
		&\coleq \set{ \xi \in \Complex ~\big|~ |\xi - \xi_0| < r }
			\subset \Complex_\xi
\fullstop
}
Obviously, $\mathbb{I} (\xi_0) = \Disc (\xi_0) \cap \Real_\xi$.
We also define $\smash{\hat{\mathbb{I}}} \coleq \smash{\hat{\mathbb{D}}} \cap \Real_\xi$.
Let $\Xi$ be the tubular neighbourhood of the positive real axis $\Real_+$ of thickness $r$:
\eqntag{\label{200619173916}
	\Xi \coleq \set{ \xi \in \Complex_\xi ~\big|~ \op{dist} (\xi, \Real_+) < r }
	~\subset~ \smash{\hat{\Disc}} \cup \Cup_{\xi_0 > \hat{r}} \Disc (\xi_0)
\fullstop
}
Now the proof of this part of the theorem can be broken down into a series of claims whose proofs are presented after this discussion.
The first task is to check that the formula \eqref{200617120419} actually makes sense and defines a holomorphic function near $\xi = 0$ which coincides with the power series $\hat{\phi}$.

\setcounter{claimx}{0}
\begin{claim}{200617120803}
The function $\phi$ is well-defined for $\xi \in \Real_+$ and in particular independent of the choice of a Borel path $\wp$.
\end{claim}

\begin{claim}{200617150331}
The series $\hat{\phi} (\xi)$ converges uniformly to $\phi (\xi)$ for all $\xi \in \smash{\hat{\mathbb{I}}} = \smash{\hat{\mathbb{D}}} \cap \Real_\xi$.
\end{claim}

Therefore, the convergent power series $\hat{\phi}$ is the analytic continuation of $\phi$ from the open interval $\smash{\hat{\mathbb{I}}}$ to the open disc $\hat{\Disc}$.
Next, we verify that $\phi$ is itself the analytic continuation of $\hat{\phi}$ along $\Real_+$.

\begin{claim}{200617133542}
The function $\phi$ is infinitely-differentiable at every point $\xi \in \Real_+$.
Furthermore, the constants $\CC, \MM > 0$ can be taken so large that all derivatives of $\phi$ satisfy the following exponential bound: for all $n \in \Integer_{\geq 0}$, all $\RR > \smash{\hat{\RR}}$, and all $\xi \in \Real_+$,
\eqntag{\label{200617192014}
	\big| \del^n_\xi \phi (\xi) \big|
		\leq \CC \MM^n n! e^{\RR \xi}
\fullstop
}
\end{claim}

\begin{claim}{200617152302}
For every $\xi_0 \in \Real_+$ with $\xi_0 > \hat{r}$, the Taylor series $\TT_{\xi_0} \phi$ of $\phi$ centred at the point $\xi_0$ is absolutely convergent on the interval $\mathbb{I} (\xi_0)$ (and hence on the disc $\Disc (\xi_0)$) where it converges to $\phi$.
\end{claim}

Therefore, the function $\phi$, defined by the formula \eqref{200617120419}, admits a unique analytic continuation to a holomorphic function (which we continue to denote by $\phi$) on the tubular neighbourhood $\Xi$.
All that remains is to show that the Laplace transform of $\phi$ is well-defined and satisfies the equality \eqref{200615203126}.

\begin{claim}{200617153903}
The holomorphic function $\phi$ on $\Xi$ has at most exponential growth at infinity (i.e., $\phi$ satisfies the bound \eqref{200615204240}), so its Laplace transform is well-defined and satisfies the equality \eqref{200615203126} for all $\xi \in \Real_+$.
\end{claim}

Now we prove these five claims.

\begin{subproof}[Proof of \autoref{200617120803}.]
We look at the second remainder term 
\eqn{
	\RR_2 (z) = f (z) - \big( f_0 + z f_1 \big)
\fullstop
}
If the Borel transforms $\Borel [ \, f_0 \, ], \Borel [ \, z f_1 \, ], \Borel [ \, \RR_2 \, ]$ are well-defined, it will follow that $\phi (\xi)$ given by \eqref{200617120419} is a well-defined function on $\Real_+$.
The first two are obviously well-defined and independent of the choice of a Borel path $\wp$ (see \autoref{200617155708}), so we just need to examine $\Borel [ \, \RR_2 \, ]$.
The asymptotic condition \eqref{200305191443} with $n = 2$ reads $\big| \RR_2 (z) \big| \leq 2! \CC \MM^2 |z|^2$.
So if $\wp$ is the Borel circle with any radius $\RR$ such that $\RR > \smash{\hat{\RR}}$, then the integral over $\wp$ is well-defined because
\eqn{
	\frac{1}{2\pi} \int_\wp \big| \RR_2 (z) \big| e^{\xi \Re (z^{-1})} \left| \frac{\dd{z}}{z^2} \right|
	\leq \frac{1}{\pi} \CC \MM^2 e^{\RR \xi}
		\int_\wp |\dd{z}|
	\leq \CC \MM^2 e^{\RR\xi} \RR^{-1}
\fullstop
}
Moreover, the integral is independent of the particular choice of the Borel path essentially because the integrand is holomorphic in the sectorial domain and the integral over any Borel circle $\wp$ decays as the radius $\RR$ increases.
\end{subproof}

\begin{subproof}[Proof of \autoref{200617150331}.]
For any $n \geq 2$, take the identity
\eqn{
	f(z) = \sum_{k=0}^{n-1} f_k z^k + \RR_{n} (z)
\fullstop{,}
}
and plug it into formula \eqref{200617120419} for $\phi$ to get
\eqntag{\label{200617172250}
	\phi (\xi)
		= \sum_{k = 0}^{n-2} \phi_k \xi^k + \EE_n (\xi)
\qtext{where}
	\EE_n (\xi) \coleq \frac{1}{2\pi i} \int_\wp \RR_{n} (z) e^{\xi/z} \frac{\dd{z}}{x^2}
\fullstop
}
Thus, to prove this claim, we have to show that, for all $\xi \in \smash{\hat{\mathbb{I}}}$, the error term $\EE_n (\xi)$ goes to zero as $n \to \infty$.
For any $\RR > \smash{\hat{\RR}}$, let $\wp (\RR)$ be the Borel circle with radius $\RR$.
Note that, because of \autoref{200617120803}, $\EE_n$ is independent of the choice of a Borel path $\wp$, and in particular of $\RR$.
So for all $\xi > 0$, we find:
\eqns{
	\big| \EE_n (\xi) \big|
		&\leq \tfrac{1}{2\pi} \int_{\wp (\RR)} \big| \RR_{n} (z) \big| e^{\RR \xi} \left| \frac{\dd{z}}{z^2} \right|
		\leq \tfrac{ 1 }{2\pi} \CC\MM^{n} n! e^{\RR \xi}
			\int_{\wp (\RR)} \Big| z^{n} \frac{\dd{z}}{z^2} \Big|
		\leq \CC \MM^{n} n! e^{\RR \xi} \RR^{-n}
\fullstop
}
Now, for any fixed $\xi$, look at the real-valued function $\RR \mapsto e^{\RR \xi} \RR^{-n}$ defined for all $\RR > 0$.
It achieves its minimum at $\RR = n / \xi$ with value $e^n n^{-n} \xi^n$.
Fix a sufficiently large integer $\NN > 0$ such that $\NN / \xi_0 > \RR$.
Then it follows that for all $n \geq \NN$, 
\eqn{
	\big| \EE_n (\xi) \big|
		\leq \CC \MM^n n! e^n n^{-n} \xi^n
		\leq \CC \MM^n e n^{1/2} \, \xi^n
\fullstop{,}
}
where we used one of Stirling's bounds.
Since $\xi < \hat{r} = 1/\MM$, it follows that $|\EE_n (\xi)|$ goes to $0$ as $n \to \infty$.
\end{subproof}

\begin{subproof}[Proof of \autoref{200617133542}.]
Let $n \geq 0$ be any integer.
We claim that the $n$-th derivative $\del^n_\xi \phi$ exists at every $\xi \in \Real_+$.
To do this, just like in \eqref{200617172250} we look at the expression
\eqn{
	\phi (\xi)
		= \sum_{k = 0}^{n} \phi_k \xi^k + \EE_{n+2} (\xi)
\qtext{where}
	\EE_{n+2} (\xi) = \frac{1}{2\pi i} \int_\wp \RR_{n+2} (z) e^{\xi/z} \frac{\dd{z}}{z^2}
\fullstop
}
Then we just need to verify that the $n$-th derivative $\del_\xi^n \EE_{n+2}$ exists at every $\xi \in \Real_+$.
In order to be able to swap differentiation and integration, we must first show that the $n$-th derivative of the integrand (which of course exists at every $\xi \in \Real_+$) is bounded by an integrable function independent of $\xi$.
For any $\RR > \smash{\hat{\RR}}$, restrict to the Borel circle $\wp$ with radius $\RR$.
For all $\xi \in \Real_+$, we find:
\eqntag{\label{200617185506}
	\left| \del^n_\xi \left( \RR_{n+2} (z) \frac{e^{\xi/z}}{z^2} \right) \right|
	= \left| \RR_{n+2} (z) \frac{e^{\xi /z}}{z^{n+2}} \right|
	\leq \CC \MM^{n+2} (n+2)! e^{\RR\xi}
\fullstop
}
On any bounded interval in $\Real_+$, the righthand side can be bounded by a constant independent of $\xi$ (although it depends on the interval) which is of course integrable.
Therefore, the derivative $\del^n_\xi \phi$ exists at every $\xi \in \Real_+$ and equals
\eqntag{\label{200617194643}
	\del^n_\xi \phi (\xi)
		= n! \phi_n + \del^n_\xi \EE_{n+2} (\xi)
		= n! \phi_n + \frac{1}{2\pi i} \int_\wp \RR_{n+2} (z) e^{\xi/z} \frac{\dd{z}}{z^{n+2}}
\fullstop
}
It remains to demonstrate the bounds \eqref{200617192014}.
First, note that the integral $\int_\wp \big| z^{n} \frac{\dd{z}}{z^2} \big|$ is convergent since $n \geq 2$.
So we can combine \eqref{200617121503}, \eqref{200617185506}, and \eqref{200617194643} to deduce:
\eqn{
	\big| \del^n_\xi \phi (\xi) \big|
		\leq n! \CC \MM^n + \CC \MM^{n+2} (n+2)! e^{\RR \xi}
		= \CC \MM^n n! \Big( e^{-\RR \xi} + \MM^2 (n+2)(n+1) \Big) e^{\RR \xi}
\fullstop
}
Now, choose any number $c > 1$ and let $\tilde{\MM} \coleq c \MM$.
Then the above expression equals
\eqn{
	\CC \tilde{\MM}^n n!
	\Big( c^{-n} e^{-\RR \xi} + \MM^2 \tfrac{(n+2)(n+1)}{c^n} \Big) e^{\RR \xi}
}
Since $e^{-\RR \xi} < e^{-\xi \smash{\hat{\RR}}}$ and $c > 1$, the expression in the brackets is bounded by a constant $\tilde{c} > 1$ which is independent of $n$ and $\RR$.
Thus, if we let $\tilde{\CC} \coleq \tilde{c} \CC$, we find that $\big| \del^n_\xi \phi (\xi) \big| \leq \tilde{\CC} \tilde{\MM} n!$.
But since $\CC \leq \tilde{\CC}$ and $\MM \leq \tilde{\MM}$, we can redefine $\CC, \MM$ to be the larger constants $\tilde{\CC}, \tilde{\MM}$.
\end{subproof}

\begin{subproof}[Proof of \autoref{200617152302}.]
Fix any $\RR > \smash{\hat{\RR}}$ and any point $\xi_0 > \hat{r}$ on the real line $\Real_+$.
We test the Taylor series 
\eqn{
	\TT_{\xi_0} \phi (\xi)
		= \sum_{n=0}^\infty \frac{1}{n!} \del^n_\xi \phi (\xi_0) \big( \xi - \xi_0 \big)^n
}
for absolute convergence on the interval $\mathbb{I} (\xi_0)$.
Using the bound \eqref{200617192014}, we find:
\eqn{
	\sum_{n=0}^\infty \frac{1}{n!} \big| \del^n_\xi \phi (\xi_0) \big| \cdot \big| \xi - \xi_0 \big|^n
	\leq \CC e^{\RR \xi} \sum_{n=0}^\infty \tfrac{(n+2)!}{n!} (\MM r)^n
	\leq \CC_0 \sum_{n=0}^\infty \tfrac{n^2 + 3n + 2}{2^n}
	< \infty
\fullstop{,}
}
where $\CC_0 \coleq \CC e^{\RR(\xi_0 + r)}$.
To see that the Taylor series $\TT_{\xi_0} \phi$ converges to $\phi$ on the interval $\mathbb{I} (\xi_0)$, we write the remainder in its mean-value form (a.k.a. Lagrange form):
\eqn{
	\GG_{m} (\xi) 
		\coleq \phi (\xi) - \sum_{n=0}^{m} \frac{1}{n!} \del^n_\xi \phi (\xi_0) (\xi - \xi_0)^n
		= \frac{1}{(m+1)!} \del^{m+1}_\xi \phi ( \xi_\ast ) \big( \xi - \xi_0)^{m+1}
}
for some point $\xi_\ast \in \mathbb{I} (\xi_0)$ that lies between $\xi_0$ and $\xi$.
Using \eqref{200617192014} again, and the fact that $r < 1/\MM$, this yields a bound that goes to $0$ as $m \to \infty$:
\eqntag{
	\big| \GG_m (\xi) \big|
	\leq \frac{(m+3)!}{(m+1)!} \CC_0 (\MM r)^m
	= \CC_0 (m+3)(m+2) (\MM r)^m
\fullstop
\tag*{\qedhere}
}
\end{subproof}

\begin{subproof}[Proof of \autoref{200617153903}.]
Fix any $\xi \in \Xi$ and assume without loss of generality that $\Re (\xi) > \hat{r}$.
Then $\xi \in \Xi$ is necessarily contained in a disc $\Disc (\xi_0)$ for some point $\xi_0 > \hat{r}$ on the real line.
On $\Disc (\xi_0)$, the holomorphic function $\phi$ is represented by its Taylor series centred at $\xi_0$.
Then for any $\RR > \smash{\hat{\RR}}$, using \eqref{200617192014}, we get:
\eqn{
	\big| \phi (\xi) \big|
	\leq \sum_{n=0}^\infty \frac{1}{n!} \big| \del^n_\xi \phi (\xi_0) \big| \cdot \big| \xi - \xi_0 \big|^n
	\leq \CC e^{\RR \xi_0} \sum_{n=0}^\infty \tfrac{(n+2)!}{n!} (\MM r)^n
\fullstop
}
The infinite sum in this expression converges to a number $\tilde{\CC}$ independent of $\xi$.
Furthermore, $\xi_0 < |\xi| + r$ because $\xi \in \Disc (\xi_0)$.
So setting $\AA \coleq \CC e^{\RR r} \tilde{\CC}$ yields an exponential bound $\big| \phi (\xi) \big| \leq \AA e^{\RR |\xi|}$ which is valid for all $\xi \in \Xi$.

To demonstrate equality \eqref{200615203126}, consider again the second remainder term $\RR_2 (z) = f(z) - \big( f_0 + z f_1 \big)$.
For any $z \in S$, let $\wp$ be a Borel circle of some radius $\rho$ such that $\Re (z^{-1}) > \rho > \RR$.
In addition, take any $0 < \epsilon < |z|$ and consider the circle in $\Complex_z$ centred at $0$ with radius $\epsilon$, oriented clockwise.
It intersects $\wp$ in exactly two points; let $C (\epsilon)$ be the arc between them that lies in $S$.
Correspondingly, denote by $\wp (\epsilon)$ the part of $\wp$ that lies outside the disc of radius $\epsilon$.
Let the combined oriented contour be denoted by $\Gamma (\epsilon)$.
Then we apply the Cauchy Integral Formula to $z^{-1} \RR_2 (z)$:
\eqn{
	z^{-1} \RR_2 (z) 
		= \frac{1}{2 \pi i} \int_{\Gamma (\epsilon)} \zeta^{-1} \RR_2 (\zeta) \frac{\dd{\zeta}}{z-\zeta}
		= \frac{1}{2 \pi i} \int_{C (\epsilon) + \wp (\epsilon) } \zeta^{-1} \RR_2 (\zeta) \frac{\dd{\zeta}}{x-\zeta}
}
Thanks to the estimate \eqref{200305191443}, we get a bound $\big| \zeta^{-1} \RR_2 (\zeta) \big| \leq 2 \CC \MM^2 |\zeta|$, so the integral along the contour $C (\epsilon)$ goes to $0$ as $\epsilon \to 0$.
It follows that
\eqn{
	z^{-1} \RR_2 (z) = \frac{1}{2 \pi i} \int_{\wp} \zeta^{-1} \RR_2 (\zeta) \frac{\dd{\zeta}}{z-\zeta}
\fullstop
}
Now, we write
~
$\displaystyle
	\frac{1}{z - \zeta} 
		= \frac{1}{z\zeta} \cdot \frac{1}{1/\zeta - 1/z}
		= \frac{1}{z\zeta} \int_{0}^{+\infty} e^{\xi/\zeta - \xi/z} \dd{\xi}
$
~
which gives:
\eqn{
	z^{-1} \RR_2 (z)
		= \frac{z^{-1}}{2 \pi i} \int_{\wp} \int_{0}^{+\infty} \RR_2 (\zeta) e^{\xi/ \zeta -\xi/z} \dd{\xi} \frac{\dd{\zeta}}{\zeta^2}
\fullstop
}
Using the estimate \eqref{200305191443} again, we can apply Fubini's Theorem in order to finally obtain
\eqn{
	\RR_2 (z) 
		= \int_0^{+\infty} e^{-\xi/z} 
		\left( \frac{1}{2 \pi i} \int_\wp \RR_2 (\zeta) e^{\xi/\zeta} \frac{\dd{\zeta}}{\zeta^2} \right) \dd{\xi}
\fullstop
\rlap{\qqqqqqqqquad$\square\blacksquare$}
}
\noqed
\end{subproof}%

\section*{Appendix C: Miscellaneous and Supplementary Proofs}
\phantomsection{}
\addcontentsline{toc}{section}{C: Miscellaneous and Supplementary Proofs}
\setcounter{section}{3}
\setcounter{subsection}{0}

\subsection{Miscellaneous Proofs}
\label{200614173443}

\begin{proof}[Proof of \autoref{200614161900}]
Write $\tilde{a}_0 \coleq a_0 / \sqrt{\DD_0}, \, \tilde{b}_0 \coleq b_0 / \sqrt{\DD_0}, \, \tilde{c}_0 \coleq c_0 / \sqrt{\DD_0}$.
Let $\AA > 0$ be such that $\big| \tilde{a}_0 (x) \big|, \big| \tilde{b}_0 (x) \big|, \big| \tilde{c}_0 (x) \big| \leq \AA$ for all $x \in U$.
If $a_0 \equiv 0$, then this is obvious because in this case $\sqrt{\DD_0} = b_0$ and hence $f_0^+ = - \underline{c}_0$.
Otherwise, if $a_0 \not\equiv 0$, then we choose any constant $\AA_1 > \AA$ and write $U = U_1 \sqcup U_2$ where
\eqn{
	U_1 \coleq \set{ x \in U ~\Big|~ \big| \tilde{a}_0 (x) \big| \geq \AA_1^{-1} }
\qtext{and}
	U_2 \coleq \set{ x \in U ~\Big|~ \big| \tilde{a}_0 (x) \big| < \AA_1^{-1} }
\fullstop
}
In fact, it is convenient to choose $\AA_1 \geq 32 \AA$.
For all $x \in U_1$, we have the bound
\eqn{
	\big| f_0^+ \big|
		= \left| \frac{-b_0 + \sqrt{\DD_0}}{2 a_0} \right|
		= \left| \frac{1 -\tilde{b}_0}{2 \tilde{a}_0} \right|
		\leq \tfrac{1}{2} \AA_1 (1 + \AA)
\fullstop
}
On the other hand, we write $U_2 = U_{2.1} \sqcup U_{2.2}$ where
\eqn{
	U_{2.1} \coleq \set{ x \in U_2 ~\Big|~ \big| \tilde{b}_0 (x) \big| < \tfrac{1}{2} }
\qtext{and}
	U_{2.2} \coleq \set{ x \in U_2 ~\Big|~ \big| \tilde{b}_0 (x) \big| \geq \tfrac{1}{2} }
\fullstop
}
For $x \in U_{2.1}$, we write
\eqn{
	f_0^+ 
		= \frac{-b_0 + \sqrt{\DD_0}}{2 a_0}
		= \frac{\DD_0^2 - b_0^2}{2a_0 (b_0 + \sqrt{\DD_0})}
		= -2 \frac{\tilde{c}_0}{\tilde{b}_0 + 1}
\fullstop
}
Then $f_0^+$ is bounded because $\big| \tilde{b}_0 + 1 \big| \geq \tfrac{1}{2}$.
Otherwise, for $x \in U_{2.2}$, we have a bound $\big| a_0 c_0 b_0^{-2} \big| = \big| \tilde{a}_0 \tilde{c}_0 \tilde{b}_0^{-2} \big| \leq \tfrac{1}{2}$, which yields a bound on $|f_0^+|$ because we can use the Taylor expansion for the square-root function to express $f_0^+$ as a product of bounded functions:
\eqntag{
	f_0^+
		= \frac{1}{2a_0} \left( - b_0 + b_0 \sqrt{ 1 - 4 \frac{a_0 c_0}{b_0^2}} \right)
		= - \frac{\tilde{c}_0}{\tilde{b}_0} \Big( 1 + \OO \big(\tfrac{a_0 c_0}{b_0^2}\big) \Big)
\fullstop
\tag*{\qedhere}
}
\end{proof}

\begin{lem}{200619145717}
Suppose $f(\hbar)$ is a holomorphic functions on a sectorial domain $S \subset \Complex_\hbar$ near the origin which admits Gevrey asymptotics along the closure $\bar{A}$ of the opening arc $A$ of $S$.
Then there are constants $\CC, \MM > 0$ such that both of the following statements hold for all $n \geq 0$:
\eqnstag{\label{200619152640}
	\big|f_n| &\leq \CC \MM^n n!
\fullstop{,}
\\
\label{200619152632}
	\RR_n (\hbar)
	&\subdomeq
	\CC \MM^n n! \hbar^n
\qquad\text{as $\hbar \to 0$ along $A$\fullstop}
\rlap{\qqqqqqqquad~~\qedhere}
}
\end{lem}

\begin{proof}
Condition \eqref{200619152632} is just the definition of $f(\hbar)$ admitting Gevrey asymptotics uniformly as $\hbar \to 0$ along $A$.
Without loss of generality, assume $\MM > 1$.
Consider the equality $f_n \hbar^n = \RR_n (\hbar) - \RR_{n+1} (\hbar)$.
Then, for all $\hbar$ sufficiently small, we have: 
\eqn{
	\big| f_n \big| 
		\leq \CC \MM^n n! + \CC \MM^{n+1} (n+1)! |\hbar| 
		\leq 2 \CC \MM^{n+1} (n+1)!
\fullstop
}
Choose any real number $c > 1$, let $\tilde{\MM} \coleq c \MM$ and write
\eqntag{\label{200619155047}
	2 \CC \MM^{n+1} (n+1)!
		= 2 \CC \MM (n+1) \MM^n n!
		= 2 \CC \MM \tfrac{n+1}{c^{n+1}} \tilde{\MM}^n n!
}
Since $c > 1$, the sequence $\set{\tfrac{n+1}{c^{n+1}}}_n$ is bounded by some constant $\tilde{c} > 1$.
So if we let $\tilde{\CC} \coleq 2 \CC \MM \tilde{c}$, then \eqref{200619155047} is bounded by $\tilde{\CC} \tilde{\MM} n!$.
But since $\CC \leq \tilde{\CC}$ and $\MM \leq \tilde{\MM}$, the asymptotic condition \eqref{200619152632} remains valid with $\CC$ and $\MM$ replaced by $\tilde{\CC}$ and $\tilde{\MM}$.
\end{proof}

\subsection{Supplementary Calculation 1}
\label{200223091621}
In this appendix subsection, we check explicitly that if $\varphi$ satisfies \eqref{200117154820} then $\FF = \Laplace_+ [\, \varphi \,]$ satisfies \eqref{210714172412}, provided that $\varphi$ has at most exponential growth at infinity in $\xi$, uniformly in $z$ (which implies in particular that the Laplace integral $\Laplace_+ [\, \varphi \,]$ is uniformly convergent).
The key is the identity
\eqntag{\label{200624192525}
	\Laplace_+ [\, \varphi \,] = \hbar a_0 + \hbar \Laplace [\, \del_\xi \varphi \,]
\fullstop{,}
}
which follows from integration by parts.
Indeed, $\varphi$ has at most exponential growth at infinity in $\xi$, so for $\hbar$ sufficiently small, the function $\varphi e^{-\xi/\hbar}$ goes to $0$ as $\xi \to +\infty$ uniformly in $z$.
Therefore,
\eqn{
	\Laplace_+ [\, \del_\xi \varphi \,]
		= \int_0^{+\infty} \del_\xi \varphi \, e^{-\xi/\hbar} \dd{\xi}
		= - \varphi (z, 0) - \hbar^{-1} \int_0^{+\infty} \varphi \, e^{-\xi/\hbar} \dd{\xi}
		= - a_0 - \hbar^{-1} \Laplace_+ [\, \varphi \,]
\fullstop
}
Now taking advantage of the uniform convergence of $\Laplace_+ [\varphi]$, we calculate:
\eqns{
	\hbar \del_z \FF - \FF
		&= \hbar \del_z \Laplace_+ [\, \varphi \,] - \Laplace_+ [\, \varphi \,]
\\		&= \hbar \Laplace_+ [\, \del_z \varphi \,] - \hbar a_0 - \hbar \Laplace_+ [\, \del_\xi \varphi \,]
\\		&= \hbar \Laplace_+ [\, \del_z \varphi - \del_x \varphi \,] - \hbar a_0
\\		&= \hbar \Laplace_+ [\, \alpha_0
			+ a_1 \varphi
			+ \alpha_1 \ast \varphi
			+ a_2 \varphi \ast \varphi
			+ \alpha_2 \ast \varphi \ast \varphi \,] - \hbar a_0
\\		&= \hbar \big( \AA_0 + \AA_1 \FF + \AA_2 \FF^2 \big)
\fullstop
}

\subsection{Supplementary Calculation 2}
\label{190907145933}

In this appendix subsection, we check explicitly that the infinite series $\varphi$ defined by \eqref{190306115025} satisfies the integral equation \eqref{190312202516} provided that $\varphi$ is uniformly convergent.
First, note the formula for the convolution product of infinite sums:
\eqn{
	\left( \sum_{n=0}^\infty \varphi_n \right) \ast \left( \sum_{n=0}^\infty \varphi_n \right)
	= \sum_{n=0}^\infty \sum_{\substack{ i,j \geq 0 \\ i + j = n}} \varphi_i \ast \varphi_j
}
We substitute \eqref{190306115025} into the righthand side of \eqref{190312202516}, use linearity of the integral operator $\II_+$, and split and rearrange the sums as follows:
\eqns{
	\varphi_0 
	&+ \II_+ \Big[ \alpha_0
					+ a_1 \varphi
					+ \alpha_1 \ast \varphi
					+ a_2 \varphi \ast \varphi
					+ \alpha_2 \ast \varphi \ast \varphi \Big]
\\
=
	\varphi_0
	&+
	\II_+ \Bigg[
			\alpha_0 + a_1 \Bigg( \varphi_0 + \varphi_1 + \varphi_2 + \varphi_3 + \sum_{n=4}^\infty \varphi_n \Bigg)
			+ \alpha_1 \ast \Bigg( \varphi_0 + \varphi_1 + \varphi_2 + \sum_{n=3}^\infty \varphi_n \Bigg)
\\
	&+ a_2 \Bigg( \varphi_0 \ast \varphi_0 + \cdots
					+ \sum_{n=3}^\infty \sum_{\substack{ i,j \geq 0 \\ i + j = n}} \varphi_i \ast \varphi_j \Bigg)
			+ \alpha_2 \ast \Bigg( \varphi_0 \ast \varphi_0 + \cdots + \sum_{n=2}^\infty \sum_{\substack{ i,j \geq 0 \\ i + j = n}} \varphi_i \ast \varphi_j \Bigg)
	\Bigg]
\\
=
	\varphi_0 
	&+
	\II_+ [ \, \alpha_0 + a_1 \varphi_0 \, ]
		+ \II_+ \big[ \, a_1 \varphi_1 + \alpha_1 \ast \varphi_0 + a_2 \varphi_0 \ast \varphi_0 \, \big]
\\
	&+ \II_+ \big[ a_1 \varphi_2 + \alpha_1 \ast \varphi_1 + a_2 \varphi_1 \ast \varphi_0 + a_2 \varphi_0 \ast \varphi_1 + \alpha_2 \ast \varphi_0 \ast \varphi_0 \big]
\\
	&+ \II_+ \big[ a_1 \varphi_3 + \alpha_1 \ast \varphi_2 + a_2 \varphi_2 \ast \varphi_0 + a_2 \varphi_1 \ast \varphi_1 + a_2 \varphi_0 \ast \varphi_2 + \alpha_2 \ast \varphi_1 \ast \varphi_0 + \alpha_2 \ast \varphi_0 \ast \varphi_1 \big]
\\
	&+ \II_+ \Bigg[ a_1 \sum_{n=4}^\infty \varphi_n 
			+ \alpha_1 \ast \sum_{n=3}^\infty \varphi_n
			+ a_2 \sum_{n=3}^\infty \sum_{\substack{ i,j \geq 0 \\ i + j = n}} \varphi_i \ast \varphi_j
			+ \alpha_2 \ast \sum_{n=2}^\infty \sum_{\substack{ i,j \geq 0 \\ i + j = n}} \varphi_i \ast \varphi_j
		\Bigg]
\\
=
	\varphi_0 
	&+ \ldots + \varphi_4
	+ \sum_{n=5}^\infty \II_+ 
		\Bigg[ a_1 \varphi_{n-1} 
			+ \alpha_1 \ast \varphi_{n-2} 
			+ a_2 \sum_{\substack{i,j \geq 0 \\ i + j = n-2}}
				\varphi_i \ast \varphi_j
			+ \alpha_2 \ast \sum_{\substack{i,j \geq 0 \\ i + j = n-3}}
				\varphi_i \ast \varphi_j
		\Bigg]
\\
= \varphi_+
&
\fullstop
}

\subsection{Some Useful Elementary Estimates}

Here we collect some elementary estimates that are used in the proof of the \hyperlink{210714172405}{Main Technical Lemma}.
Their proofs are straightforward, but for completeness we supply them here anyway.

\begin{lem}{180824194855}
For any $\RR \geq 0$, any $\LL \geq 0$, and any nonnegative integer $n$,
\eqn{
	\int_0^{\RR} \frac{r^n}{n!} e^{\LL r} \dd{r}
		\leq \frac{\RR^{n+1}}{(n+1)!} e^{\LL \RR}
\fullstop
\tag*{\qedhere}
}	
\end{lem}

\begin{proof}
Define a function
\eqn{
	f (\RR) \coleq \frac{\RR^{n+1}}{(n+1)!} e^{\LL \RR} - \int_0^{\RR} \frac{r^n}{n!} e^{\LL r}
\fullstop{,}
}
which is the difference between the righthand side and the lefthand side of the inequality we want to prove.
We want to show that $f(\RR) \geq 0$ for all $\RR \geq 0$.
Notice that for all $\RR > 0$,
\eqn{
	\frac{\dd{f}}{\dd{\RR}} = \frac{\LL \RR^{n+1}}{(n+1)!} e^{\LL\RR} > 0
\fullstop{,}
}
so $f$ is an increasing function.
The lemma follows from the fact that $f(0) = 0$.
\end{proof}

\begin{lem}{180824192144}
For any $\RR \geq 0$, and any integers $m, n \geq 0$,
\eqn{
	\int_0^\RR (\RR - r)^m r^n \dd{r}
		= \frac{m! n!}{(m+n+1)!} \RR^{m+n+1}
\fullstop
\tag*{\qedhere}
}
\end{lem}

This formula is an instance of the relationship between the gamma and the beta functions, but it can be justified in the following more elementary way.

\begin{proof}
Straightforward integration gives the formula
\eqn{
	\int_0^\RR (\RR - r)^m r^n \dd{r}
		= \CC_{m,n} \RR^{m+n+1}
\qqtext{where}
	\CC_{m,n} \coleq \sum_{k=0}^m {m \choose k} \frac{(-1)^k}{k+n+1}
\fullstop
}
To obtain the formula $\CC_{m,n} = m! n! \big/ (m+n+1)!$, first prove the identity 
\eqn{
	\CC_{m+1, n} = \CC_{m,n} - \CC_{m,n+1}
\fullstop{,}
}
which readily follows from Pascal's rule
\eqn{
	{m + 1 \choose k + 1} = {m \choose k} + {m \choose k+1}
\fullstop
}
Then use induction on $m$.
\end{proof}

\begin{lem}{180824195940}
Let $i,j$ be nonnegative integers, and let $f_i(\xi), f_j(\xi)$ be holomorphic functions on a tubular neighbourhood $\Xi_+ \coleq \set{\xi ~\big|~ \op{dist} (\xi, \Real_+) < \epsilon}$ of the positive real axis $\Real_+ \subset \Complex_\xi$ for some $\epsilon > 0$.
If there are constants $\MM_i, \MM_j, \LL \geq 0$ such that
\eqn{
	\big| f_i(\xi) \big| \leq \MM_i \frac{|\xi|^i}{i!} e^{\LL |\xi|}
\qqtext{and}
	\big| f_j(\xi) \big| \leq \MM_j \frac{|\xi|^j}{j!} e^{\LL |\xi|}
\rlap{\qquad $\forall \xi \in \Xi_+$\fullstop{,}}
}
then their convolution product satisfies the following bound:
\eqn{
	\big| f_i \ast f_j (\xi) \big| \leq \MM_i \MM_j \frac{|\xi|^{i+j+1}}{(i+j+1)!} e^{\LL |\xi|}
\rlap{\qqquad $\forall \xi \in \Xi_+$\fullstop}
\tag*{\qedhere}
}
\end{lem}

\begin{proof}
By definition,
\eqn{
	f_i \ast f_j (\xi) = \int_0^\xi f_i (\xi - u) f_j (u) \dd{u}
\fullstop{,}
}
where the integration contour is chosen to be the straight line segment $[0, \xi]$.
Write $\xi = |\xi| e^{i \theta}$ and parameterise the integration contour by $u (r) \coleq r e^{i \theta}$.
Then
\eqns{
	\big| f_i \ast f_j (\xi) \big|
	&\leq \int_0^{\xi} \big| f_i (\xi - u) \big| \cdot \big| f_j (u) \big| |\dd{u} |
\\	&\leq \frac{\MM_i \MM_j}{i! j!} \int_0^{\xi} \big| \xi - u \big|^i e^{\LL |\xi - u|} \cdot |u|^j e^{\LL |u|} |\dd{u}|
\\	&= \frac{\MM_i \MM_j}{i! j!} \int_0^{|\xi|} \big( |\xi| - r \big)^i e^{\LL (|\xi| - r)} \cdot r^j e^{\LL r} \dd{r}
\\	&= \frac{\MM_i \MM_j}{i! j!} e^{\LL |\xi|} \int_0^{|\xi|} \big( |\xi| - r \big)^i r^j \dd{r}
\fullstop
}
The result follows from \autoref{180824192144}.
\end{proof}

\end{appendices}

\begin{adjustwidth}{-2cm}{-1.5cm}
{\footnotesize
\bibliographystyle{nikolaev}
\bibliography{/Users/Nikita/Documents/Library/References}

\begin{thebibliography}{GMN13b}

\bibitem[AKT91]{MR1166808}
T.~Aoki, T.~Kawai, and Y.~Takei,  {\em The {B}ender-{W}u analysis and the
  {V}oros theory},  in {\em Special functions ({O}kayama, 1990)}, ICM-90
  Satell. Conf. Proc., pp.~1--29.
\newblock Springer, Tokyo, 1991.

\bibitem[BS15]{MR3349833}
T.~Bridgeland and I.~Smith,  {\em Quadratic differentials as stability
  conditions},  \href{http://dx.doi.org/10.1007/s10240-014-0066-5}{{\em Publ.
  Math. Inst. Hautes {\'E}tudes Sci.} {\bfseries 121} (2015) 155--278},
  \href{http://arxiv.org/abs/1302.7030}{{\ttfamily arXiv:1302.7030 [math.AG]}}.

\bibitem[DDP93]{MR1209700}
{\'E}.~Delabaere, H.~Dillinger, and F.~Pham,  {\em R{\'e}surgence de {V}oros et
  p{\'e}riodes des courbes hyperelliptiques},
  \href{http://dx.doi.org/10.5802/aif.1326}{{\em Ann. Inst. Fourier (Grenoble)}
  {\bfseries 43} no.~1, (1993) 163--199}.

\bibitem[DLS93]{MR1232828}
T.~M. Dunster, D.~A. Lutz, and R.~Sch\"{a}fke,  {\em Convergent
  {L}iouville-{G}reen expansions for second-order linear differential
  equations, with an application to {B}essel functions},
  \href{http://dx.doi.org/10.1098/rspa.1993.0003}{{\em Proc. Roy. Soc. London
  Ser. A} {\bfseries 440} no.~1908, (1993) 37--54}.

\bibitem[GMN13a]{MR3115984}
D.~Gaiotto, G.~W. Moore, and A.~Neitzke,  {\em Spectral networks},
  \href{http://dx.doi.org/10.1007/s00023-013-0239-7}{{\em Ann. Henri
  Poincar{\'e}} {\bfseries 14} no.~7, (2013) 1643--1731},
  \href{http://arxiv.org/abs/1204.4824}{{\ttfamily arXiv:1204.4824 [hep-th]}}.

\bibitem[GMN13b]{MR3003931}
D.~Gaiotto, G.~W. Moore, and A.~Neitzke,  {\em Wall-crossing, {H}itchin
  systems, and the {WKB} approximation},
  \href{http://dx.doi.org/10.1016/j.aim.2012.09.027}{{\em Adv. Math.}
  {\bfseries 234} (2013) 239--403},
  \href{http://arxiv.org/abs/0907.3987}{{\ttfamily arXiv:0907.3987 [hep-th]}}.

\bibitem[IN14]{MR3280000}
K.~Iwaki and T.~Nakanishi,  {\em Exact {WKB} analysis and cluster algebras},
  \href{http://dx.doi.org/10.1088/1751-8113/47/47/474009}{{\em J. Phys. A}
  {\bfseries 47} no.~47, (2014) 474009, 98}.

\bibitem[KT05]{MR2182990}
T.~Kawai and Y.~Takei, {\em Algebraic analysis of singular perturbation
  theory}, vol.~227 of {\em Translations of Mathematical Monographs}.
\newblock American Mathematical Society, Providence, RI, 2005.
\newblock Translated from the 1998 Japanese original by Goro Kato, Iwanami
  Series in Modern Mathematics.

\bibitem[LR16]{MR3495546}
M.~Loday-Richaud, \href{http://dx.doi.org/10.1007/978-3-319-29075-1}{{\em
  Divergent series, summability and resurgence. {II}}}, vol.~2154 of {\em
  Lecture Notes in Mathematics}.
\newblock Springer, [Cham], 2016.
\newblock Simple and multiple summability, With prefaces by Jean-Pierre Ramis,
  \'{E}ric Delabaere, Claude Mitschi and David Sauzin.

\bibitem[{Mal}95]{zbMATH00797135}
B.~{Malgrange},  {\em {Sommation des s\'eries divergentes.}},  {\em {Expo.
  Math.}} {\bfseries 13} no.~2-3, (1995) 163--222.

\bibitem[Nem21]{MR4226390}
G.~Nemes,  {\em On the Borel summability of WKB solutions of certain
  Schr{\"o}dinger-type differential equations},
  \href{http://dx.doi.org/10.1016/j.jat.2021.105562}{{\em J. Approx. Theory}
  {\bfseries 265} (2021) 105562, 30},
  \href{http://arxiv.org/abs/2004.13367}{{\ttfamily arXiv:2004.13367
  [math.CA]}}.

\bibitem[Nev18]{nevanlinna1918theorie}
F.~Nevanlinna, {\em Zur Theorie der asymptotischen Potenzreihen}.
\newblock No.~v. 1 in Annales Academiae Scientiarum Fennicae. Series A.
  Alexander University of Finland, 1918.

\bibitem[Nik19a]{1902.03384}
N.~Nikolaev,  {\em {Abelianisation of Logarithmic
  $\mathfrak{sl}_2$-Connections}},
  \href{http://arxiv.org/abs/1902.03384}{{\ttfamily arXiv:1902.03384
  [math.AG]}}.

\bibitem[Nik19b]{nikolaev2019triangularisation}
N.~Nikolaev,  {\em Triangularisation of Singularly Perturbed Linear Systems of
  Ordinary Differential Equations: Rank Two at Regular Singular Points},
  \href{http://arxiv.org/abs/1909.04011}{{\ttfamily arXiv:1909.04011
  [math.CA]}}.

\bibitem[Nik21]{MY210623112236}
N.~Nikolaev,  {\em Existence and Uniqueness of Exact WKB Solutions for
  Second-Order Linear ODEs},  \href{http://arxiv.org/abs/2106.10248}{{\ttfamily
  arXiv:2106.10248 [math.AP]}}.

\bibitem[Olv97]{MR1429619}
F.~W.~J. Olver, {\em Asymptotics and special functions}.
\newblock AKP Classics. A K Peters Ltd., Wellesley, MA, 1997.

\bibitem[Rei72]{MR0357936}
W.~T. Reid, {\em Riccati differential equations}.
\newblock Academic Press, New York-London, 1972.
\newblock Mathematics in Science and Engineering, Vol. 86.

\bibitem[Sil85]{MR819680}
H.~J. Silverstone,  {\em J{WKB} connection-formula problem revisited via
  {B}orel summation},
  \href{http://dx.doi.org/10.1103/PhysRevLett.55.2523}{{\em Phys. Rev. Lett.}
  {\bfseries 55} no.~23, (1985) 2523--2526}.

\bibitem[Sok80]{MR558468}
A.~D. Sokal,  {\em An improvement of {W}atson's theorem on {B}orel
  summability},  \href{http://dx.doi.org/10.1063/1.524408}{{\em J. Math. Phys.}
  {\bfseries 21} no.~2, (1980) 261--263}.

\bibitem[Str84]{MR743423}
K.~Strebel, \href{http://dx.doi.org/10.1007/978-3-662-02414-0}{{\em Quadratic
  differentials}}, vol.~5 of {\em Ergebnisse der Mathematik und ihrer
  Grenzgebiete (3) [Results in Mathematics and Related Areas (3)]}.
\newblock Springer-Verlag, Berlin, 1984.

\bibitem[Tak17]{takei2017wkb}
Y.~Takei,  {\em WKB analysis and stokes geometry of differential equations},
  in {\em Analytic, Algebraic and Geometric Aspects of Differential Equations},
  pp.~263--304.
\newblock Springer, 2017.

\bibitem[Vor83]{MR729194}
A.~Voros,  {\em The return of the quartic oscillator. {The complex WKB
  method}},  {\em \href{http://eudml.org/doc/76217}{Annales de l'I.H.P.
  Physique th{\'e}orique}} {\bfseries 39} no.~3, (1983) 211--338.

\bibitem[Was76]{MR0460820}
W.~Wasow, {\em Asymptotic expansions for ordinary differential equations}.
\newblock Robert E. Krieger Publishing Co., Huntington, N.Y., 1976.
\newblock Reprint of the 1965 edition.

\end{thebibliography}
}
\end{adjustwidth}
\end{document}